\documentclass[11pt]{elsarticle}


\makeatletter
\def\ps@pprintTitle{%
  \ps@plain 
  \let\@oddhead\@empty
  \let\@evenhead\@empty
}
\makeatother

\makeatletter
\def\elscormarksym#1{%
  \ifcase#1\or$\dagger$\or$\ddagger$\fi
}

\def\corref#1{%
  \edef\cnotenum{\elsRef{#1}}%
  \edef\@corref{\elscormarksym{\cnotenum}\hskip-1pt}%
}

\def\cormark[#1]{%
  \edef\cnotenum{\elsRef{#1}}%
  \unskip\textsuperscript{\sep\elscormarksym{\cnotenum}\hspace{-1pt}}%
  \let\sep=,%
}

\def\cortext[#1]#2{%
  \g@addto@macro\@cornotes{%
    \refstepcounter{cnote}\elsLabel{#1}%
    \def\thefootnote{\ifcase\thecnote\or$\dagger$\or$\ddagger$\fi}%
    \footnotetext{#2}%
  }%
}
\makeatother

\usepackage{lineno,hyperref}
\modulolinenumbers[5]

\usepackage{amsmath,amssymb,amsthm,amsfonts,afterpage,float,bm,stmaryrd,paralist,wrapfig,bbm,soul,extarrows}

\usepackage{cancel}
\usepackage{color}
\usepackage{fancyhdr}
\usepackage{graphics}
\usepackage{hyperref}
\usepackage{graphicx,epsf}
\usepackage{makeidx}
\usepackage{epsfig}
\usepackage{lscape}
\usepackage{upgreek}
\usepackage{mathbbol} 
\usepackage{setspace}
\usepackage{etoolbox}

\usepackage[margin=0.25in]{geometry}
\usepackage{pgfplots}
\usepackage{tikz}
\usetikzlibrary{arrows.meta}

\usepackage{empheq,xcolor}

\pssilent

\newtheorem{theorem}{Theorem}[section]  
\newtheorem{lemma}[theorem]{Lemma}
\newtheorem{example}[theorem]{Example}
\newtheorem{definition}[theorem]{Definition}
\newtheorem{corollary}[theorem]{Corollary}

\newtheorem{remark}[theorem]{Remark}

\AtBeginEnvironment{theorem}{\setstretch{1.3}}
\AtBeginEnvironment{lemma}{\setstretch{1.3}}
\AtBeginEnvironment{proof}{\setstretch{1.3}}
\AtBeginEnvironment{example}{\setstretch{1.3}}
\AtBeginEnvironment{definition}{\setstretch{1.3}}
\AtBeginEnvironment{corollary}{\setstretch{1.3}}
\AtBeginEnvironment{proposition}{\setstretch{1.3}}
\AtBeginEnvironment{remark}{\setstretch{1.3}}

\numberwithin{equation}{section}
\numberwithin{figure}{section}
\numberwithin{table}{section}

\def\XXint#1#2#3{{\setbox0=\hbox{$#1{#2#3}{\int}$}
\vcenter{\hbox{$#2#3$}}\kern-.51\wd0}}

\newcommand{\reff}[1]{{\rm (\ref{#1})}}

\newcommand{\bbR}{{\mathbb R}}

\newcommand{\bx}{\mathbf x}
\newcommand{\by}{\mathbf y}
\newcommand{\bz}{\mathbf z}
\newcommand{\bfm}{\mathbf m}
\newcommand{\bn}{\mathbf n}

\newcommand{\bv}{\mathbf v}
\newcommand{\bu}{\mathbf u}
\newcommand{\bw}{\mathbf w}
\newcommand{\bA}{\mathbf A}
\newcommand{\bS}{\mathbf S}
\newcommand{\bT}{\mathbf T}
\newcommand{\rmT}{\mathrm T}

\renewcommand{\div}{\operatorname{div}}

\newcommand{\KL}{\operatorname{KL}}

\newcommand{\TK}{\mathcal{T}_\mathrm{K}}

\graphicspath{{Figure/}}

\begin{document}

\setlength{\pdfpageheight}{\paperheight}
\setlength{\pdfpagewidth}{\paperwidth}
\title{Synchronization of Unbalanced Dynamical Optimal Transport across Multiple Spaces\tnoteref{alpha}}

\author[1,2]{Zixuan Cang\corref{cor1}}
\author[1]{Jingfeng Wang}
\author[1]{Xiaoqi Wei}
\author[3]{Yanxiang Zhao\corref{cor2}}

\address[1]{Department of Mathematics, North Carolina State University, Raleigh, NC, USA}
\address[2]{Center for Research in Scientific Computation, North Carolina State University, Raleigh, NC, USA}
\address[3]{Department of Mathematics, George Washington University, Washington, DC, USA}

\cortext[cor1]{Corresponding author. Email: zcang@ncsu.edu}
\cortext[cor2]{Corresponding author. Email: yxzhao@gwu.edu}

\tnotetext[alpha]{Authors are listed in alphabetical order by last name.}

\begin{abstract}

  Many biological systems are observed through heterogeneous modalities, requiring transport models that couple dynamics across spaces while allowing mass variation. To address this challenge, we introduce Unbalanced Synchronized Optimal Transport (UnSyncOT), a novel dynamical framework that synchronizes transport-reaction flows between spaces via either geometric embeddings (Monge type) or Markov kernels (Kantorovich type). For both cases we prove that UnSyncOT can be reduced to a single-space problem: the Monge model becomes a Benamou-Brenier problem with a metric-modified kinetic energy, and the Kantorovich model yields a nonlocal action induced by the synchronization operator, both of which fit within a dissipation-distance formulation. We also analyze the pure transport (Wasserstein) and pure reaction (Fisher-Rao) limits and derive structural properties. For the Kantorovich case we propose an approximate UnSyncOT by introducing a Hellinger-Kantorovich based trapezoidal time discretization of the secondary action for efficient computation. Finally we present staggered-grid discretizations and primal-dual solvers, validate the convergence, stability, and efficiency, and demonstrate coherent dynamics reconstructions across spaces.
  
\end{abstract}


\date{\today}
\maketitle

\section{Introduction}\label{section:Intro}

Optimal transport (OT) finds an optimal coupling between probability distributions that minimizes a prescribed cost and induces a geometry on the space of measures \cite{peyre2019computational,villani2009optimal}. OT has extensive applications in imaging \cite{chizat2018scaling}, machine learning \cite{tolstikhin2017wasserstein,fatras2019learning}, and biology \cite{schiebinger2019optimal,cang2020inferring,cang2023screening,sha2024reconstructing,tong2020trajectorynet}. Significant theoretical and computational advances have produced scalable algorithms \cite{cuturi2013sinkhorn,ferradans2014regularized,benamou2015iterative}. A rich family of variants addresses common modeling needs such as partial and unbalanced or unnormalized OT that relax marginal constraints \cite{chizat2018scaling,figalli2010optimal,igbida2018augmented,chizat2018unbalanced,gangbo2019unnormalized,lee2021generalized}, structured OT for graphs \cite{alvarez2018structured}, and multi-marginal OT \cite{pass2015multi,strossner2023low}.

In many applications, the dynamics of a changing system can be modeled by the dynamical formulation of OT. For example, cellular differentiation dynamics has been modeled using the Benamou-Brenier (BB) formulation \cite{tong2020trajectorynet,benamou2000computational}. Taking the squared Euclidean distance as the cost, BB formulation is equivalent to the static OT \cite{benamou2000computational}. Recently, primal-dual based numerical methods have been developed to solve dynamical OT \cite{papadakis2014optimal}, and Wasserstein gradient flow \cite{carrillo2022primal}. Meanwhile, deep learning based approaches such as TrajectoryNet, MIOFlow, and Conditional Flow Matching have also been designed to solve dynamical OT \cite{tong2020trajectorynet,huguet2022manifold,tong2024improving}. When the mass variation is allowed,  dynamical OT can also be extended to various unbalanced variants such as Wasserstein-Fisher-Rao (WFR) \cite{chizat2018interpolating} form or Hellinger-Kantorovich (HK) form \cite{liero2016optimal,liero2018optimal}, and unnormalized forms \cite{gangbo2019unnormalized,lee2021generalized}. Numerical methods, including traditional primal-dual based schemes \cite{gangbo2019unnormalized,lee2021generalized, chizat2018interpolating}, and deep learning based methods \cite{sha2024reconstructing,zhang2025modeling}, have been developed to solve unbalanced dynamical OT problems. 

Modern data acquisition increasingly yields multimodal measurements of the same system, such as joint single-cell modalities (RNA, ATAC) and spatial transcriptomics \cite{heumos2023best}. Each modality lives in a distinct feature space and only partially observes the underlying state, and biological processes naturally involve mass variation (proliferation, death, activation), calling for \emph{unbalanced} OT models (WFR or equivalently HK). Existing single-space dynamical OT methods reconstruct convincing flows per modality but do not guarantee coherence across spaces, namely, a trajectory optimal in one space may induce implausible dynamics in another. This motivates a principled framework that jointly models transport-reaction dynamics across coupled spaces while allowing mass change.

In our early work \cite{cang2025synchronized}, as the first attempt, we introduce the synchronized (balanced) optimal transport, which assumes equal mass along the temporal trajectory of the dynamics. In this work, we consider a more general case and propose \emph{Unbalanced Synchronized Optimal Transport} (UnSyncOT), a dynamical framework that synchronizes transport-reaction flows across multiple spaces via prescribed correspondences while permitting mass variation. We treat two correspondence models: (I) Monge UnSyncOT, where a geometric embedding pushes forward dynamics to secondary spaces; and (II) Kantorovich UnSyncOT, where synchronization is mediated by a Markov kernel. In both cases we show that UnSyncOT reduces to a single-space problem: the Monge model becomes a Benamou-Brenier problem with a metric-modified kinetic energy, and the Kantorovich model yields a nonlocal action induced by the synchronization operator, both of which admit a natural dissipation-distance (Onsager) formulation. For the pure transport (Wasserstein) and pure reaction (Fisher-Rao) limits, we also analyze some structural properties such as constant-speed geodesics and bound estimates. For computation, we introduce an HK-based trapezoidal quadrature for the secondary action in the Kantorovich case and discretize both formulations on staggered grids, leading to finite-dimensional convex problems solved by primal-dual algorithms. Numerical experiments validate convergence, stability, and efficiency, and demonstrate coherent trajectory reconstructions across spaces.

\section{Notations and Background of Dynamical Optimal Transport} \label{section:Overview}

\subsection{Useful Notations}\label{subsection:notation}

In this paper, the spatial domain of interest is a convex and open subset $X\subseteq\mathbb{R}^m$, on which the natural Euclidean norm is denoted by $\|\cdot\|$. The set of Borel measures (respectively, nonnegative Borel measures) on $X$ is denoted by $\mathcal{M}(X)$(respectively, $\mathcal{M}_+(X)$), or sometimes simply by $\mathcal{M}$ (respectively, $\mathcal{M}_+$). The set of probability measures on $X$ is denoted by $\mathcal{P}(X)$, and the set of probability measures on $X$ with finite $p$-moment is denoted by $\mathcal{P}_p(X)$. We use $d\bx$ to denote the standard Lebesgue measure. Throughout the paper, we use a mild abuse of notation and identify all probability measures with their densities: $\rho(d\bx) = d\rho(\bx) = \rho(\bx)d\bx$, and $d\pi(\bx,\by)=\pi(d\bx, d\by) = \pi(\bx,\by)d\bx d\by, \pi(\bx,d\by) = \pi(\by|\bx)d\by, \pi(d\bx,\by) = \pi(\bx|\by)d\bx$. 

Notation $\mu\ll \nu$ means measure $\mu$ is absolutely continuous with respect to $\nu$. Throughout the paper, we assume $\mu, \nu \in \mathcal{M}_+(X)$ are absolutely continuous with respect to $d\bx$, and we refer to them simply as absolutely continuous measures. We also assume $\rho_t: [0,1]\rightarrow \mathcal{M}_+(X)$ is an absolutely continuous curve of measures. Given a Borel map $\mathbf{T}:X\rightarrow Y$, $\bT_{\sharp}$ denotes the pushforward operator, which maps $\mu\in\mathcal{M}_+(X)$ to $\bT_{\sharp}\mu = \mu\circ\mathbf{T}^{-1}$. For any (possibly vector-valued) function $f(\bx,t)$, we will interchangeably use $f, f_t, f_t(\bx), f(\bx,t)$ to denote it, depending on which notation is more appropriate in context. Indicator function of a convex set $\mathcal{C}$ is denoted by $\iota_{\mathcal{C}}$, which takes the value 0 on $\mathcal{C}$ and $\infty$ everywhere else.

Since the paper involves two synchronized spaces $X$ and $Y$, we will always assume a convex and open subset (flat) $X\subseteq\mathbb{R}^m$. For space $Y$, we distinguish two different cases. In Case I, we assume $Y$ is a convex and open subset (flat) $Y\subset\mathbb{R}^n$. This case will apply to defining Monge (\ref{eqn:MongeOT}) and Kantorovich (\ref{eqn:KantorovichOT}) OT problems, and the Kantorovich UnSyncOT (\ref{eqn:KantorovichUnSynOT}), and the Kantorovich synchronized Fisher-Rao problem (\ref{eqn:FR_SynOT_2Spaces}). In Case II, we need to introduce a $C^1$-embedding from $\mathbb{R}^m$ to $\mathbb{R}^n (m\le n)$, denoted by $\mathbf{T}: X \rightarrow Y = \bT(X)$ such that $Y$, as the image of $X$, is a $C^1$-embedded (curved) $m$-dimensional submanifold of $\mathbb{R}^n$. This case will be apply to Monge UnSyncOT (\ref{eqn:SynOT_2Spaces}) and Monge synchronized Fisher-Rao problem (\ref{eqn:FR_SynOT_2Spaces}). Indeed, $\mathbf{T}$, as a $C^1$-embedding, is a map $\bT: \mathbb{R}^m \rightarrow \bT(\mathbb{R}^m)$. In this paper, we mainly focus on the case in which $\bT$ is restricted on a convex and open subset $X\subset \mathbb{R}^m$, such that $Y = \bT(X)$ is relatively open to $\bT(\mathbb{R}^m)$. 

The induced Riemannian metric of $Y$ is denoted by $g$. We denote the resulting Riemannian manifold by $(Y,g)$, with the geodesic distance
\begin{align}\label{eqn:geodesic_distance}
    d^2_g(\by,\by'): = \inf_{\gamma}\left\{ \int_0^1 |\dot{\gamma}(t)|^2_g \ dt \ \middle|\  \gamma\in C^1((0,1);Y), \gamma(0) = \by, \gamma(1) = \by' \right\}.
\end{align}

Now we need to briefly review the induced Riemannian metric $g$, and the differential operators on $(Y,g)$.
Using the local coordinates $\bx = (x^1,\cdots, x^m)$, the coordinate tangent basis on tangent space $T_{\by}Y$ is given by 
\begin{align*}
    E_{i}:= \partial_{x^i}\mathbf{T}(\bx) \in T_{\by}Y, \quad \by = \bT(\bx), \quad i=1,\cdots,m.
\end{align*}
The ambient Euclidean inner product on $\mathbb{R}^n$, when restricted to each tangent space $T_\by Y$, defines the induced Riemannian metric $g_{\by}$ on $Y$:
\begin{align*}
    &g_{ij}:= \langle E_i, E_j \rangle_{\mathbb{R}^n}, \quad 
    g:= \big(g_{ij}\big)_{ij}, \quad (g^{ij}) = (g_{ij})^{-1} \\
    & g_\by(\cdot,\cdot): T_\by Y\times T_\by Y \rightarrow \mathbb{R}, \ g_\by(\bv,\bw) = v^i w^j g_{ij}, \mathrm{\ for\ } \bv = v^i E_i, \bw = w^i E_i, \\
    & \|\bv\|_g^2 := g_\by(\bv,\bv),
\end{align*}
where the standard Einstein notation is adopted.
We also denote 
\begin{align}\label{eqn:G_norm}
    G(\bx) := \big(\nabla\mathbf{T}(\bx)\big)^{\mathrm{T}}\nabla\mathbf{T}(\bx), \quad \|\bu\|_G^2:= \bu^\mathrm{T} G \bu , \ \bu\in\mathbb{R}^m.
\end{align}
Given $\bu(\bx)\in T_\bx X$, we have $\bv(\bT(\bx)):= (\nabla\bT)\bu(\bx) = \sum_{i=1}^m u^iE_i\in T_{\bT(\bx)}Y$, and the two norms $\|\cdot\|_g$ and $\|\cdot\|_G$ are related by
\begin{align}\label{eqn: two_norms}
    \|\bv(\bT(\bx))\|_g^2 = u^i u^j g_{ij} = u^i u^j \langle E_i, E_j \rangle_{\mathbb{R}^n} = \| (\nabla\bT)\bu(\bx) \|^2_{\mathbb{R}^n} = \|\bu(\bx)\|_G^2.
\end{align}

The Riemannian volume on $Y$ is given as
\begin{align*}
    d\mathrm{vol}_Y(\by) = \sqrt{\mathrm{det}\, g(\bx)} \ d\bx,
\end{align*}
which is equivalent to the Hausdorff surface measure $d\mathcal{H}^m$ on $Y$. Given a smooth function $f$, the gradient on $Y$ is defined as
\begin{align}
    \nabla_Y f\big(\mathbf{T}(\bx)\big) = \sum_{i} \Big( \sum_{j} g^{ij}(\bx)\partial_{x^j}(f\circ\mathbf{T})(\bx) \Big)E_i(\bx).
\end{align}
Given tangent vector field $W(\by)\in T_{\by}Y$ under coordinate tangent basis $W(\mathbf{T}(\bx)) = W^{i}(\bx)E_i(\bx)$, the divergence operator on $Y$ is defined as
\begin{align}
    \mathrm{div}_Y W(\mathbf{T}(\bx)) = \frac{1}{\sqrt{\det g(\bx)}}\partial_{x^i}\Big(\sqrt{\det g(\bx)}\ W^i(\bx)\Big).
\end{align}
The gradient and divergence operators on $Y$ satisfy the integration-by-part formula
\begin{align}
    \int_Y \phi\ \mathrm{div}_Y W d\mathrm{vol}_Y = -\int_Y \langle \nabla_Y\phi, W\rangle_g d\mathrm{vol}_Y, \quad \forall \phi\in C_{c}^{\infty}(Y).
\end{align}

While in the curved space $Y = \bT(X)$ we denote the gradient and divergence operators by $\nabla_Y$ and $\div_Y$, we will still use the standar notation $\nabla$ and $\nabla\cdot$ for these operators in the flat space $X$.

\subsection{Review of Optimal Transport Theory}

Let $X\subseteq\mathbb{R}^m$ and $Y\subseteq\mathbb{R}^n$ be convex and open subsets of Euclidean spaces. Given two probability measures $\bar{\rho}_0\in\mathcal{P}(X), \bar{\rho}_1\in\mathcal{P}(Y)$, and $c: X \times Y \rightarrow [0,\infty]$ lower semi-continuous, the Monge and Kantorovich OT problems are defined as follows \cite{figalli2021invitation}:
\begin{align}
&C_{\mathrm{M}}(\bar{\rho}_0,\bar{\rho}_1) = \inf_{\bT}\left\{ \int_{X} c(\mathbf{x},\mathbf{T}(\mathbf{x})) \ d\bar{\rho}_0(\mathbf{x})\ |\  {\mathbf{T}_{\sharp} \bar{\rho}_0 = \bar{\rho}_1} \right\}, \label{eqn:MongeOT}\\
&C_{\mathrm{K}}(\bar{\rho}_0,\bar{\rho}_1) = \inf_{\pi}\left\{ \int_{X\times Y} c(\mathbf{x},\mathbf{y}) \ d\pi(\mathbf{x},\mathbf{y})\ |\  \pi \in \Gamma(\bar{\rho}_0,\bar{\rho}_1) \right\}. \label{eqn:KantorovichOT}
\end{align}
Here $\mathbf{T}_{\sharp}$ is the pushforward operator, and $c(\mathbf{x},\mathbf{y})$ is the cost function. $\Gamma(\bar{\rho}_0,\bar{\rho}_1)$ is the set of couplings between $\bar{\rho}_0$ and $\bar{\rho}_1$. 

When $X=Y\subseteq\mathbb{R}^m$ are convex and open, and given $\bar{\rho}_{0,1} \in \mathcal{P}_{p}(X)$, their $p$-Wasserstein distance, $W_p(\bar{\rho}_0,\bar{\rho}_1)$, is defined as \cite{figalli2021invitation}
\begin{align}\label{eqn:pWasserstein_flat}
  W_{p}^p(\bar{\rho}_0,\bar{\rho}_1) := \inf_{\pi \in \Gamma(\bar{\rho}_0,\bar{\rho}_1)} \int_{X\times X} \| \mathbf{x}-\mathbf{x}' \|^p\ d\pi(\mathbf{x},\mathbf{x}').
\end{align}
When $p=2$, Benamou and Brenier \cite{benamou2000computational} showed that the squared 2-Wasserstein distance $W_2^2(\bar{\rho}_0,\bar{\rho}_1)$ for $\bar{\rho}_{0,1}\in\mathcal{P}_2(X)$ is equivalent to the following Benamou-Brenier dynamical OT formulation
\begin{align}\label{eq:DynamicalOT_flat}
    {\mathrm{BB}}_X^2(\bar{\rho}_0,\bar{\rho}_1) = \inf_{(\rho_t,\mathbf{u}_t)\in\mathcal{C}_{\mathrm{BB}}^X}  A_{\mathrm{BB}}[\rho_t, \mathbf{u}_t] = \int_0^1\int_{X} \|\mathbf{u}_t(\mathbf{x})\|^2 \rho_t(\bx)  d\bx dt,
\end{align}
where we define the set of constraints for $(\rho_t,\mathbf{u}_t)\in \mathcal{P}_2(X)\times TX$ as
\begin{align}\label{eqn:feasible_BB_flat}
  \mathcal{C}^X_{\mathrm{BB}}(\bar{\rho}_0,\bar{\rho}_1):= 
  \left\{ (\rho_t,\mathbf{u}_t) \ \middle| 
  \begin{array}{l}
  \partial_t\rho_t + \nabla\cdot(\rho_t \bu_t)=0 \mathrm{\ in\ } \mathcal{D}'(X), \ \rho_{0,1}=\bar{\rho}_{0,1}, \ \mathbf{u}_t\cdot \mathbf{n}|_{\partial X} = 0
  \end{array}
  \right\}.
\end{align}
Here we adopt the notation $T X$ to represent the trivial tangent bundle of the flat $X$, in other words, $u_t$ is a time-dependent vector field on $X$.

When $X=Y\subseteq\mathbb{R}^m$ are $m$-dimensional Riemannian manifold $(X,g)$, and given $\bar{\rho}_{0,1} \in \mathcal{P}_{p}(X)$, their $p$-Wasserstein distance is defined as \cite{figalli2021invitation}
\begin{align}\label{eqn:pWasserstein_curved}
  W_{p}^p(\bar{\rho}_0,\bar{\rho}_1) := \inf_{\pi \in \Gamma(\bar{\rho}_0,\bar{\rho}_1)} \int_{X\times X} d_g(\mathbf{x},\mathbf{x}')^p\ d\pi(\mathbf{x},\mathbf{x}').
\end{align}
in which $d_g(\bx,\bx')$ is the geodesic distance on $(X,g)$ as defined in (\ref{eqn:geodesic_distance}).
When $p=2$, Otto and Villani \cite{otto2000generalization} showed that the Benamou-Brenier equivalence also holds between the 2-Wasserstein distance $W_2^2(\bar{\rho}_0,\bar{\rho}_1)$ on manifold for $\bar{\rho}_{0,1}\in\mathcal{P}_2(X)$ and the dynamical OT on manifold
\begin{align}\label{eq:DynamicalOT_curved}
    {\mathrm{BB}}_X^2(\bar{\rho}_0,\bar{\rho}_1) = \inf_{(\rho_t,\mathbf{u}_t)\in\mathcal{C}^X_{\mathrm{BB}}}  A_{\mathrm{BB}}[\rho_t, \mathbf{u}_t] = \int_0^1\int_{X} \|\mathbf{u}_t(\mathbf{x})\|^2 \rho_t(\bx)  d\mathrm{vol}_X dt,
\end{align}
where we define the set of constraints for $(\rho_t,\mathbf{u}_t)\in \mathcal{P}_2(X)\times TX$ as
\begin{align}\label{eqn:feasible_BB_curved}
  \mathcal{C}^X_{\mathrm{BB}}(\bar{\rho}_0,\bar{\rho}_1)= 
  \left\{ (\rho_t,\mathbf{u}_t) \ \middle| 
  \begin{array}{l}
  \partial_t\rho_t + \div_X(\rho_t \bu_t)=0 \mathrm{\ in\ } \mathcal{D}'(X), \ \rho_{0,1}=\bar{\rho}_{0,1}, \langle \mathbf{u}_t, \mathbf{n}|_{\partial X} \rangle_g = 0
  \end{array}
  \right\}.
\end{align}
Here with abuse of notation, we use $W_p$, $\mathrm{BB}_X$, $A_{\mathrm{BB}}$, $\mathcal{C}_{\mathrm{BB}}^X$ both in the Euclidean and Riemannian settings. The intended meaning will be clear from context.

Using Otto's calculus, the dynamical OT (\ref{eq:DynamicalOT_flat}) (similarly for (\ref{eq:DynamicalOT_curved})) can be also reformulated in terms of the Wasserstein metric and Wasserstein norm
\begin{align}\label{eqn:DynamicalOT02}
  {\mathrm{BB}}_X^2(\bar{\rho}_0,\bar{\rho}_1) = \inf_{\rho_t} \left\{ \int_0^1 \left\|\partial_t\rho_t \right\|_{-1,{\rho_t}}^2 dt \ \middle | \ (\rho_t)_{t\in[0,1]}\in AC([0,1];\mathcal{P}_2(X)),\,\rho_{0,1} = \bar{\rho}_{0,1} \right\}.
\end{align}
The Wasserstein norm of the derivative $\partial_t\rho$ at $\rho_t(\cdot)$ in \eqref{eqn:DynamicalOT02} is defined as a minimization
\begin{align}\label{eqn:WassersteinNorm}
    \|\partial_t\rho_t\|_{-1,\rho_t}^2 : = \inf_{\mathbf{u}_t} 
    \left\{ \int_{X} \|\mathbf{u}_t(\bx)\|^2 \rho_t(\bx) d\bx: -\nabla\cdot(\rho_t\mathbf{u}_t)=\partial_t\rho_t, \ \mathbf{u}_t\cdot\mathbf{n}|_{\partial X} = 0 \right\}.
\end{align}
There are some other notations in the field of OT for the Wasserstein norm such as $\|\cdot\|_{\rho_t}$, $\|\cdot\|_{T_{\rho_t}}$.
Denoting by $(-\Delta_{\rho})^{-1}h$ the Wasserstein potential $\psi_h$ associated with a generic source term $h$,
\begin{align*}
    \begin{cases}
    (-\Delta_{\rho})\psi_h: = -\nabla\cdot(\rho \nabla \psi_h) = h & \mathrm{in } X, \\
    \frac{\partial \psi_h}{\partial \bn} = 0 & \text{on } \partial X,
    \end{cases}
\end{align*}
the Wasserstein norm has an alternative notation $\|h\|_{-1,{\rho}}^2 = \langle (-\Delta_{\rho})^{-1}h, h \rangle_{L^2} = \|(-\Delta_{\rho})^{-\frac{1}{2}}h\|^2$, therefore
\begin{align*}
   W_2^2(\bar{\rho}_0,\bar{\rho}_1) =  {\mathrm{BB}}_X^2(\bar{\rho}_0,\bar{\rho}_1)
    = \inf_{\rho_t \in \Gamma(\bar{\rho}_0,\bar{\rho}_1)} \int_0^1 \|(-\Delta_{\rho_t})^{-\frac{1}{2}} \partial_t\rho_t\|^2 dt.
\end{align*}

For later use, we need to extend $W_2$ over flat $X$ (\ref{eqn:pWasserstein_flat}) into the one with $\bA$-norm,
\begin{align}\label{eqn:pWasserstein_flat_A}
  W_{2,\bA}^2(\bar{\rho}_0,\bar{\rho}_1) := \inf_{\pi \in \Gamma(\bar{\rho}_0,\bar{\rho}_1)} \int_{X\times X} d_\bA(\bx,\bx')^2  d\pi(\mathbf{x},\mathbf{x}') := \int_{X\times X}   \| \mathbf{x}-\mathbf{x}' \|_\bA^2 d\pi(\mathbf{x},\mathbf{x}')
\end{align}
where $\bA\in\bbR^{m\times m}$ is symmetric positive definite and $\|\cdot\|_\bA$ is defined as in (\ref{eqn:G_norm}). For the Benamou-Brenier dynamical OT (\ref{eq:DynamicalOT_flat}) on flat $X$, we can also extend it into one with $\bA$-norm,
\begin{align}\label{eq:DynamicalOT_flat_A}
    {\mathrm{BB}}_{X,\bA}^2(\bar{\rho}_0,\bar{\rho}_1) := \inf_{(\rho_t,\mathbf{u}_t)\in\mathcal{C}_{\mathrm{BB}}^X(\bar{\rho}_0,\bar{\rho}_1)}  A_{\mathrm{BB},\bA}[\rho_t, \mathbf{u}_t] = \int_0^1\int_{X} \|\mathbf{u}_t(\mathbf{x})\|_{\bA}^2 \rho_t(\bx)  d\bx dt.
\end{align}
Let $\bS$ be a $C^1$-embedding $\bS: X \rightarrow Z = \bS(X)\subseteq\bbR^n$, then $(Z,g)$ is a $C^1$-embedded $m$-dimensional Riemannian submanifold of $\bbR^n$ with induced metric $g$, we can consider $\bS$-induced 2-Wasserstein distance $W_2(\bS_{\sharp}\bar{\rho}_0,\bS_\sharp\bar{\rho}_1)$ on curved $Z$ as defined in (\ref{eqn:pWasserstein_curved}), and $\bS$-induced Benamou-Brenier dynamical OT $\mathrm{BB}_Z(\bS_{\sharp}\bar{\rho}_0,\bS_\sharp\bar{\rho}_1)$ as defined in (\ref{eq:DynamicalOT_curved}). We have the following lemma for their relations \cite{villani2009optimal,pooladian2023neural}.

\begin{lemma}\label{lemma:W2_Anorm}
    Let $X\subseteq\bbR^m$ be open and convex, $\bS$ be a $C^1$-embedding $\bS: X \rightarrow Z = \bS(X)\subseteq\bbR^n$, and $\bA = (\nabla\bS)^\rmT(\nabla\bS)$, then we have
    \begin{align}
        W_{2,\mathbf{A}}(\bar{\rho}_0,\bar{\rho}_1) 
        =\mathrm{BB}_{X,\mathbf{A}}(\bar{\rho}_0,\bar{\rho}_1)= \mathrm{BB}_{Z}(\mathbf{S}_{\sharp}\bar{\rho}_0, \mathbf{S}_{\sharp}\bar{\rho}_1) 
        = W_2({\mathbf{S}}_{\sharp}\bar{\rho}_0, {\mathbf{S}}_{\sharp}\bar{\rho}_1).
    \end{align}
\end{lemma}
\begin{proof}
    We first prove that $W_{2,\mathbf{A}}(\bar{\rho}_0,\bar{\rho}_1) = W_2({\mathbf{S}}_{\sharp}\bar{\rho}_0, {\mathbf{S}}_{\sharp}\bar{\rho}_1)$. Given the Riemannian manifold $(Z,g)$, we equip $X$ with the pullback metric $\bA:=(\nabla\bS)^\rmT(\nabla\bS)$. Consider every $C^1$-curve $\gamma$ connecting $\bx$ and $\bx'$ in $(X,\bA)$. Due to identity (\ref{eqn: two_norms}), we have $\|\frac{d}{dt}\gamma\|_\bA = \|\frac{d}{dt}\bS(\gamma)\|_g$. Therefore $d_\bA(\bx,\bx') = d_g(\bS(\bx),\bS(\bx'))$.
    For any coupling $\pi \in \Gamma(\bar{\rho}_0,\bar{\rho}_1)$, we have
    \begin{align*}
        \int_{X\times X} d_\bA(\bx,\bx')^2 d\pi = \int_{X\times X} d_g(\bS(\bx),\bS(\bx'))^2 d\pi 
        =\int_{Z\times Z} d_g(\bz,\bz')^2 d((\bS\times\bS)_\sharp \pi).
    \end{align*}
    Minimizing over the coupling $\pi$ on two sides, it follows that $W_{2,\mathbf{A}}(\bar{\rho}_0,\bar{\rho}_1) = W_2({\mathbf{S}}_{\sharp}\bar{\rho}_0, {\mathbf{S}}_{\sharp}\bar{\rho}_1)$. In other words, $\bS_\sharp: (\mathcal{P}_2(X),W_{2,\bA})\rightarrow (\mathcal{P}_2(Z), W_2)$ is a bijective isometry.

    By the Benamou-Brenier theory \cite{benamou2000computational,otto2000generalization}, we have $\mathrm{BB}_{Z}(\mathbf{S}_{\sharp}\bar{\rho}_0, \mathbf{S}_{\sharp}\bar{\rho}_1) 
    = W_2({\mathbf{S}}_{\sharp}\bar{\rho}_0, {\mathbf{S}}_{\sharp}\bar{\rho}_1)$. So it remains to prove $\mathrm{BB}_{X,\mathbf{A}}(\bar{\rho}_0,\bar{\rho}_1) = \mathrm{BB}_{Z}(\mathbf{S}_{\sharp}\bar{\rho}_0, \mathbf{S}_{\sharp}\bar{\rho}_1)$.

    Define the feasible set of $\mathrm{BB}_Z(\bS_\sharp\bar{\rho}_0,\bS_\sharp\bar{\rho}_1)$ for $(\eta_t,\mathbf{w}_t)\in \mathcal{P}_2(Z)\times T Z$ as 
    \begin{align}\label{eqn:C_BB^Z}
      \mathcal{C}^Z_{\mathrm{BB}}:= 
      \left\{ (\eta_t,\mathbf{w}_t) \ \middle| 
      \partial_t\eta_t + \div_Z(\eta_t \bw_t)=0 \mathrm{\ in\ } \mathcal{D}'(Z), \ \eta_{0,1}=\bS_{\sharp}\bar{\rho}_{0,1}, \  \langle \mathbf{w}_t, \mathbf{n}|_{\partial Z} \rangle_g = 0
      \right\}.
    \end{align}
    Then the following maps define a one-to-one correspondence between feasible pairs $(\rho_t,\mathbf{u}_t)\in \mathcal{C}^X_{\mathrm{BB}}$ and $(\eta_t,\bw_t)\in \mathcal{C}^Z_{\mathrm{BB}}$,
    \begin{align}
        (\mathrm{forward}): &\ \mathbf{S}_{\sharp}(\rho_t d\bx) = \eta_t d\mathrm{vol}_Z,\ \mathbf{w}_t(\mathbf{S}(\bx))=\nabla \mathbf{S}(\bx)\mathbf{u}_t(\bx)\in T_{\mathbf{S}(\bx)}Z, \label{eqn:forward_map_XZ}\\
        (\mathrm{backward}):&\ (\mathbf{S}^{-1})_{\sharp}(\eta_t d\mathrm{vol}_Z) = \rho_t d\bx, \ \mathbf{u}_t(\bx)= (\nabla \mathbf{S}(\bx))^{\dagger}\mathbf{w}_t(\mathbf{S}(\bx)), \label{eqn:backward_map_XZ}
    \end{align}
    with $(\nabla \mathbf{S}(\bx))^{\dagger}: = \bA(\bx)^{-1}(\nabla \mathbf{S}(\bx))^{\mathrm{T}}$, and the following actions coincide:
    \begin{align}\label{eqn:Yaction_equal_Xaction_Anorm}
        \int_0^1\int_Z \|\mathbf{w}_t(\bz)\|_g^2 (\eta_t d\mathrm{vol}_Z) dt = 
        \int_0^1\int_X \|\mathbf{w}_t(\bS(\bx))\|_g^2 (\rho_t d\bx) dt =
        \int_0^1\int_X \|\mathbf{u}_t(\bx)\|_{\bA}^2 (\rho_t d\bx) dt,
    \end{align}
    in which the last equation is due to (\ref{eqn: two_norms}). Minimizing over the corresponding feasible sets on two sides, the desired identity holds: $\mathrm{BB}_{X,\mathbf{A}}(\bar{\rho}_0,\bar{\rho}_1)= \mathrm{BB}_{Z}(\mathbf{S}_{\sharp}\bar{\rho}_0, \mathbf{S}_{\sharp}\bar{\rho}_1)$.
\end{proof}

The dynamical OT (\ref{eq:DynamicalOT_flat}) (similarly the manifold case (\ref{eq:DynamicalOT_curved})) has been extended to unbalanced dynamical OT, independently by three groups \cite{chizat2018unbalanced, liero2016optimal,liero2018optimal,kondratyev2016new}, which is now called  Wasserstein-Fisher-Rao (WFR) distance. Following the notation in \cite{liero2016optimal}, unbalanced dynamical OT in WFR form introduces a growth term as follows
\begin{align}\label{eqn:WFR}
  {\mathrm{WFR}^2_{\alpha,\beta}}(\bar{\rho}_0,\bar{\rho}_1) = \inf_{(\rho,\mathbf{u},g)\in\mathcal{C}_{\mathrm{WFR}_{\alpha,\beta}}} A_{\mathrm{WFR}_{\alpha,\beta}}[\rho,\mathbf{u},g] = \int_0^1 \int_{X}\left(\alpha\|\mathbf{u}_t\|^2 + \beta |g_t|^2 \right) \rho_t d\bx dt,
\end{align}
in which the feasible set of $(\rho,\mathbf{u},g)\in \mathcal{P}_2(X)\times TX \times C^1(X)$ is defined as
\begin{align}\label{eqn:feasible_WFR_flat}
  \mathcal{C}_{\mathrm{WFR}_{\alpha,\beta}}(\bar{\rho}_0,\bar{\rho}_1) = 
  \left\{ (\rho,\mathbf{u},g) \ \middle| 
  \partial_t\rho + \alpha\nabla\cdot(\rho \bu)=
  \beta\rho g \mathrm{\ in\ } \mathcal{D}'(X), \ \rho_{0,1}=\bar{\rho}_{0,1}, \mathbf{u}_t\cdot \mathbf{n}|_{\partial X} = 0
  \right\}.
\end{align}
Here $\alpha$ and $\beta$ are weights controlling the trade-off between mass transportation and mass creation/destruction, and $\rho_t(\cdot): (0,1)\rightarrow\mathcal{M}_+(X)$ is an absolutely continuous and unnormalized measure curve, since the growth term does not guarantee the mass conservation $\int_X \partial_t\rho_t d\bx = 0$. With abuse of notation, we use $g$ both for Riemannian metric and the growth term. But the intended meaning is clear from context.

The WFR formulation (\ref{eqn:WFR}) is equivalent to \cite{liero2016optimal}
\begin{align*}
   \inf\limits_{(\rho,g)} \left\{ \int_0^1 \int_{X} (\alpha\|\nabla g\|^2 + \beta |g|^2) \rho d\bx dt \ \middle| \ \partial_t\rho + \alpha\nabla\cdot(\rho\nabla g) = \beta\rho g,  \rho_{0,1} = \bar{\rho}_{0,1}, \nabla g\cdot \mathbf{n}|_{\partial X} = 0 \right\}.
\end{align*}
Namely, the optimal solution $(\rho^*,\mathbf{u}^*,g^*)$ of the WFR dynamical OT (\ref{eqn:WFR}) satisfies $\mathbf{u}^* = \nabla g^*$.

Unbalanced dynamical OT can also be formulated as in \cite{chizat2018interpolating} using one trade-off parameter $\lambda$,
\begin{align}\label{eqn:WFR02}
  {\mathrm{WFR}^2_{\lambda}}(\bar{\rho}_0,\bar{\rho}_1) = \inf_{(\rho,\mathbf{v},h)\in\mathcal{C}_{\mathrm{WFR}_{\lambda}}} A_{\mathrm{WFR}_{\lambda}}[\rho,\mathbf{v},h] = \int_0^1 \int_{X} (\|\mathbf{v}_t\|^2 + \lambda |h_t|^2) \rho_t d\bx dt,
\end{align}
in which the feasible set of $(\rho,\mathbf{v},h)\in \mathcal{P}_2(X)\times TX \times C^1(X)$ is defined as
\begin{align}\label{eqn:feasible_WFR_flat_02}
  \mathcal{C}_{\mathrm{WFR}_{\lambda}}(\bar{\rho}_0,\bar{\rho}_1) = 
  \left\{ (\rho,\mathbf{v},h) \ \middle| \ 
  \partial_t\rho + \nabla\cdot(\rho \bv)=
  \rho h \mathrm{\ in\ } \mathcal{D}'(X), \ \rho_{0,1}=\bar{\rho}_{0,1}, \mathbf{v}_t\cdot \mathbf{n}|_{\partial X} = 0
  \right\}.
\end{align}

\begin{lemma}
    The two WFR formulations (\ref{eqn:WFR}) and (\ref{eqn:WFR02}) of the unbalanced dynamical optimal transport are equivalent in the sense that 
    \begin{align}
    {\mathrm{WFR}^2_{\alpha/\beta}} = \alpha {\mathrm{WFR}^2_{\alpha,\beta}}
    \end{align}
\end{lemma}
\begin{proof}
    The equivalence can be easily verified by changing the variables
    \begin{align*}
        \mathbf{v} = \alpha \mathbf{u}, \quad h = \beta g, \quad \lambda = \frac{\alpha}{\beta},
    \end{align*}
    in (\ref{eqn:WFR02}), and then resulting in $\alpha {\mathrm{WFR}^2_{\alpha,\beta}}$ from (\ref{eqn:WFR}).
\end{proof}

Another interesting problem related to the WFR dynamical OT is the Fisher-Rao (FR) problem
\begin{align}\label{eqn:FR}
    \mathrm{FR}^2(\bar{\rho}_0,\bar{\rho}_1)=\inf_{(\rho,g)\in\mathcal{C}_{\mathrm{FR}}} A_{\mathrm{FR}}[\rho,g] = \int_0^1\int_X |g_t(\bx)|^2\rho_t(\bx) d\bx dt
\end{align}
with the feasible set
\begin{align}\label{eqn:FR_feasibility}
    \mathcal{C}_{\mathrm{FR}}(\bar{\rho}_0,\bar{\rho}_1):=\Big\{ (\rho,g) \Big| \ \partial_t\rho = \rho g,\ \rho_{0,1} = \bar{\rho}_{0,1} \Big\}.
\end{align}
The optimal solution $(\rho^*,g^*)$ of the FR problem (\ref{eqn:FR}) is explicitly given by \cite{chizat2018interpolating}
\begin{align*}
    \rho^*_t(\bx) = (t\sqrt{\bar{\rho}_1}+(1-t)\sqrt{\bar{\rho}_0})^2, \quad g^*_t(\bx) = \partial_t\rho^*(\bx)/\rho^*(\bx),
\end{align*}
and the corresponding optimal action is equal to the squared Hellinger distance \cite{hellinger1909neue}
\begin{align}\label{eqn:FR_explicitform}
    \mathrm{FR}^2(\bar{\rho}_0,\bar{\rho}_1) = A_{\mathrm{FR}}[\rho^*,g^*] = \int_X 4(\sqrt{\bar{\rho}_1}-\sqrt{\bar{\rho}_0})^2 d\bx := 4\mathrm{dis}^2_\mathrm{H}(\bar{\rho}_0,\bar{\rho}_1).
\end{align}

\begin{lemma}(\cite{chizat2018interpolating, liero2016optimal})
As $\beta\rightarrow 0$ (or equivalently $\lambda\rightarrow \infty$), if $\bar{\rho}_0$ and $\bar{\rho}_1$ are of equal mass, then WFR formula (\ref{eqn:WFR}) degenerates to the balanced dynamical OT (\ref{eq:DynamicalOT_flat}). In contrast, as $\alpha\rightarrow 0$ (or equivalently $\lambda\rightarrow 0$), WFR formula (\ref{eqn:WFR}) approaches the FR distance (\ref{eqn:FR}) in the sense that \cite{chizat2018interpolating}
\begin{align*}
    \lim_{\alpha\rightarrow 0}{\mathrm{WFR}^2_{\alpha,\beta}}(\bar{\rho}_0,\bar{\rho}_1) = {\mathrm{WFR}^2_{0,\beta}}(\bar{\rho}_0,\bar{\rho}_1) = \dfrac{1}{\beta}{\mathrm{FR}}^2(\bar{\rho}_0,\bar{\rho}_1) = \frac{4}{\beta}\mathrm{dis}^2_{\mathrm{H}}(\bar{\rho}_0,\bar{\rho}_1) = \frac{4}{\beta} \int_X (\sqrt{\bar{\rho}_1}-\sqrt{\bar{\rho}_0})^2 d\bx.
\end{align*}
\end{lemma}

Hereafter, we will adopt the WFR unbalanced dynamical OT (\ref{eqn:WFR}) with two trade-off parameters $(\alpha, \beta)$ to define the unbalanced synchronized OT. This formulation is more convenient when exploring the two extreme cases. See Sections \ref{section:Extreme_case_I} and \ref{section:Extreme_case_II}.

 The WFR formulation of the unbalanced dynamical OT (\ref{eqn:WFR}) can also fit into the general framework of the dissipation distance associated with the dissipation potential $\Psi$ \cite{liero2016optimal, peletier2014variational}, 
 \begin{align}\label{eqn:dissipation_potential}
     D_{\mathrm{diss}}(\bar{\rho}_0,\bar{\rho}_1) : = \inf \left\{ \int_0^1 \Psi(\rho,\dot{\rho}) dt \ \middle| \  \rho\in H^1\Big([0,1];\mathcal{M}_+(X)\Big), \rho_{0,1} = \bar{\rho}_{0,1} \right\},
 \end{align}
 in which overhead dot stands for the time derivative,  $\Psi(\rho,\sigma) = \Psi_{\rho}(\sigma) = \langle \mathbb{G}(\rho)\sigma,\sigma\rangle$ is a quadratic form, and $\mathbb{G}(\rho): T_\rho \mathcal{M}_+ \rightarrow T^*_\rho\mathcal{M}_+$ is called the Riemannian operator, and usually assumed symmetric and positive definite. The Legendre conjugate of $\Psi$ with respect to the second variable, denoted by $\Psi^*$, is formulated as $\Psi^*(\rho,g) = \Psi^*_{\rho}(g)= \langle g, \mathbb{K}(\rho)g\rangle$. Here $\mathbb{K}$ is called the Onsager operator. The Riemannian and Onsager operators are inverse to each other: $\mathbb{G}(\rho)^{-1} = \mathbb{K}(\rho)$ and $\mathbb{K}(\rho)^{-1} = \mathbb{G}(\rho)$.

 In the theory of Kantorovich-Wasserstein distance \cite{villani2009optimal,liero2016optimal, jordan1998variational,otto2001geometry}, the Onsager operator $\mathbb{K}$ is of form
 \begin{align*}
     \mathbb{K}(\rho)g = -\nabla\cdot(M(\rho)\nabla g) + H(\rho)g,
 \end{align*}
 in which $M(\rho)$ is a symmetric and positive definite mobility tensor, and $H(\rho)$ is a reaction matrix. The Onsager operator $\mathbb{K}(\rho)$ can be seen as the inverse of a metric tensor $\mathbb{G}(\rho)$ that gives rise to a geodesic distance between $(\bar{\rho}_0,\bar{\rho}_1)$ defined as
 \begin{align*}
     \mathrm{dis}^2_{\mathbb{K}}(\bar{\rho}_0,\bar{\rho}_1): = \inf_{\rho}\left\{ \int_0^1 \langle \mathbb{G}(\rho)\dot{\rho},\dot{\rho} \rangle dt \ \middle| \ \rho_{0,1} = \bar{\rho}_{0,1}\right\}
     = \inf_{\rho}\left\{ \int_0^1 \langle \mathbb{K}^{-1}(\rho)\dot{\rho},\dot{\rho}\rangle dt \ \middle| \ \rho_{0,1} = \bar{\rho}_{0,1} \right\}. 
 \end{align*}
 Since in general the inverse of $\mathbb{K}$ is intractable, it is better to use $\Psi^*$ to reformulate the general Wasserstein distance in terms of the dual variable $g = \mathbb{K}^{-1}(\rho)\dot{\rho}$,
 \begin{align}
  \mathrm{dis}^2_{\mathbb{K}}(\bar{\rho}_0,\bar{\rho}_1)
  =& \inf_{\rho}\left\{ \int_0^1 \langle g,\mathbb{K}(\rho)g\rangle dt \ \middle| \ \partial_t\rho = \mathbb{K}(\rho)g, \ \rho_{0,1} = \bar{\rho}_{0,1} \right\}, \label{eqn:dissipation_distance_K}\\
  =& \inf_{\rho}\left\{ \int_0^1 \int_X \|\nabla g\|_{M}^2 + \|g\|_{H}^2 d\bx dt\ \middle| \ \partial_t\rho = -\nabla\cdot (M\nabla g) + Hg, \ \rho_{0,1} = \bar{\rho}_{0,1} \right\}. \label{eqn:general_Waserstein}
 \end{align}
 It is shown in \cite{liero2016optimal} that general Wasserstein distance (\ref{eqn:general_Waserstein}) is equivalent to the minimization
 \begin{align}
     \mathrm{dis}_{\mathbb{K}}^2(\bar{\rho}_0,\bar{\rho}_1): =\inf_{\rho}\left\{ \int_0^1 \int_X \|\mathbf{u}\|_{M}^2 + \|g\|_{H}^2 d\bx dt\ \middle| \ \partial_t\rho = -\nabla\cdot (M\mathbf{u}) + Hg, \ \rho_{0,1} = \bar{\rho}_{0,1} \right\}. 
 \end{align}
Namely, the velocity-reaction pair $(\mathbf{u},g)$ are related by $\mathbf{u} = \nabla g$ at optima. 

Now under the framework of dissipation distance, the Onsager operator $\mathbb{K}$ (or sometimes the dissipation potential $\Psi_{\rho}(\sigma) = \langle \mathbb{G}(\rho)\sigma,\sigma\rangle$ if $\mathbb{K}$ is intractable, see (\ref{list:K_SyncOT_dissipation}) and (\ref{list:K_UnSyncOT_dissipation})), determines the general Wasserstein distance. Here we list several special cases:
\begin{subequations}\label{eqn:K_operators}
  \begin{align}
    \text{BB}: \quad & \mathbb{K}(\rho)g=-\nabla\cdot(\rho\nabla g); \\
    \text{FR}: \quad & \mathbb{K}(\rho)g= \rho g; \\
    \text{WFR}: \quad & \mathbb{K}(\rho)g=-\alpha\nabla\cdot(\rho\nabla g) + \beta \rho g; \\
    \text{M. SyncOT}: \quad  & \mathbb{K}(\rho)g = -\nabla\cdot(\rho\bA^{-1}\nabla g); \ (\mathrm{Corollary\ } \ref{corollary:Monge_SyncOT_Onsager}) \\
    \text{M. UnSyncOT}: \quad  & \mathbb{K}(\rho)g = -\alpha\nabla\cdot(\rho\bA^{-1}\nabla g) + \beta \rho g; 
    \ (\mathrm{Corollary\ } \ref{corollary:Monge_UnSyncOT_Onsager}) \\
    \text{K. SyncOT}: \quad  & \mathbb{G}(\rho)\dot{\rho} = \left(c_1(-\Delta_{\rho})^{-1} + c_2 \TK^*(-\Delta_{\TK\rho})^{-1}\TK \right)\dot{\rho};  \ (\mathrm{Corollary\ } \ref{corollary:K_SyncOT_dissipation}) \label{list:K_SyncOT_dissipation} \\
    \text{K. UnSyncOT}: \quad  & \Psi_\rho(\sigma) = \mathrm{eqn.\ } (\ref{eqn:K_UnSyncOT_dissipationpotential}); \ (\mathrm{Corollary\ } \ref{corollary:K_UnSyncOT_dissipation}) \label{list:K_UnSyncOT_dissipation}
  \end{align}
\end{subequations}


Similar to the equivalence between the squared 2-Wasserstein distance \eqref{eqn:pWasserstein_flat} and the BB formula \eqref{eq:DynamicalOT_flat}, the WFR formulation of the unbalanced dynamical OT \eqref{eqn:WFR} also has an equivalent static OT formulation, which is now called Hellinger-Kantorovich (HK) distance \cite{liero2018optimal}
\begin{align}\label{eqn:Hellinger_Kantorovich}
    \mathrm{HK}^2_{\alpha,\beta}(\bar{\rho}_0,\bar{\rho}_1): = \inf_{\pi\in\mathcal{M}_+(X\times X)}\ \int_{X\times X} c_{\mathrm{HK}} d\pi + \frac{4}{\beta}\KL((P_\bx)_{\sharp}\pi | \bar{\rho}_0) + \frac{4}{\beta}\KL((P_\by)_{\sharp}\pi | \bar{\rho}_1).
\end{align}
Here $P_\bx(\bx,\by)=\bx, P_\by(\bx,\by)=\by$, and the cost function is of a special form
\begin{align*}
    c_{\mathrm{HK}}(\bx,\by) = -\frac{8}{\beta}\log \cos \left(\min\left\{ \sqrt{\frac{\beta}{4\alpha}} \|\bx-\by\|, \frac{\pi}{2} \right\}\right).
\end{align*}
The HK distance \eqref{eqn:Hellinger_Kantorovich} is independently proposed in \cite{chizat2018unbalanced, liero2018optimal}, which will play a key role in the definition of the unbalanced synchronized optimal transport, see details in Section \ref{subsection:SyncOT_Kantorovich}.

Besides, there are several additional extensions of the dynamical OT. Unnormalized dynamical OT has been proposed \cite{gangbo2019unnormalized,lee2021generalized}, by taking the growth term in (\ref{eqn:feasible_WFR_flat}) as a more general form $g$ instead of $\rho g$, either in time-dependent only form \cite{gangbo2019unnormalized} or in both time- and spatial-dependent form \cite{lee2021generalized}. A further extension of the dynamical OT has been developed to incorporate more general energy functionals by replacing the kinetic energy in (\ref{eq:DynamicalOT_flat}) with a general Lagrangian $\mathcal{L}(\rho,\dot{\rho},t)$ \cite{pooladian2023neural,neklyudov2023computational}.

\section{Unbalanced Synchronized Optimal Transport Formulation} \label{section:SynOT}

In this section, we propose {\it unbalanced synchronized optimal transport} (UnSyncOT), to attain coherent dynamics across multiple spaces. Let $X^{(i)} \subseteq \mathbb{R}^{n_i}, i=1,\cdots,d$ be $d$ spaces of interest. We call $X^{(1)}$ the primary space, and $X^{(i)}, i=2,\cdots, d$ secondary spaces. UnSyncOT considers the following minimization problem
\begin{align}\label{eqn:SynOT_dSpaces}
 \mathrm{USOT}(\bar{\rho}_0,\bar{\rho}_1) = \inf_{(\rho^{(i)},\mathbf{u}^{(i)},g^{(i)})_{i=1}^d \in \mathcal{C}_{\mathrm{USOT}}} \ \sum_{i=1}^{d} c_i \int_0^1 \int_{X^{(i)}} ( \alpha \|\mathbf{u}^{(i)}\|^2 + \beta|g^{(i)}|^2) \rho^{(i)}  d\mathrm{vol}_{X^{(i)}}dt
\end{align}
with the feasible set defined as
\begin{align}\label{eqn:SynOT_dSpaces_Constraints}
  \mathcal{C}_{\mathrm{USOT}}(\bar{\rho}_0,\bar{\rho}_1)= \left\{ (\rho^{(i)},\mathbf{u}^{(i)}, g^{(i)})_{i=1}^d \middle| 
  \begin{array}{l}
  \partial_t\rho^{(i)} + \alpha \div_{X^{(i)}}(\rho^{(i)}\mathbf{u}^{(i)}) = \beta\rho^{(i)}g^{(i)} \mathrm{\ in\ } \mathcal{D}'(X^{(i)}), \\ \rho^{(i)}_{0,1} = \bar{\rho}_{0,1}^{(i)}, \ \langle \bu^{(i)}, \mathbf{n}|_{\partial X^{(i)}} \rangle_{g^{(i)}}=0, \ i=1:d,\\
  \rho^{(i)} = \mathcal{T}^{(i)}\rho^{(1)}, \ i=2,\ldots d.
  \end{array}
  \right\}
\end{align}
Here we adopt the notation in Section \ref{subsection:notation} to assume all spaces embedded Riemannian manifold $(X^{(i)}, g^{(i)})$, though it is possible that some spaces are flat spaces in which the Riemannian structure degenerates to Euclidean structure. Here $c_i$ are the weights across spaces $X^{(i)}$ satisfying $c_i\ge 0$ and $\sum_{i} c_i = 1$. In each space $X^{(i)}$, $(\alpha,\beta)$ are trade-off parameters between the mass transportation and mass growth. $\rho^{(1)}_t: (0,1)\rightarrow \mathcal{M}_+(X^{(1)})$ represents an absolutely continuous curve of measures in the primary space $X^{(1)}$ (the primary dynamics). Dynamics in the remaining spaces are induced from the primary one through maps $\mathcal{T}^{(i)}:\mathcal{M}_+(X^{(1)})\rightarrow\mathcal{M}_+(X^{(i)})$. We denote the initial and terminal measures in $X^{(i)}$ by $\rho_0^{(i)}=\bar{\rho}_0^{(i)}$ and $\rho_1^{(i)}=\bar{\rho}_1^{(i)}$.  Within the UnSyncOT formulation (\ref{eqn:SynOT_dSpaces}), these couplings impose natural marginal compatibility conditions: for each space $X^{(i)}, i=2,\cdots,d$, there exists at least one map $\mathcal{T}^{(i)}$ such that the endpoint measures are consistent, namely, $\rho_{0,1}^{(i)} = \mathcal{T}^{(i)}\rho_{0,1}^{(1)}$.

\begin{figure}[t]
  \centering
  \includegraphics[width=0.8\linewidth]{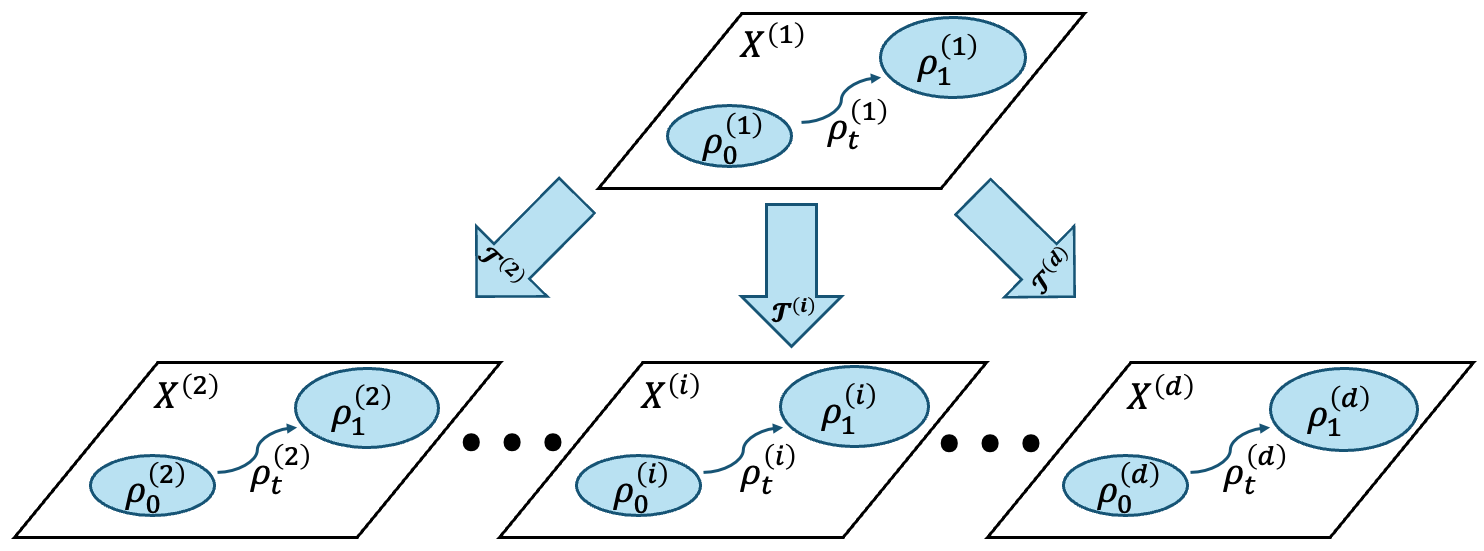}
  \caption{Schematic of unbalanced synchronized optimal transport. There is a given mapping from the primary space $X^{(1)}$ to each secondary space $X^{(i)}$, for $i=2,\cdots, d$.}
  \label{fig:myplot}
\end{figure}

Following our earlier work for the balanced SyncOT in \cite{cang2024synchronized}, we consider two commonly encountered forms of $\mathcal{T}^{(i)}$:
\vspace{-.1in}
\begin{itemize}
\item {\it Monge form}: $\mathcal{T}^{(i)}_{\mathrm{M}} = \mathbf{T}_{\sharp}^{(i)}$ for a known map $\mathbf{T}^{(i)}:X^{(1)}\rightarrow X^{(i)}$. For instance, the map can be obtained by a trained neural network, then one can use the pushforward operator to define $\rho^{(i)} = \mathbf{T}^{(i)}_\sharp\rho^{(1)}$.
\vspace{-.1in}
\item {\it Kantorovich form}: $\mathcal{T}^{(i)}_{\mathrm{K}} = \int_{X^{(1)}}\text{d}\pi^{(i)}$ for a known joint distribution (a Markov kernel) $\pi^{(i)}\in\mathcal{M}_+(X^{(1)}\times X^{(i)})$. For instance, $\pi$ can be obtained by solving a static Kantorovich OT on two datasets representing $X^{(1)}$ and $X^{(i)}$ and normalizing the resulting transport matrix row-wisely, then one can define $\rho_t^{(i)}(\bx^{(i)})=\int_{X^{(1)}}\rho_t^{(1)}(\bx^{(1)}) \pi^{(i)}(\bx^{(1)},\bx^{(i)}) \mathrm{d}\bx^{(1)} $.
\end{itemize}

Given the maps $\mathcal{T}^{(i)}$, the primary velocity-growth pair $(\bu^{(1)},g^{(1)})$ and the secondary pair $(\bu^{(i)},g^{(i)})$ are usually implicitly linked through $\mathcal{T}^{(i)}$. In the rest of the paper, we will explore the links between them, which will lead to simpler formulations for UnSyncOT for simple numerical implementation.

For brevity, we focus on the case of $d = 2$, and denote by $X\subseteq\mathbb{R}^m$ and $Y\subseteq\mathbb{R}^n$ the primary and secondary spaces, and $(\rho, \mathbf{u}, g)$ and $(\xi, \mathbf{v},h)$ the corresponding density-velocity-growth triples. The map between the primary and secondary spaces is denoted by $\mathcal{T}$ (either $\mathcal{T}_{\mathrm{M}}$ or $\mathcal{T}_{\mathrm{K}}$). The extension of the theories and numerical methods for UnSyncOT to the general $d$-space case is straightforward. In what follows, we will focus on the two-space UnSyncOT in both Monge and Kantorovich forms.

\subsection{Unbalanced Synchronized Optimal Transport in Monge Form}\label{subsection:Monge_UnSyncOT}

In the Monge UnSyncOT, we assume $X\subseteq \bbR^m$ to be open and convex. The $C^1$-embedding $\bT: X\rightarrow Y = \bT(X) \subseteq \bbR^n$ define the embedded Riemannian submanifold $(Y,g)$ of $\bbR^n$.

Now we consider the Monge UnSyncOT as follows:
\begin{align}\label{eqn:Monge_UnSyncOT}
\mathrm{USOT}_{\mathrm{M}}(\bar{\rho}_0,\bar{\rho}_1) = \inf_{(\rho,\mathbf{u},g;\xi,\bv,h)\in \mathcal{C}^{\mathrm{M}}_{\mathrm{USOT}} } A_{\mathrm{USOT}}^{\mathrm{M}} = \int_0^1 \bigg[ & c_1 \int_{X} \Big( \alpha\|\mathbf{u}_t\|^2 + \beta g_t^2 \Big) \rho_t d\bx   \nonumber \\
& + c_2 \int_{Y}  \Big( \alpha\|\mathbf{v}_t\|^2 +  \beta h_t^2 \Big) \xi_t d\mathrm{vol}_Y   \bigg] dt,
\end{align}
with the feasible set $\mathcal{C}^{\mathrm{M}}_{\mathrm{USOT}}(\bar{\rho}_0,\bar{\rho}_1)$ defined as
\begin{align}\label{eqn:Monge_UnSyncOT_Constraints}
  \mathcal{C}^{\mathrm{M}}_{\mathrm{USOT}} = \left\{ (\rho,\mathbf{u},g;\xi,\bv,h) \middle| 
  \begin{array}{l}
  \partial_t\rho + \alpha\nabla\cdot(\rho\mathbf{u}) = \beta\rho g \ \mathrm{in}\ \mathcal{D}'(X), 
  \rho_{0,1} = \bar{\rho}_{0,1},  \bu\cdot\mathbf{n}|_{\partial X} = 0,\\
  \partial_t\xi + \alpha\div_Y(\xi \mathbf{v}) = \beta\xi h\ \mathrm{in}\ \mathcal{D}'(Y), 
  \xi_t = \mathbf{T}_{\sharp}\rho_t,
    \langle \bv, \mathbf{n}|_{\partial Y} \rangle_g = 0, \\
  h_t(\by) = g_t(\mathbf{T}^{-1}(\mathbf{y})).
  \end{array}
  \right\}
\end{align}
The primary dynamics $\rho_t(\mathbf{x})$ induces a secondary dynamics $\xi_t(\by)$ through map $\mathbf{T}_{\sharp}:\mathcal{M}_+(X)\rightarrow\mathcal{M}_+(Y)$, with the assumption of the marginal compatibility condition that $\mathbf{T}$ is a $C^1$-embedding such that $\xi_{0,1} = \bT_\sharp \rho_{0,1}$.
The $C^1$-embedding $\bT$ guarantees to map the boundary of $X$ to the boundary of $Y$ such that the boundary conditions of $\bu$ and $\bv$ are both satisfied. Besides, growth terms $g_t$ in $X$ and $h_t = g_t\circ\mathbf{T}^{-1}$ in $Y$ are carefully taken in order to maintain the mass consistency in the two spaces,
\begin{align*}
    \frac{d}{dt} \int_X \rho_t(\bx) d\bx = \beta\int_X g_t(\bx) \rho_t(\bx) d\bx = \beta \int_Y g_t\left(\mathbf{T}^{-1}(\by)\right) \xi_t(\by)d\mathrm{vol}_Y = \frac{d}{dt}\int_Y \xi_t(\by) d\mathrm{vol}_Y,
\end{align*}
where the second equation above is due to the condition $\xi_t d\mathrm{vol}_Y = \mathbf{T}_{\sharp}(\rho_t d\bx)$.

The UnSyncOT in Monge form (\ref{eqn:Monge_UnSyncOT}) can be recast into a single-space unbalanced dynamical OT form where the ground metric depends on $\mathbf{T}$. We summarize it in the following theorem.

\begin{theorem}
The Monge unbalanced synchronized optimal transport (\ref{eqn:Monge_UnSyncOT}) is equivalent to the following single-space unbalanced dynamical OT:
\begin{align}\label{eqn:Monge_UnSyncOT_SingleSpace}
  \mathrm{USOT}_{\mathrm{M}}(\bar{\rho}_0,\bar{\rho}_1) = \inf_{(\rho,\mathbf{u},g)\in\mathcal{C}_{\mathrm{WFR}_{\alpha,\beta}}} \ \int_0^1 \int_{X} \Big[ \alpha\left\|\mathbf{u}_t\right\|^2_{\mathbf{A}} + \beta g_t^2 \Big] \rho_t d\bx dt
\end{align}
in which $\mathbf{A} = c_1 \mathbf{I}+c_2(\nabla \mathbf{T})^{\mathrm{T}}(\nabla \mathbf{T})$, $\nabla\mathbf{T}$ is the Jacobian of $\mathbf{T}$, $\|\mathbf{u}\|_{\mathbf{A}}^2 := \mathbf{u}^{\mathrm{T}} \mathbf{A}\mathbf{u}$, and $\mathcal{C}_{\mathrm{WFR}_{\alpha,\beta}}$ is defined in \eqref{eqn:feasible_WFR_flat}.
\end{theorem}

\begin{proof}

 To prove the equivalence, the key step is to calculate the total cost in the secondary space. By the relation $\xi_t = \mathbf{T}_{\sharp}\rho_t$, the secondary space cost becomes
    \begin{align}\label{eqn:Tpush_eqn}
       \alpha \int_{Y}  \|\bv_t\|^2 \xi_t d\mathrm{vol}_Y + \beta \int_{Y}  |g_t\circ\mathbf{T}^{-1}|^2 \xi_t d\mathrm{vol}_Y 
       = \alpha\int_{X}  \|\bv_t\circ\mathbf{T}\|^2 \rho_t d\bx + \beta \int_X g_t^2 \rho_t d\bx.
    \end{align}
    Now we replace $\bv$ in terms of $\bu$. Note that $\bv$ satisfies the continuity equation with source term 
    in the distributional sense, hence for any test function $\phi\in C_{c}^{\infty}(Y)$, it holds that
    \[
    \int_{Y} \frac{\partial}{\partial t} \xi_t(\mathbf{y}) \phi(\mathbf{y}) d\mathrm{vol}_Y
    = - \alpha\int_{Y} \div_Y\left(\xi_t(\mathbf{y})\bv_t(\mathbf{y})\right)\phi(\mathbf{y}) d\mathrm{vol}_Y + \beta\int_{Y} \xi_t(\mathbf{y})g_t(\mathbf{T}^{-1}(\by))\phi(\mathbf{y}) d\mathrm{vol}_Y.
    \]
    From the above equation, we have
    \begin{align*}
      \text{LHS} &= \frac{d}{dt} \int_{Y} \xi_t(\mathbf{y}) \phi(\mathbf{y}) d\mathrm{vol}_Y  = \frac{d}{dt} \int_{X} \rho_t(\mathbf{x}) \phi(\bT(\mathbf{x})) d\mathbf{x} = \int_{X}  \partial_t\rho_t(\mathbf{x}) \phi(\bT(\mathbf{x})) d\mathbf{x} \\
      &= -\alpha\int_{X} \nabla\cdot \left(\rho_t(\mathbf{x})\bu_t(\mathbf{x})\right)\phi(\bT(\mathbf{x})) d\mathbf{x} 
      + \beta\int_X \rho_t(\bx)g_t(\bx)\phi(\mathbf{T}(\bx)) d\bx
      \\
      & = \alpha\int_{X}\nabla\Big(\phi(\bT(\mathbf{x}))\Big)\cdot \left(\rho_t(\mathbf{x})\bu_t(\mathbf{x})\right) d\mathbf{x} + \beta\int_X \rho_t(\bx)g_t(\bx)\phi(\mathbf{T}(\bx))d\bx \\
      & = \alpha\int_{X} \Big[(\nabla_Y\phi)(\bT(\mathbf{x}))\Big]^{\mathrm{T}} (\nabla \bT(\bx)) \bu_t(\mathbf{x})\rho_t(\mathbf{x})d\bx + \beta \int_X \rho_t(\bx)g_t(\bx)\phi(\mathbf{T}(\bx)) d\bx, \\
    \intertext{\text{and}}
      \text{RHS} & = \alpha\int_{Y} \nabla_{Y}\phi(\mathbf{y})\cdot \left(\xi_t(\mathbf{y})\bv_t(\mathbf{y})\right) d\mathrm{vol}_Y + \beta\int_{Y} \xi_t(\mathbf{y})g_t(\mathbf{T}^{-1}(\mathbf{y}))\phi(\mathbf{y}) d\mathrm{vol}_Y \\
      & = \alpha \int_{X} \Big[(\nabla_{Y}\phi)(\bT(\mathbf{x}))\Big]^{\mathrm{T}} \bv_t(\bT(\mathbf{x})) \rho_t(\mathbf{x}) d\bx + \beta\int_X \rho_t(\bx)g_t(\bx)\phi(\mathbf{T}(\bx)) d\bx.
    \end{align*}
    Since $\phi$ is arbitrary, it holds that
    \begin{align}\label{eqn:v_relation}
        \bv_t(\bT(\mathbf{x})) = (\nabla \bT(\bx)) \bu_t(\mathbf{x}) .
    \end{align}
    Finally, inserting the relation (\ref{eqn:v_relation}) between $\bv$ and $\bu$ into (\ref{eqn:Tpush_eqn}), we have
    \begin{align*}
      \alpha\int_{Y} \xi_t \|\bv_t\|^2 d\mathrm{vol}_Y + \beta \int_{Y} \xi_t |g_t\circ\mathbf{T}^{-1}|^2 d\mathrm{vol}_Y
      = \alpha \int_{X} \rho_t \big\| \bu_t \big\|^2_{(\nabla \bT)^\mathrm{T}(\nabla \bT)} d\mathbf{x} + \beta \int_X \rho_t g_t^2 d\bx.
    \end{align*}
    The proof is therefore completed by combining the above secondary space cost (in terms of $\rho,\bu,g$) and the primary space cost.
\end{proof}

We can further provide the dual formulation in term of $\Phi_t(\bx)$ for Monge UnSyncOT (\ref{eqn:Monge_UnSyncOT_SingleSpace}).
\begin{corollary}\label{corollary:Monge_UnSyncOT_dual}
    The minimizer $(\rho_t(\bx),\bu_t(\bx),g_t(\bx))$ of the single-space Monge UnSyncOT problem (\ref{eqn:Monge_UnSyncOT_SingleSpace}) satisfies
    \begin{align}\label{eqn:Monge_UnSyncOT_dual_u_g}
        \bu_t(\bx) = \bA^{-1}\nabla\Phi_t(\bx), \quad g_t(\bx) = \Phi_t(\bx),
    \end{align}
    in which the dual variable $\Phi_t(\bx)$ is determined by
    \begin{align}\label{eqn:Monge_UnSyncOT_dualform}
        \begin{cases}
            \partial_t\rho_t + \alpha\nabla\cdot(\rho_t\bA^{-1}\nabla\Phi_t) = \beta\rho_t\Phi_t,\\
            \partial_t\Phi_t + \frac{\alpha}{2}\|\nabla\Phi_t\|_{\bA^{-1}}^2 + \frac{\beta}{2}\Phi_t^2 \le 0, \\
            \rho_0 = \bar{\rho}_0, \rho_1 = \bar{\rho}_1.
        \end{cases}
    \end{align}
    If $\rho_t > 0$, then
    \begin{align}
        \partial_t\Phi_t + \frac{\alpha}{2}\|\nabla\Phi_t\|_{\bA^{-1}}^2 + \frac{\beta}{2}\Phi_t^2 = 0. 
    \end{align}
\end{corollary}
\begin{proof}
    Introducing auxiliary variables $\bfm_t = \rho_t\bu_t$ and $H_t = \rho_t g_t$ and the lower semicontinuous convex function $J(\rho_t,\bfm_t,H_t)$
    \begin{align}\label{eqn:J_linearlize}
        J(\rho_t,\bfm_t,H_t)=
        \begin{cases}
            \alpha\rho^{-1}_t\|\bfm_t\|_\bA^2 + \beta \rho_t^{-1}H_t^2, \quad & \mathrm{if}\ \rho_t>0, \\
            0, \quad & \mathrm{if}\ (\rho_t,\bfm_t,H_t)=(0,\mathbf{0},0),\\
            +\infty, \quad & \mathrm{if}\ \mathrm{otherwise}.
        \end{cases}
    \end{align}
    Then the Monge UnSyncOT (\ref{eqn:Monge_UnSyncOT_SingleSpace}) can be recast into
    \begin{align}
        \min_{(\rho,\bfm,H)} \left\{\int_0^1\int_X J(\rho,\bfm,H) d\bx dt \ \middle| \ \partial_t\rho + \alpha\nabla\cdot\bfm = \beta H, \rho_{0,1} = \bar{\rho}_{0,1}, \bfm_t\cdot\bn|_{\partial X} =0\right\}.
    \end{align}
    Now we introducing the Lagrange multiplier $\Phi_t$ of the constraints to the Lagrangian
    \begin{align*}
        \mathcal{L}(\rho,\bfm,H;\Phi) = \int_0^1\int_X \frac{\alpha}{2}\rho^{-1}\|\bfm\|_\bA^2 + \frac{\beta}{2}\rho^{-1}H^2 + \Phi(\partial_t\rho + \alpha \nabla\cdot\bfm - \beta H) d\bx dt.
    \end{align*}
    Taking the first-order optimality condition, we have
    \begin{align*}
        \begin{cases}
            -\frac{\alpha}{2}\rho^{-2} \|\bfm\|_\bA^2 - \frac{\beta}{2}\rho^{-2}H^2-\partial_t\Phi \ge0, \\
            \alpha\rho^{-1}\bA\bfm = \alpha\nabla\Phi, \\
            \beta\rho^{-1}H = \beta\Phi.
        \end{cases}
    \end{align*}
    The last two equations in the above system implies (\ref{eqn:Monge_UnSyncOT_dual_u_g}), namely, the velocity-reaction pair $(\bu_t,g_t)$ at optima is related by $\bu_t = \bA^{-1}\nabla g_t$.
    
    Inserting (\ref{eqn:Monge_UnSyncOT_dual_u_g}) into the first equation in the above system leads
    \begin{align*}
        \partial_t\Phi + \frac{\alpha}{2}\|\nabla\Phi\|_{\bA^{-1}}^2 + \frac{\beta}{2}\Phi^2 \le 0.
    \end{align*}
    Additionally, inserting (\ref{eqn:Monge_UnSyncOT_dual_u_g}) into the continuity equation yields
    \begin{align*}
        \partial_t \rho + \alpha\nabla\cdot(\rho \bA^{-1}\nabla\Phi) = \beta\rho\Phi.
    \end{align*}
    Therefore we prove the system (\ref{eqn:Monge_UnSyncOT_dualform}). If $\rho_t > 0$, then the inequality in the $\Phi$-equation becomes equality.
\end{proof}

The dual form in Corollary \ref{corollary:Monge_UnSyncOT_dual} directly implies the reformulation of the single-space Monge UnSyncOT in terms of the dissipation distance.
\begin{corollary}\label{corollary:Monge_UnSyncOT_Onsager}
    In the framework of dissipation distance (\ref{eqn:dissipation_potential}), the single-space Monge unbalanced synchronized optimal transport (\ref{eqn:Monge_UnSyncOT_SingleSpace}) can be recast into the form (\ref{eqn:dissipation_distance_K})
    with the Onsager operator $\mathbb{K}$ given as  $\mathbb{K}(\rho_t)g_t = -\alpha\nabla\cdot(\rho_t\bA^{-1}\nabla g_t) + \beta\rho_t g_t$. This is one of the cases in List \ref{eqn:K_operators}.
\end{corollary}

\subsection{Unbalanced Synchronized Optimal Transport in Kantorovich Form}\label{subsection:SyncOT_Kantorovich}

In the Kantorovich UnSyncOT, we assume $X\subseteq \bbR^m$ and $Y\subseteq \bbR^n$ to be open and convex. The maps $\mathcal{T}_\mathrm{K}$ and $\mathcal{T}_\mathrm{K}^*$ are defined as
\begin{align}
    \mathcal{T}_{\mathrm{K}}: \mathcal{M}_+(X) \rightarrow \mathcal{M}_+(Y), \quad (\mathcal{T}_{\mathrm{K}}\rho)(\by) = \int_X \rho(\bx) \pi(\bx,\by) d\bx, \label{eqn:TK}\\
    \mathcal{T}_{\mathrm{K}}^*: \mathcal{M}_+(Y) \rightarrow \mathcal{M}_+(X), \quad (\mathcal{T}_{\mathrm{K}}^*\xi)(\bx) = \int_Y \xi(\by) \pi(\bx,\by) d\by. \label{eqn:TK_star}
\end{align}
Here $\pi\in\mathcal{M}_+(X \times Y)$ is a given joint measure between $X$ and $Y$ with compact support. We require that $\pi(\bx,\by)$ is a Markov kernel,
\begin{align}\label{eqn:Markov_kernel}
\int_Y \pi(\bx,\by) d\by = 1, \quad \forall \bx \in X,
\end{align}
such that later when defining the relation between $\rho_t$ and $\xi_t$ in equation \reff{eqn:KantorovichUnSynOT_constraints}, the mass is conserved across the spaces: $\int_X \rho(\mathrm{d}\mathbf{x})= \int_Y \xi(\mathrm{d}\mathbf{y})$.

We consider the Kantorovich UnSyncOT as follows:
\begin{align}\label{eqn:KantorovichUnSynOT}
\mathrm{USOT}_{\mathrm{K}}(\bar{\rho}_0,\bar{\rho}_1) := \inf_{(\rho,\mathbf{u},g;\xi,\bv,h)\in\mathcal{C}^{\mathrm{K}}_{\mathrm{USOT}} } A_{\mathrm{USOT}}^{\mathrm{K}} = \int_0^1 \bigg(& c_1 \int_{X} \Big[ \alpha\|\mathbf{u}_t\|^2 
+ \beta g_t^2 \Big] \rho_t d\bx   \nonumber\\
+ & c_2 \int_{Y}  \Big[ \alpha \|\mathbf{v}_t\|^2 +  \beta h_t^2 \Big] \xi_t d\by  \bigg)dt,
\end{align}
with the feasible set $\mathcal{C}^{\mathrm{K}}_{\mathrm{USOT}}(\bar{\rho}_0,\bar{\rho}_1)$ defined as
\begin{align}\label{eqn:KantorovichUnSynOT_constraints}
  \mathcal{C}^{\mathrm{K}}_{\mathrm{USOT}} = \left\{ (\rho,\mathbf{u},g;\xi,\bv,h) \middle| 
  \begin{array}{l}
  \partial_t\rho + \alpha\nabla\cdot(\rho\mathbf{u}) = \beta \rho g \mathrm{\ in\ } \mathcal{D}'(X), \
  \rho_{0,1} = \bar{\rho}_{0,1}, \  \bu\cdot\mathbf{n}|_{\partial X} = 0,\\
  \partial_t\xi + \alpha\nabla\cdot(\xi \mathbf{v}) = \beta\xi h \mathrm{\ in\ } \mathcal{D}'(Y),\ \bv\cdot\mathbf{n}|_{\partial Y} = 0,\\
  \xi_t(\mathbf{y}) = (\TK\rho_t)(\by) = \int_{X} \rho_t(\bx)\pi(\bx,\by) d\bx, \\
  \xi_t(\by)h_t(\by) = (\TK(\rho_tg_t))(\by).
  \end{array}
  \right\}
\end{align}

\begin{theorem}\label{theorem:Kantorovich_UnSyncOT}
The Kantorovich unbalanced synchronized optimal transport (\ref{eqn:KantorovichUnSynOT}) is equivalent to the following single-space unbalanced dynamical OT:
\begin{align}\label{eqn:KantorovichUnSynOT_onespace}
\mathrm{USOT}_{\mathrm{K}}(\bar{\rho}_0,\bar{\rho}_1) = \inf_{(\rho,\mathbf{u},g)\in\mathcal{C}_{\mathrm{WFR}_{\alpha,\beta}}} \ \int_0^1 \bigg[ & c_1 \int_{X} \Big( \alpha\|\mathbf{u}_t\|^2 + \beta g_t^2 \Big) \rho_t d\bx   \nonumber \\
 + & c_2 \int_X  \Big(\alpha \TK^*(\|\nabla\Phi_t\|^2) + \beta \TK^* (h_t^2) \Big) \rho_t d\bx   \bigg] dt,
\end{align}
in which $\mathcal{C}_{\mathrm{WFR}_{\alpha,\beta}}$ is defined in \eqref{eqn:feasible_WFR_flat}, and $\Phi_t$ is defined as
\begin{align}\label{eqn:Phi}
    \Phi_t(\by) = (-\Delta_{\xi_t})^{-1}\left( \int_X -\nabla\cdot (\rho_t(\bx)\mathbf{u}_t(\bx))\pi(\bx,\by)d\bx \right) = (-\Delta_{\xi_t})^{-1}\Big(\TK(-\nabla\cdot(\rho_t\bu_t))(\by)\Big),
\end{align}
and $h_t$ is defined in the $\mathcal{C}^{\mathrm{K}}_{\mathrm{USOT}}$ condition (\ref{eqn:KantorovichUnSynOT_constraints}). Besides, we have the following bound estimate for Kantorovich UnSyncOT,
\begin{align}\label{eqn:Kantorovich_UnSyncOT_bound}
    c_1 \mathrm{WFR}_{\alpha,\beta,X}^2(\bar{\rho}_0,\bar{\rho}_1) + c_2 \mathrm{WFR}_{\alpha,\beta,Y}^2(\TK\bar{\rho}_0,\TK\bar{\rho}_1) \le \mathrm{USOT}_{\mathrm{K}}(\bar{\rho}_0,\bar{\rho}_1) \le c \mathrm{WFR}_{\alpha,\beta,X}^2(\bar{\rho}_0,\bar{\rho}_1),
\end{align}
in which $c = \max\{ c_1+c_2\Lambda^2,1 \}$ with $\Lambda = \mathrm{esssup}_{t\in(0,1)}\|\TK\|_{T_{\rho_t}\rightarrow T_{\TK\rho_t}}$.
\end{theorem}

\begin{proof}
    Firstly we prove the single-space reformulation of the Kantorovich UnSyncOT (\ref{eqn:KantorovichUnSynOT_onespace}).
    By the relation between $\xi_t = \TK \rho_t$, the total cost in the secondary space becomes
    \begin{align}\label{eqn:Tpush_eqn_Kantorovich}
       \alpha\int_{Y} \|\bv_t\|^2 (\TK\rho_t) d\by + \beta\int_{Y} h_t^2 (\TK\rho_t) d\by 
       = \ &\alpha\int_{X}  \TK^*(\|\mathbf{v}_t\|^2) \rho_t d\bx + \beta  \int_X \TK^*(h_t^2)  \rho_t d\bx.
    \end{align}
    Note that $\bv_t$ satisfies the continuity equation with source term 
    in the distributional sense, hence for any test function $\phi\in C_{c}^{\infty}(Y)$, it holds that
    \[
    \frac{d}{dt} \int_{Y}  \xi_t(\mathbf{y}) \phi(\mathbf{y})d\by
    =  \alpha\int_{Y} - \nabla\cdot\left(\xi_t(\mathbf{y})\bv_t(\mathbf{y})\right)\phi(\mathbf{y}) d\by + \beta \int_{Y} \xi_t(\mathbf{y})h_t(\by)\phi(\mathbf{y}) d\by.
    \]
    Since $\phi\in C_{c}^{\infty}(Y)$ and $\pi$ is of compact support, we have $\TK^*\phi \in C_c^{\infty}(X)$. It follows that
    \begin{align*}
    &\mathrm{LHS} =  \frac{d}{dt} \langle \TK^*\phi, \rho_t \rangle_{L^2(X)} = 
    \alpha \big\langle \TK^*\phi, -\nabla\cdot(\rho_t\bu_t) \big\rangle_{L^2(X)} +
    \beta \big\langle \TK^*\phi, \rho_t g_t \big\rangle_{L^2(X)}\\
     &\hspace{1.6in}= \alpha \big\langle \phi, \TK(-\nabla\cdot(\rho_t\bu_t)) \big\rangle_{L^2(Y)} +
    \beta \big\langle \phi, \TK(\rho_t g_t) \big\rangle_{L^2(Y)}\\
    &\mathrm{RHS}= \alpha \big\langle \phi, -\nabla\cdot(\xi_t\bv_t) \big\rangle_{L^2(Y)} + \beta \big\langle \phi,  \xi_t h_t \big\rangle_{L^2(Y)}.
    \end{align*}
    Feasibility condition (\ref{eqn:KantorovichUnSynOT_constraints}) implies that 
    $
    \xi_t h_t = \TK(\rho_t g_t),
    $
    therefore by comparing the two sides, $\mathbf{v}_t$ satisfies
    \begin{align}\label{eqn:v_Kantorovich}
     -\nabla\cdot(\xi_t(\by)\bv_t(\by)) = \TK(-\nabla\cdot(\rho_t \bu_t))(\by) = \int_{X}-\nabla\cdot(\rho_t(\bx)\bu_t(\bx)) \pi(\bx,\by) d\bx.
    \end{align}
    Using Otto's argument \cite{figalli2021invitation}, among all velocity fields $\bv_t$ satisfying (\ref{eqn:v_Kantorovich}),  there is a unique optimal one with smallest $L^2(\xi_t;Y)$-norm $\int_{Y} \|\bv_t(\by)\|^2 \xi_t(d\mathbf{y})$:  $\mathbf{v}_t = \nabla\Phi_t$ in which $\Phi_t(\by)$ satisfies
    \begin{align}\label{eqn:UnSyncOT_2nd_v}
        \begin{cases}
            -\nabla\cdot(\xi_t(\by)\nabla\Phi_t(\by)) = \TK(-\nabla\cdot(\rho_t \bu_t))(\by), \quad &\mathrm{in}\ Y, \\
            \frac{\partial\Phi_t}{\partial \mathbf{n}} = 0, \quad &\mathrm{on} \ \partial Y.
        \end{cases}
    \end{align}
 Denoting $(-\Delta_{\xi_t})\Phi_t(\by): = -\nabla\cdot(\xi_t(\by)\nabla\Phi_t(\by))$,  we obtain (\ref{eqn:Phi}).

 Now we prove the bound estimate (\ref{eqn:Kantorovich_UnSyncOT_bound}). The lower bound is straightforward. Indeed, taking any feasible point $(\rho,\bu,g;\xi,\bv,h) \in \mathcal{C}^{\mathrm{K}}_{\mathrm{USOT}}(\bar{\rho}_0,\bar{\rho}_1)$, we have that $(\rho,\bu,g)\in \mathcal{C}_{\mathrm{WFR}_{\alpha,\beta,X}}(\bar{\rho}_0,\bar{\rho}_1)$ and $(\xi,\bv,h)\in \mathcal{C}_{\mathrm{WFR}_{\alpha,\beta,Y}}(\TK\bar{\rho}_0,\TK\bar{\rho}_1)$, then it follows that
 \begin{align*}
     c_1 \mathrm{WFR}_{\alpha,\beta,X}^2(\bar{\rho}_0,\bar{\rho}_1) + c_2 \mathrm{WFR}_{\alpha,\beta,Y}^2(\TK\bar{\rho}_0,\TK\bar{\rho}_1) \le 
     A_{\mathrm{USOT}}^{\mathrm{K}}[\rho,\bu,g;\xi,\bv,h],
 \end{align*}
 taking infimum over $\mathcal{C}^{\mathrm{K}}_{\mathrm{USOT}}$ on the right hand side yields the lower bound estimate.

 For the upper bound, we consider the transport and reaction parts separately. We again consider a feasible point $(\rho,\bu,g;\xi,\bv,h) \in \mathcal{C}^{\mathrm{K}}_{\mathrm{USOT}}$, and write $f_t(\bx) = -\nabla\cdot(\rho_t \bu_t)\in T_{\rho_t}$. Here $T_{\rho_t}$ represents the tangent space of the Wasserstein space $(\mathcal{P}_2(X), W_2)$ at $\rho_t$. In the secondary space $Y$, the unique minimal-energy velocity is $\bv_t = \nabla\Phi_t$ with $-\Delta_{\xi_t}\Phi_t = \TK f_t$ as given from (\ref{eqn:v_Kantorovich}), and $\TK f_t \in T_{\xi_t}$. Then the secondary transport cost in (\ref{eqn:KantorovichUnSynOT}) or (\ref{eqn:KantorovichUnSynOT_onespace}) is
 \begin{align*}
     \int_Y \xi_t\|\bv_t\|^2 d\by = \int_Y \xi_t \|\nabla\Phi_t\|^2 d\by = \|\TK f_t\|_{-1,\xi_t}^2
 \end{align*}
 in which $\|\TK f_t\|_{-1,\xi_t}^2$ is the Wasserstein norm defined in (\ref{eqn:WassersteinNorm}). For fixed $t\in(0,1)$, we define the norm of $\TK: T_{\rho_t} \rightarrow T_{\xi_t}$ as
 \begin{align*}
     \|\TK\|:=\|\TK\|_{T_{\rho_t}\rightarrow T_{\TK\rho_t}}: = \sup_{0\neq h \in T_{\rho_t}} \frac{\|\TK h\|_{-1,\TK\rho_t}}{\|h\|_{-1,\rho_t}}, \quad \Lambda = \mathrm{esssup}_{t\in(0,1)} \|\TK\|_{T_{\rho_t}\rightarrow T_{\TK\rho_t}},
 \end{align*}
 then it follows that
 \begin{align}\label{eqn:Transport_part}
     \int_Y \xi_t\|\bv_t\|^2 d\by = \|\TK f_t\|_{-1,\xi_t}^2 \le \|\TK\|^2 \|f_t\|^2_{-1,\rho_t} \le \|\TK\|^2 \int_X \rho_t \|\bu_t\|^2 d\bx \le \Lambda^2 \int_X \rho_t \|\bu_t\|^2 d\bx,
 \end{align}
 in which the second inequality is due to the definition of the Wasserstein norm (\ref{eqn:WassersteinNorm}).

 For the reaction part, again we let $\xi_t = \TK \rho_t$. According to the feasibility condition (\ref{eqn:KantorovichUnSynOT_constraints}), the reaction cost in secondary space $Y$ given in (\ref{eqn:KantorovichUnSynOT}) or  (\ref{eqn:KantorovichUnSynOT_onespace}) can be rewritten as:
 \begin{align*}
     \int_Y \xi_t h_t^2 d\by = \int_{Y}   \frac{(\mathcal{T}_{\mathrm{K}}(\rho_t g_t) )^2}{\mathcal{T}_{\mathrm{K}}\rho_t} d\by.
 \end{align*}
 Given a fixed $t \in (0,1)$, for all $\by$ such that $\xi_t(\by)>0$, we define the conditional probability on $X$
    \begin{align*}
        Q(d\bx|\by):=\frac{\rho_t(\bx)\pi(\bx,\by)}{\xi_t(\by)}d\bx.
    \end{align*}
    Then for any integrable $f:X\rightarrow\mathbb{R}$,
    \begin{align*}
        \mathbb{E}[f|\by] 
        = \int_X f(\bx)Q(d\bx|\by)
        = \frac{\TK(\rho_t f)(\by)}{\TK\rho_t(\by)}.
    \end{align*}
    In particular
    \begin{align*}
        \mathbb{E}[g_t|\by] 
        = \frac{\TK(\rho_t g_t)}{\TK \rho_t} = h_t(\by), \quad
        \mathbb{E}[g_t^2|\by] 
        = \frac{\TK(\rho_t g_t^2)}{\TK \rho_t},        
    \end{align*}
    namely, $h_t(\by)$ is the conditional mean of $g_t$ given $\by$.
    Define the conditional variance at $\by$,
    \begin{align*}
        \mathrm{Var_{\bx|\by}}(g_t):=
        \mathbb{E}[g_t^2|\by] - (\mathbb{E}[g_t|\by])^2,
    \end{align*}
    it follows that
    \begin{align}\label{eqn:var}
        \frac{(\TK(\rho_t g_t)(\by))^2}{\TK\rho_t(\by)} 
        = \TK(\rho_t g_t^2)(\by) - \TK \rho_t(\by) \mathrm{Var}_{\bx|\by}(g_t). 
    \end{align}
    By convention, on set $\{ \xi_t=0 \}$ we define $\mathrm{Var}_{\bx|\by}(g_t)=0$ and $\TK(\rho_t g_t) = 0$, then the equation (\ref{eqn:var}) still holds.
    Taking integral on two sides of (\ref{eqn:var}), and using $\int_Y \pi(\bx,\by)d\by = 1$, we have
    \begin{align}
        \int_Y \xi_t h_t^2 d\by = \int_Y \frac{(\TK(\rho_t g_t)(\by))^2}{\TK\rho_t(\by)} d\by
        &= \int_Y \TK(\rho_t g_t^2)(\by) d\by - \int_Y \TK \rho_t(\by) \mathrm{Var}_{\bx|\by}(g_t) d\by \nonumber\\
        &= \int_X \rho_t g_t^2 d\bx - \int_Y \TK \rho_t(\by) \mathrm{Var}_{\bx|\by}(g_t) d\by \le \int_X \rho_t g_t^2 d\bx, \label{eqn:Reaction_part}
    \end{align}
    which implies that the reaction cost in the secondary space is no greater than that in the primary space. Combining the domination of the transport cost in the primary space over that in the secondary space in equation (\ref{eqn:Transport_part}) and the domination of the reaction cost in the primary space over that in the secondary space in equation (\ref{eqn:Reaction_part}), and taking the necessary infimum, we get the upper bound estimate (\ref{eqn:Kantorovich_UnSyncOT_bound}).
\end{proof}

\begin{corollary}\label{corollary:K_UnSyncOT_dissipation}
    The Kantorovich unbalanced synchronized optimal transport (\ref{eqn:KantorovichUnSynOT}) can be rewritten in the form of dissipation potential (\ref{eqn:dissipation_potential})
    with $\Psi(\rho,\sigma) = \Psi_{\rho}(\sigma)$ being a quadratic form 
    \begin{align}\label{eqn:K_UnSyncOT_dissipationpotential}
        \Psi_{\rho}(\sigma) = \min_g &\ c_1 \left(  \alpha \left\| \frac{\sigma-\beta\rho g}{\alpha}\right\|_{-1,\rho}^2 + \beta \int_X \rho g^2 d\bx \right)  \notag\\
        &+ c_2 \left( \alpha \left\| \TK\left( \frac{\sigma-\beta\rho g}{\alpha} \right) \right\|_{-1,{\TK}\rho}^2 + \beta \int_Y \frac{(\TK(\rho g))^2}{\TK\rho} d\by \right).
    \end{align}
\end{corollary}
\begin{proof}
    Kantorovich UnSyncOT is a minimization for a time-space double integral. For fixed $\rho=\rho_t$, we can consider the inner minimization 
    \begin{align}\label{eqn:K_UnSyncOT_inner_min}
        \min_{(\bu,g,\bv)} c_1 \int_X \alpha\rho\|\bu\|^2 + \beta\rho g^2 d\bx  + c_2 \int_Y \alpha \xi \|\bv\|^2 + \beta \frac{\TK(\rho g)^2}{\xi} d\by,
    \end{align}
    subject to 
    \begin{align*}
        -\alpha\nabla\cdot(\rho\bu) + \beta \rho g = \sigma, \ \bu\cdot\bn|_{\partial X} = 0; \quad
        -\nabla\cdot(\xi \bv) = \TK(-\nabla\cdot(\rho\bu)), \ \bv\cdot\bn|_{\partial Y} = 0
    \end{align*}
    Let $\bu^* = \nabla\phi$ satisfy $-\nabla\cdot(\rho\nabla\phi) = (\sigma-\beta\rho g)/\alpha$ with zero Neumann boundary condition, which makes the Wasserstein norm $\int_X \rho\|\bu\|^2 d\bx = \|(\sigma-\beta\rho g)/\alpha\|^2_{-1,\rho}$. On the other hand, the minimizer $\bv^*$ for $\min_\bv \int_Y \xi\|\bv\|^2 d\by$ satisfies $\bv^* = \nabla\psi$ and $-\nabla(\xi \nabla\psi) = \TK(-\nabla\cdot(\rho\bu))$, in which $-\nabla\cdot(\rho\bu)$ is equal to $(\sigma-\beta\rho g)/\alpha$ and is independent of the choice of $\bu$, therefore $\bv^*$ makes the Wasserstein norm $\int_Y \xi\|\bv\|^2 d\by = \|\TK(\sigma-\beta\rho g)/\alpha\|^2_{-1,\xi} = \|\TK(\sigma-\beta\rho g)/\alpha\|^2_{-1,{\TK}\rho}$. Finally replace $\xi$ by $\TK\rho$ in the last term of (\ref{eqn:K_UnSyncOT_inner_min}) leads to the desired result.
\end{proof}

\begin{remark}
    For Kantorovich UnSyncOT, we are unable to rewrite it in terms of the Onsager operator $\mathbb{K}(\rho)$, therefore we only leave it in terms of the dissipation potential $\Psi_{\rho}(\sigma)=\Psi(\rho,\sigma)$ (\ref{eqn:K_UnSyncOT_dissipationpotential}) as shown in the list \ref{eqn:K_operators}.
\end{remark}

Though Kantorovich UnSyncOT (\ref{eqn:KantorovichUnSynOT}) can be recast into a single-space problem (\ref{eqn:KantorovichUnSynOT_onespace}), computing $\Phi_t$ at each step requires solving an elliptic equation with Neumann boundary condition, which is computationally intensive. To improve the computational efficiency, we instead adopt an approximate formulation. To this end, we introduce a lemma which uses trapezoidal quadrature with respect to the Hellinger-Kantorovich (HK) metric given in \reff{eqn:Hellinger_Kantorovich} \cite{liero2016optimal} to approximate the squared WFR distance. We use a shorter notation $\mathrm{HK}=\mathrm{HK}_{\alpha,\beta}$ in this section.

\begin{lemma}\label{lemma:HK_quadrature}
    Let $\xi_t \in \mathrm{AC}\big([0,1];(\mathcal{M}_+(Y),\mathrm{HK})\big)$ be an absolutely continuous curve in the metric space $(\mathcal{M}_+(Y),\mathrm{HK})$, and the Hellinger-Kantorovich metric derivative $t \rightarrow |\dot{\xi}_t|_{\mathrm{HK}}$ belong to $L^2(0,1)$. Then we have that
    \begin{align}\label{eqn:HK_metric}
        |\dot{\xi}_t|^2_{\mathrm{HK}} = \inf_{(\bv,h)} \left\{ \int_Y \left(\alpha\|\bv_t\|^2 + \beta |h_t|^2\right)\xi_t d\by: \partial_t\xi_t + \alpha\nabla\cdot(\xi_t\bv_t) = \beta\xi_th_t, \ \bv_t\cdot\mathbf{n}|_{\partial Y} = 0  \right\}.
    \end{align}
    For a partition $\Pi=\{ 0=t_0<\cdots < t_N = 1 \}$ with mesh size $|\Pi|:=\max_i \Delta t_i$ and $\Delta t_i = t_{i+1} - t_i$, define the discrete action
    \begin{align}\label{eqn:discrete_action}
        A_{\Pi}(\xi_t): = \sum_{i=0}^{N-1}\frac{\mathrm{HK}^2(\xi_{t_i},\xi_{t_{i+1}})}{\Delta t_i},
    \end{align}
    then we have that for any partition $\Pi$,
    \begin{align}\label{eqn:partition_bound}
        A_{\Pi}(\xi_t) \le \int_0^1 |\dot{\xi}_t|_{\mathrm{HK}}^2 dt,
    \end{align}
    and 
    \begin{align}
        \lim_{|\Pi|\rightarrow 0 } A_{\Pi} (\xi_t) = \int_0^1 |\dot{\xi}_t|_{\mathrm{HK}}^2 dt.
    \end{align}
    In particular, if $\xi_t$ is the constant-speed $\mathrm{HK}$-geodesic between $\xi_0$ and $\xi_1$, then $A_{\Pi}(\xi_t) = \mathrm{HK}^2(\xi_0,\xi_1)$ for any partition $\Pi$.  
\end{lemma}

\begin{proof}
    The equation (\ref{eqn:HK_metric}) is a direct result of Theorem 17 and Theorem 18 in \cite{liero2016optimal}.
    By the comprehensive study on Hellinger-Kantorovich distance in \cite{liero2016optimal,liero2018optimal} and the property of metric derivative in \cite{ambrosio2005gradient} Theorem 1.1.2, we have that for any $0<r<s<1$, $\mathrm{HK}(\xi_{r},\xi_s) \le \int_r^s |\dot{\xi}_t|_{\mathrm{HK}}dt$. Then squaring the two sides and applying Cauchy-Schwarz inequality, we have
    \begin{align*}  
    \frac{\mathrm{HK}^2(\xi_{t_i},\xi_{t_{i+1}})}{\Delta t_i} \le \frac{1}{\Delta t_i} \left( \int_{t_i}^{t_{i+1}}|\dot{\xi}_t|_{\mathrm{HK}}\ dt\right)^2 \le \int_{t_i}^{t_{i+1}} |\dot{\xi}_t|_{\mathrm{HK}}^2 \ dt.
    \end{align*}
    Summing over $i$ gives the bound (\ref{eqn:partition_bound}). 

    By Theorem 1.1.2 of \cite{ambrosio2005gradient}, the metric derivative $|\dot{\xi}_t|_{\mathrm{HK}}$ exists for a.e. $t\in(0,1)$ and 
    \begin{align*}
        \lim_{h\rightarrow 0^+} \frac{\mathrm{HK}(\xi_{t+h},\xi_{t})}{h} = |\dot{\xi}_t|_{\mathrm{HK}},
    \end{align*}
    so it follows that 
    \begin{align}\label{eqn:DCT_inequality}
        \frac{\mathrm{HK}^2(\xi_{t+h},\xi_{t})}{h^2} = |\dot{\xi}_t|_{\mathrm{HK}}^2 + o(1).
    \end{align}
    Then by the dominated convergence theorem, using the bound \reff{eqn:DCT_inequality} to dominate, we have on a uniform partition $\Pi$ that
    \begin{align*}
        \lim_{\Delta t \rightarrow 0} \sum_{i=0}^{N-1} \frac{\mathrm{HK}^2(\xi_{t_i},\xi_{t_{i+1}})}{\Delta t} = \lim_{\Delta t \rightarrow 0} \sum_{i=0}^{N-1} \Big(|\dot{\xi}_{t_i}|_{\mathrm{HK}}^2 \Delta t + o(\Delta t)\Big) = \int_0^1 |\dot{\xi}_t|_{\mathrm{HK}}^2 dt.
    \end{align*}
    The same argument works for any refining nonuniform sequence with $|\Pi|\rightarrow 0$.

    When $\xi_t$ is a constant-speed HK-geodesic (see \cite{liero2016optimal}, Theorem 12), then $\mathrm{HK}(\xi_{t_i},\xi_{t_{i+1}}) = (t_{i+1}-t_i)\mathrm{HK}(\xi_0,\xi_1)$, so every term in the discrete action $A_{\Pi}(\xi_t)$ equals $\Delta t_i \mathrm{HK}^2(\xi_0,\xi_1)$, and the sum is $\mathrm{HK}^2(\xi_0,\xi_1)$.
\end{proof}

\begin{remark}
    Lemma \ref{lemma:HK_quadrature} indicates that for a given curve $\xi_t$ in $(\mathcal{M}_+(Y),\mathrm{HK})$, the secondary cost 
    \begin{align}
        \inf_{(\bv_t,h_t)} \left\{ \int_Y \left(\alpha\|\bv_t\|^2 + \beta |h_t|^2\right)\xi_t(d\by): \partial_t\xi_t + \alpha\nabla\cdot(\xi_t\bv_t) = \beta\xi_th_t, \ \bv_t\cdot\mathbf{n}|_{\partial Y} = 0  \right\},
    \end{align}
    can be approximated by the discrete action $A_{\Pi}(\xi_t)$ in (\ref{eqn:discrete_action}), with an error $o(1)$.
\end{remark}

According to lemma \ref{lemma:HK_quadrature}, we introduce the approximate Kantorovich UnSyncOT problem.
\begin{definition}[Approximate Kantorovich UnSyncOT]
    Given a partition 
    \[
    \Pi=\{ 0=t_0<\cdots < t_N = 1 \},
    \]
    we define the approximate Kantorovich unbalanced synchronized optimal transport problem as
\begin{align}\label{eqn:KantorovichSynOT_approximate}
    \mathrm{USOT}_{\mathrm{K},\Delta t}(\bar{\rho}_0,\bar{\rho}_1) = \inf_{(\rho,\mathbf{u},g)\in\mathcal{C}_{\mathrm{USOT}}^{\mathrm{K},\Delta t}} c_1 \int_0^1 \int_{X} \Big[ \alpha\|\mathbf{u}\|^2 + \beta |g|^2 \Big] \rho d\bx dt   + c_2 \sum_{i=0}^{N-1} \frac{1}{\Delta t_i}\mathrm{HK}^2(\xi_{t_i},\xi_{t_{i+1}}),
    \end{align}
    in which    \begin{align}\label{eqn:KantorovichSynOT_approximate_constraints}
      \mathcal{C}^{\mathrm{K},\Delta t}_{\mathrm{USOT}}:= \left\{ (\rho,\mathbf{u},g) \middle| 
      \begin{array}{l}
      \partial_t\rho_t + \alpha\nabla\cdot(\rho_t\mathbf{u}_t) = \beta\rho_t g_t \ \mathrm{in}\ \mathcal{D}'(X),\
      \rho_{0,1} = \bar{\rho}_{0,1}, \bu_t\cdot\mathbf{n}|_{\partial X} = 0,\\
      \xi_{t_i}(\mathbf{y}) = (\TK\rho_{t_i})(\by)= \int_{X} \pi(\bx,\by)\rho_{t_i}(\bx) d\bx.
      \end{array}
      \right\}
    \end{align}
\end{definition}

\begin{remark}
    In most cases, $\mathrm{HK}(\xi_{t_i},\xi_{t_{i+1}})$ is computed on flat space $Y$. In a special case, when we treat the Monge UnSyncOT as a degenerated Kantorovich UnSyncOT, $\mathrm{HK}(\xi_{t_i},\xi_{t_{i+1}})$ needs to be solved on $Y = \mathbf{T}(X)$, an embedded Riemannian manifold. See Section \ref{expr} for the numerical experiments.
\end{remark}

\section{Pure Transport Case ($\beta=0$): Balanced Synchronized Optimal Transport}\label{section:Extreme_case_I}

When $\beta = 0$ (and set $\alpha = 1$) and $\bar{\rho}_0,\bar{\rho}_1$ are of equal mass, the UnSyncOT (\ref{eqn:SynOT_dSpaces}) degenerates to (balanced) synchronized optimal transport (SyncOT), which was proposed in \cite{cang2025synchronized}. In a two-space case when $X$ is an open subset of $\mathbb{R}^m$, and $Y (\subset \mathbb{R}^n)$ is open (for Kantorovich case) or relatively open (for Monge case) with respect to $\bT(\mathbb{R}^m)$, SyncOT aims to minimize the following cross-space action
\begin{align}\label{eqn:SynOT_2Spaces}
\mathrm{SOT}(\bar{\rho}_0,\bar{\rho}_1) := \inf_{(\rho,\mathbf{u};\xi,\bv)\in\mathcal{C}_{\mathrm{SOT}}} \ \int_0^1 \bigg(  c_1 \int_{X}  \rho_t\|\mathbf{u}_t\|^2 d\mathrm{vol}_X   
 + c_2 \int_{Y}   \xi_t \|\mathbf{v}_t\|^2 d\mathrm{vol}_Y \bigg) dt,
\end{align}
with feasible set defined as
\begin{align}\label{eqn:SynOT_2Spaces_Constraints}
  \mathcal{C}_{\mathrm{SOT}}(\bar{\rho}_0,\bar{\rho}_1):= \left\{ (\rho,\mathbf{u};\xi,\bv) \middle| 
  \begin{array}{l}
  \partial_t\rho + \nabla\cdot(\rho\mathbf{u}) = 0, \ 
  \rho_{0,1} = \bar{\rho}_{0,1}, \langle \bu, \mathbf{n}|_{\partial X} \rangle_{g_X}= 0,\\
  \partial_t\xi + \nabla\cdot(\xi \mathbf{v}) = 0, \ \xi_t = \mathcal{T}\rho_t,
  \ \langle \bv, \mathbf{n}|_{\partial Y} \rangle_{g_Y} = 0.
  \end{array}
  \right\}
\end{align}
Here we adopt a generic Riemannian manifold notation for both spaces, just as what we do in \reff{eqn:SynOT_dSpaces}, even though the primary space $X$ is flat. In what follows, we will use Riemannian manifold notation (respectively Euclidean notation) for the secondary space in the Monge formulation (respectively Kantorovich formulation).

\subsection{Monge Synchronized Optimal Transport}\label{subsection:Monge_SyncOT}

Let the map $\mathcal{T}$ be of the Monge form $\mathcal{T} = \mathcal{T}_{\mathrm{M}} = \mathbf{T}_{\sharp}$ for a $C^1$-embedding $\mathbf{T}:X\subseteq \mathbb{R}^m\rightarrow Y = \bT(X)$ such that $Y$ is a $C^1$ embedded $m$-dimensional submanifold of $\mathbb{R}^n$. Then the SyncOT (\ref{eqn:SynOT_2Spaces}) becomes the Monge SyncOT problem
\begin{align}\label{eqn:Monge_SynOT_2Spaces}
{\mathrm{SOT}}_{\mathrm{M}}(\bar{\rho}_0,\bar{\rho}_1) := \inf_{(\rho,\mathbf{u};\xi,\bv)\in\mathcal{C}_{\mathrm{SOT}}^{\mathrm{M}}} \ \int_0^1 \bigg(  c_1 \int_{X} \|\mathbf{u}_t\|^2 \rho_t d\bx   
 + c_2 \int_{Y}   \|\mathbf{v}_t\|^2 \xi_t d\mathrm{vol}_Y \bigg) dt,
\end{align}
for $\rho_t(\bx)\in\mathcal{P}_2(X), \bu_t(\bx)\in T_\bx X, \xi_t(\by)\in\mathcal{P}_2(Y), \bv_t(\by)\in T_{\by}Y$
with feasible set defined as
\begin{align}\label{eqn:Monge_SynOT_2Spaces_Constraints}
  \mathcal{C}_{\mathrm{SOT}}^{\mathrm{M}}(\bar{\rho}_0,\bar{\rho}_1):= \left\{ (\rho,\mathbf{u};\xi,\bv) \middle| 
  \begin{array}{l}
  \partial_t\rho + \nabla\cdot(\rho\mathbf{u}) = 0, \ 
  \rho_{0,1} = \bar{\rho}_{0,1}, \  \bu_t \cdot\mathbf{n}|_{\partial X} = 0,\\
  \partial_t\xi + \mathrm{div}_{Y}(\xi \mathbf{v}) = 0, \ \xi_t d\mathrm{vol}_Y = \mathbf{T}_{\sharp}(\rho_t d\bx),
  \ \langle \bv_t, \mathbf{n}|_{\partial Y} \rangle_g = 0.
  \end{array}
  \right\}
\end{align}

In \cite{cang2025synchronized}, we proved that the Monge SyncOT can be recast into a single-space dynamical OT with a modified kinetic energy. The conclusion of this theorem can be viewed as a direct consequence of lemma \ref{lemma:W2_Anorm}.
\begin{theorem}[\cite{cang2025synchronized}]\label{theorem:MongeSync}
    When taking $\mathcal{T} = \mathbf{T}_{\sharp}$ in the SyncOT \reff{eqn:Monge_SynOT_2Spaces}-\reff{eqn:Monge_SynOT_2Spaces_Constraints}, it can be reformulated as the following single-space dynamical OT:
    \begin{align}\label{eqn:Monge_SyncOT_SingleSpace}
    \mathrm{BB}^2_{X,\mathbf{A}}(\bar{\rho}_0,\bar{\rho}_1) = \min_{(\rho,\mathbf{u})\in\mathcal{C}^X_{\mathrm{BB}}(\bar{\rho}_0,\bar{\rho}_1)} \int_0^1\int_{X} \rho_t(\mathbf{x}) \|\mathbf{u}_t(\mathbf{x})\|_{\mathbf{A}}^2  d\bx dt,
\end{align}
where $\mathcal{C}^X_{\mathrm{BB}}$ is defined in \reff{eqn:feasible_BB_flat}, $\mathbf{A} = c_1 \mathbf{I} + c_2(\nabla\mathbf{T})^{\mathrm{T}}(\nabla\mathbf{T})$, and $\|\mathbf{u}\|_{\mathbf{A}}^2:=\mathbf{u}^{\mathrm{T}}\mathbf{A}\mathbf{u}$.
\end{theorem}
\begin{proof}
    The second half of the action in \reff{eqn:Monge_SynOT_2Spaces} is the action in the secondary space $Y$, which is identical to the action with $G$-norm in the primary space $X$ with $G = (\nabla\mathbf{T})^{\mathrm{T}}(\nabla\mathbf{T})$, as shown in the proof of Lemma \ref{lemma:W2_Anorm}, particularly the equation (\ref{eqn:Yaction_equal_Xaction_Anorm}). Then this half together with the first half of the action lead to a (total) action with $\mathbf{A}$-norm in the primary space $X$. Taking infimum over the feasible set $\mathcal{C}_{\mathrm{BB}}^X$ for the total action yields \reff{eqn:Monge_SyncOT_SingleSpace}. 
\end{proof}

Here we give the dual form of the Monge SyncOT which is derived in \cite{cang2025synchronized}.
\begin{corollary}[\cite{cang2025synchronized}]\label{corollary:Monge_SyncOT_dual}
    The minimizer $(\rho_t(\bx),\bu_t(\bx))$ of the single-space Monge SyncOT problem \reff{eqn:Monge_SyncOT_SingleSpace} satisfies
    \begin{align}\label{eqn:Monge_SyncOT_dual}
        \bu_t(\bx) = \bA^{-1}\nabla\Phi_t(\bx),
    \end{align}
    in which the dual variable $\Phi_t(\bx)$ is determined by
    \begin{align}\label{eqn:Monge_SyncOT_dualform}
        \begin{cases}
            \partial_t\rho_t + \nabla\cdot(\rho_t\bA^{-1}\nabla\Phi_t) = 0, \quad \rho_{0,1} = \bar{\rho}_{0,1};\\
            \partial_t\Phi_t + \frac{1}{2}\|\nabla\Phi_t\|_{\bA^{-1}}^2 \le 0.          
        \end{cases}
    \end{align}
    If $\rho_t > 0$, then
    \begin{align}
        \partial_t\Phi_t + \frac{1}{2}\|\nabla\Phi_t\|_{\bA^{-1}}^2 = 0. 
    \end{align}
\end{corollary}

The dual form in Corollary \ref{corollary:Monge_SyncOT_dual} directly implies the reformulation of the single-space Monge SyncOT in terms of the dissipation distance.
\begin{corollary}\label{corollary:Monge_SyncOT_Onsager}
    In the framework of dissipation distance (\ref{eqn:dissipation_potential}), the single-space Monge (balanced) synchronized optimal transport (\ref{eqn:Monge_SyncOT_SingleSpace}) can be recast into the form (\ref{eqn:dissipation_distance_K}) with the Onsager operator $\mathbb{K}$ given as  $\mathbb{K}(\rho_t)g_t = -\nabla\cdot(\rho_t\bA^{-1}\nabla g_t)$. This is one of the cases in List \ref{eqn:K_operators}.
\end{corollary}

Additionally, using Theorem \ref{theorem:MongeSync}, we can show that Monge SyncOT generates constant-speed geodesics under some specific conditions. 
\begin{corollary}\label{corollary:geodesics}
    Given the Monge SyncOT problem \reff{eqn:Monge_SynOT_2Spaces}, and assume there exists a $C^1$-embedding $\mathbf{S}: X \rightarrow Z = \mathbf{S}(X) \subseteq \mathbb{R}^n$, such that
    \begin{align}\label{eqn:A_equal_S}
        \mathbf{A} = c_1 \mathbf{I} + c_2(\nabla\mathbf{T})^{\mathrm{T}}(\nabla\mathbf{T}) = (\nabla\mathbf{S})^{\mathrm{T}}(\nabla\mathbf{S}), 
    \end{align}    
    then we have
    \begin{align}\label{eqn:W2A_BBZ}
        W_{2,\mathbf{A}}(\bar{\rho}_0,\bar{\rho}_1) 
        =\mathrm{BB}_{X,\mathbf{A}}(\bar{\rho}_0,\bar{\rho}_1)  = \mathrm{BB}_{Z}(\mathbf{S}_{\sharp}\bar{\rho}_0, \mathbf{S}_{\sharp}\bar{\rho}_1) 
        = W_2({\mathbf{S}}_{\sharp}\bar{\rho}_0, {\mathbf{S}}_{\sharp}\bar{\rho}_1),
    \end{align}
    namely, the map
    \begin{align}
        \mathbf{S}_{\sharp}: (\mathcal{M}_+(X),W_{2,{\mathbf{A}}})\rightarrow (\mathcal{M}_+(Z),W_2)
    \end{align}
    is an isometric bijection.
    Besides, the minimizer $t\rightarrow \rho_t^*$ is a constant-speed geodesic with respect to the $W_{2,\mathbf{A}}$ metric.
\end{corollary}
\begin{proof}
    The proof of the identities in \reff{eqn:W2A_BBZ} is a direct consequence of Lemma \ref{lemma:W2_Anorm} given the condition \reff{eqn:A_equal_S}. Additionally, 
    since $t\rightarrow \eta_t^* = \mathbf{S}_{\sharp}\rho_t^*$ is a constant-speed geodesic in $(\mathcal{P}(Z),W_2)$, and $\mathbf{S}_{\sharp}$ introduces an isometric bijection between $(\mathcal{P}(X),W_{2,\mathbf{A}})$ and $(\mathcal{P}(Z),W_2)$, therefore $t\rightarrow \rho_t^*$ is also a constant-speed geodesic in $(\mathcal{P}(X),W_{2,\mathbf{A}})$.
\end{proof}

Corollary \ref{corollary:geodesics} indicates that the jointly optimal trajectory between spaces $X$ and $Y$ can be recovered by the optimal trajectory in space $Z$. Inspired by this fact, we can define the following interpolation between spaces.
\begin{definition}[Wasserstein Barycenter between Spaces]
    Given two spaces $X$ and $Y$ linked by a $C^1$ embedding $\mathbf{T}:X\rightarrow Y=\mathbf{T}(X)$, and a weight vector $\mathbf{c}=(c_1,c_2)^{\mathrm{T}}\in\Sigma_2$, a space $Z$ is called the Wasserstein barycenter between $X$ and $Y$, denoted by $Z = \mathrm{WB}(X,Y;\mathbf{c})$, provided that the SyncOT \reff{eqn:Monge_SyncOT_SingleSpace} between $X$ and $Y$ is equivalent to the classic dynamical OT problem in $Z$, namely, there exists a $C^1$-embedding $\mathbf{S}: X \rightarrow Z = \mathbf{S}(X)$ such that equation (\ref{eqn:A_equal_S}) holds.
\end{definition}

\begin{example}
    Consider $X = \mathbb{R}^2$ and $\mathbf{T}: (x,y)\rightarrow (x,y, F(x,y))$ with $F\in C_{c}^{\infty}(\mathbb{R}^2)$. Then 
    \begin{align*}
        \mathbf{A} = c_1 \mathbf{I} + c_2(\nabla\mathbf{T})^{\mathrm{T}}(\nabla\mathbf{T}) 
        = \mathbf{I} + c_2(\nabla F)(\nabla F)^\mathrm{T}.
    \end{align*}
    We can take $\mathbf{S}: (x,y)\rightarrow (x,y, \sqrt{c_2}F(x,y))$, then one can verify that $\mathbf{A} = (\nabla\mathbf{S})^{\mathrm{T}}(\nabla\mathbf{S})$. Therefore, Corollary \ref{corollary:geodesics} applies and optimal trajectory $\rho_t^*(\bx)$ is a constant-speed geodesic with respect to $W_{2,\mathbf{A}}$ metric.

    Furthermore, the space $Z:=\big((x,y,z)| z = \sqrt{c}_2F(x,y)\big) \subseteq \mathbb{R}^3$ is the Wasserstein Barycenter between $X = \mathbb{R}^2$ and $Y:=\big((x,y,z)| z = F(x,y)\big) \subseteq \mathbb{R}^3$. In the extreme cases when $\mathbf{c} = (1,0)^{\mathrm{T}}$, $Z = X$; when $\mathbf{c} = (0,1)^{\mathrm{T}}$, $Z = Y$.
\end{example}

\begin{remark}
    Given Riemaniann metric $\mathbf{A}$, finding a diffeomorphism $\mathbf{S}$ such that $\mathbf{A} = (\nabla\mathbf{S})^{\mathrm{T}}(\nabla\mathbf{S})$ is indeed the isometric immersion problem. A short review of this problem can be seen in \cite{han2024isometric}.
\end{remark}

\subsection{Kantorovich Synchronized Optimal Transport}
In the setting of Kantorovich SyncOT, we assume $X\subseteq \bbR^m$ and $Y\subseteq \bbR^n $ to be open and convex.
When the map $\mathcal{T}$ is of the Kantorovich form $\mathcal{T} = \mathcal{T}_{\mathrm{K}}$ as defined in \reff{eqn:TK}, the SyncOT (\ref{eqn:SynOT_2Spaces}) becomes the Kantorovich SyncOT problem
\begin{align}\label{eqn:Kantorovich_SynOT_2Spaces}
\mathrm{SOT}_{\mathrm{K}}(\bar{\rho}_0,\bar{\rho}_1) := \inf_{(\rho,\mathbf{u};\xi,\bv)\in\mathcal{C}^{\mathrm{K}}_{\mathrm{SOT}}} \ \int_0^1 \bigg(  c_1 \int_{X}  \rho_t \|\mathbf{u}_t\|^2 d\bx   
 + c_2 \int_{Y}   \xi_t \|\mathbf{v}_t\|^2 d\by   \bigg) dt,
\end{align}
with feasible set defined as
\begin{align}\label{eqn:Kantorovich_SynOT_2Spaces_Constraints}
  \mathcal{C}^{\mathrm{K}}_{\mathrm{SOT}}= \left\{ (\rho,\mathbf{u};\xi,\bv) \middle| 
  \begin{array}{l}
  \partial_t\rho + \nabla\cdot(\rho\mathbf{u}) = 0 \mathrm{\ in\ } \mathcal{D}'(X), \ 
  \rho_{0,1} = \bar{\rho}_{0,1}, \  \bu\cdot\mathbf{n}|_{\partial X} = 0 ,\\
  \partial_t\xi + \nabla\cdot(\xi \mathbf{v}) = 0 \mathrm{\ in\ } \mathcal{D}'(Y),
  \ \bv\cdot\mathbf{n}|_{\partial Y} = 0,\ 
  \xi_t(\by) = (\TK\rho_t)(\by).
  \end{array}
  \right\}
\end{align}
As a special case of Theorem \ref{theorem:Kantorovich_UnSyncOT}, we have the following corollary which describes the Kantorovich SyncOT (\ref{eqn:Kantorovich_SynOT_2Spaces}) solely in the primary space $X$.

\begin{corollary}\label{corollary:Kantorovich_SyncOT}
The Kantorovich balanced synchronized optimal transport (\ref{eqn:Kantorovich_SynOT_2Spaces}) is equivalent to the following single-space balanced dynamical OT:
\begin{align*}
\mathrm{SOT}_{\mathrm{K}}(\bar{\rho}_0,\bar{\rho}_1) = \inf_{(\rho,\mathbf{u})\in\mathcal{C}_{\mathrm{BB}}^X} \ \int_0^1 \int_{X} \bigg[ c_1 \|\mathbf{u}_t\|^2 + c_2\TK^*\left(\|\nabla\Phi_t\|^2\right) \bigg] \rho_t d\bx  
dt,
\end{align*}
in which $\mathcal{C}_{\mathrm{BB}}^X$ is defined in \eqref{eqn:feasible_BB_flat}, and $\Phi_t$ is defined in (\ref{eqn:Phi}). Besides, we have the following bound estimate for Kantorovich SyncOT,
\begin{align*}
    c_1 \mathrm{BB}_{X}^2(\bar{\rho}_0,\bar{\rho}_1) + c_2 \mathrm{BB}_{Y}^2(\TK\bar{\rho}_0,\TK\bar{\rho}_1) \le \mathrm{SOT}_{\mathrm{K}}(\bar{\rho}_0,\bar{\rho}_1) \le (c_1+c_2\Lambda^2) \mathrm{BB}_{X}^2(\bar{\rho}_0,\bar{\rho}_1),
\end{align*}
in which $\Lambda = \mathrm{esssup}_{t\in(0,1)}\|\TK\|_{T_{\rho_t}\rightarrow T_{\TK\rho_t}}$.
\end{corollary}

As a special case of Corollary \ref{corollary:K_UnSyncOT_dissipation}, the following corollary provides a reformulation of the Kantorovich SyncOT in the form of dissipation potential.

\begin{corollary}\label{corollary:K_SyncOT_dissipation}
    The Kantorovich (balanced) synchronized optimal transport \reff{eqn:Kantorovich_SynOT_2Spaces} can be rewritten in the form of dissipation potential (\ref{eqn:dissipation_potential}) 
    with $\Psi(\rho,\sigma) = \Psi_{\rho}(\sigma)$ given below
    \begin{align*}
        \Psi_{\rho}(\sigma) = c_1 \left\| \sigma \right\|_{-1,\rho}^2 +
         c_2 \left\| \TK\sigma \right\|_{-1,{\TK}\rho}^2,
    \end{align*}
    which leads to the Riemannian operator $\mathbb{G}(\rho)$,
    \begin{align*}
        \mathbb{G}(\rho) = c_1(-\Delta_\rho)^{-1} + c_2 \TK^*(-\Delta_{\TK\rho})^{-1}\TK.
    \end{align*}
    This is one of the cases in the list \ref{eqn:K_operators}.
    
\end{corollary}

\section{Pure Reaction Case ($\alpha=0$): Synchronized Fisher-Rao Problem}\label{section:Extreme_case_II}

When $\alpha = 0$ (and set $\beta = 1$), the UnSyncOT (\ref{eqn:SynOT_dSpaces}) degenerates to synchronized Fisher-Rao (SyncFR) problem. In a two-space case when $X \subseteq \mathbb{R}^m, Y \subseteq \mathbb{R}^n$ are open and convex, SyncFR aims to minimize the following cross-space action
\begin{align}\label{eqn:FR_SynOT_2Spaces}
\mathrm{SFR}^2(\bar{\rho}_0,\bar{\rho}_1) := \inf_{(\rho,g;\xi,h)\in\mathcal{C}_{\mathrm{SFR}}} A_{\mathrm{SFR}} = \int_0^1 \bigg(  c_1 \int_{X}  \rho_t\|g_t\|^2 d\bx   
 + c_2 \int_{Y}   \xi_t \|h_t\|^2   d\by   \bigg) dt,
\end{align}
with feasible set defined as
\begin{align}\label{eqn:FR_SynOT_2Spaces_Constraints}
  \mathcal{C}_{\mathrm{SFR}}= \left\{ (\rho,g;\xi,h) \middle| 
  \begin{array}{l}
  \partial_t\rho = \rho g \mathrm{\ in\ } \mathcal{D}'(X), \ 
  \rho_{0,1} = \bar{\rho}_{0,1}, \\
  \partial_t\xi = \xi h \mathrm{\ in\ } \mathcal{D}'(Y), \ \xi_t d\by = \mathcal{T}(\rho_t d\bx).
  \end{array}
  \right\}
\end{align}

\subsection{Monge Synchronized Fisher-Rao Problem}

When $\mathcal{T} = \mathcal{T}_{\mathrm{M}} = \mathbf{T}_{\sharp}$, the synchronization constraint $\xi_t = \mathbf{T}_{\sharp}\rho_t$ in (\ref{eqn:FR_SynOT_2Spaces_Constraints}) with $\mathbf{T}$ being a $C^1$ embedding $\mathbf{T}:X\rightarrow Y = \mathbf{T}(X)$ forces $g_t(\bx) = h_t(\mathbf{T}(\bx))$. In this case, the two-space action in \reff{eqn:FR_SynOT_2Spaces} collapses to the single FR action in \reff{eqn:FR}. So the synchronized two-space FR brings no new dynamics beyond one space. We summary this fact in the following theorem.

\begin{theorem}
    Let $\mathbf{T}$ be a $C^1$ embedding $\mathbf{T}: X \rightarrow Y = \mathbf{T}(X)$ and consider Monge synchronized Fisher-Rao problem $\mathrm{SFR}_{\mathrm{M}}$ as given in (\ref{eqn:FR_SynOT_2Spaces}) with $\mathcal{T}=\mathbf{T}_{\sharp}$ in (\ref{eqn:FR_SynOT_2Spaces_Constraints}). The feasibility condition (\ref{eqn:FR_SynOT_2Spaces_Constraints}) implies that $h_t(\mathbf{T}(\bx)) = g_t(\bx)$. Then the action $A^{\mathrm{M}}_{\mathrm{SFR}} = A_{\mathrm{SFR}}$ in (\ref{eqn:FR_SynOT_2Spaces}) reduces to
    \begin{align*}
        A^{\mathrm{M}}_{\mathrm{SFR}} = \int_0^1 \int_X  |g_t|^2 \rho_t d\bx dt.
    \end{align*}
    Consequently,
    \begin{align}
        \mathrm{SFR}_{\mathrm{M}}(\bar{\rho}_0,\bar{\rho}_1) = \mathrm{FR}_X(\bar{\rho}_0,\bar{\rho}_1),
    \end{align}
    and the minimizing path is the single-space FR geodesic on $X$,
    \begin{align}
        \rho_t(\bx) = \Big( (1-t)\sqrt{\bar{\rho}_0(\bx)}
        +t\sqrt{\bar{\rho}_1(\bx)} \Big)^2, \quad 
        g_t(\bx) = \frac{\partial_t\rho_t(\bx)}{\rho_t(\bx)} = \partial_t \log \rho_t(\bx).
    \end{align}
\end{theorem}

\begin{proof}
    For any test function $\phi\in C_c^{\infty}(Y)$, we have
    \begin{align*}
        \frac{d}{dt}\int_Y \phi(\by) \xi_t(\by) d\mathrm{vol}_Y
        =\frac{d}{dt}\int_X \phi(\mathbf{T}(\bx))\rho_t(\bx)d\bx
        =\int_X \phi(\mathbf{T}(\bx))g_t(\bx)\rho_t(\bx) d\bx.
    \end{align*}
    Since $\xi_t$ satisfies $\partial_t\xi_t = \xi_t h_t$, then
    \begin{align*}
        \frac{d}{dt}\int_Y \phi(\by) \xi_t(\by) d\mathrm{vol}_Y 
        = \int_Y \phi(\by) \xi_t(\by) h_t(\by) d\mathrm{vol}_Y
        = \int_X \phi(\mathbf{T}(\bx))  h_t(\mathbf{T}(\bx)) 
        \rho_t(\bx) d\bx.
    \end{align*}
    Comparing the two right-hand sides yields $h_t(\mathbf{T}(\bx)) = g_t(\bx)$. Then the SyncFR action becomes
    \begin{align*}
        A^{\mathrm{M}}_{\mathrm{SFR}} &= \int_0^1 \bigg( c_1 \int_{X} |g_t|^2 \rho_t d\bx   
        + c_2 \int_{Y} |h_t|^2 \xi_t d\mathrm{vol}_Y   \bigg) dt \\
        &= \int_0^1 \bigg( c_1 \int_{X} |g_t|^2 \rho_t d\bx   
        + c_2 \int_{X} |h_t\circ\mathbf{T}|^2 \rho_t d\bx   \bigg) dt \\
        &= \int_0^1 \bigg( c_1 \int_{X} |g_t|^2 \rho_t d\bx   
        + c_2 \int_{X} |g_t|^2 \rho_t d\bx   \bigg) dt
        = A_{\mathrm{FR}}.
    \end{align*}
    Consequently, $\mathrm{SFR}_{\mathrm{M}}(\bar{\rho}_0,\bar{\rho}_1) = \mathrm{FR}_X(\bar{\rho}_0,\bar{\rho}_1)$. The minimizing path is due to the results in \cite{chizat2018interpolating}.
\end{proof}

\subsection{Kantorovich Synchronized Fisher-Rao Problem}

In the setting of Kantorovich SyncFR problem, we assume $X\subseteq \bbR^m$ and $Y\subseteq \bbR^n $ to be open and convex. Taking the map $\mathcal{T}$ in SyncFR constraint \reff{eqn:FR_SynOT_2Spaces_Constraints} to be $\mathcal{T} = \mathcal{T}_{\mathrm{K}}$ as defined in \reff{eqn:TK}, the Kantorovich SyncFR does not collapse the two-space action $A_{\mathrm{SFR}}$ into the single FR action $A_{\mathrm{FR}}$ in \reff{eqn:FR}. In this case, we rather provide lower and upper bound estimates for the two-space action. We summary the results in the following theorem.
\begin{theorem}\label{theorem:Kantorovich_FR}
    The Kantorovich synchronized Fisher-Rao problem $\mathrm{SFR}_{\mathrm{K}}$ given in (\ref{eqn:FR_SynOT_2Spaces}) with $\mathcal{T} = \mathcal{T}_{\mathrm{K}}$ defined in \reff{eqn:FR_SynOT_2Spaces_Constraints} can be recast into the following minimization problem    \begin{align}\label{eqn:Kantorovich_FR_1space}
        \mathrm{SFR}_{\mathrm{K}}^2(\bar{\rho}_0,\bar{\rho}_1) = \inf_{(\rho,g)\in\mathcal{C}_{\mathrm{FR}}(\bar{\rho}_0,\bar{\rho}_1)} \int_0^1 \int_{X}  \Big(c_1 |g_t|^2  
        + c_2  \TK^*\left(|h_t|^2\right) \Big) \rho_t d\bx dt,
\end{align}
in which $\mathcal{C}_{\mathrm{FR}}$ is the feasible set defined in (\ref{eqn:FR_feasibility}) for the classic single-space FR problem, and $h_t$ satisfies $\xi_t h_t = \TK(\rho_t g_t)$. We further have the following sharp bounds for SyncFR,
\begin{align}\label{eqn:SyncFR_bounds}
    c_1 \mathrm{FR}_X^2(\bar{\rho}_0,\bar{\rho}_1) 
    + 
    c_2 \mathrm{FR}_Y^2(\mathcal{T}_\mathrm{K}\bar{\rho}_0,\mathcal{T}_\mathrm{K}\bar{\rho}_1) \le \mathrm{SFR}_{\mathrm{K}}^2 
    \le
    \mathrm{FR}_{X}^2(\bar{\rho}_0,\bar{\rho}_1).
\end{align}
Moreover, FR contracts under $\TK$,
\begin{align}\label{eqn:FR_contractor}
    \mathrm{FR}_Y^2(\TK\bar{\rho}_0, \TK\bar{\rho}_1) \le 
    \mathrm{FR}_X^2(\bar{\rho}_0, \bar{\rho}_1).
\end{align}

\end{theorem}

\begin{proof}
    Given the feasibility condition of Kantorovich SyncFR
    \begin{align*}
      \mathcal{C}^{\mathrm{K}}_{\mathrm{SFR}}= \left\{ (\rho,g;\xi,h) \middle| 
      \begin{array}{l}
      \partial_t\rho = \rho g \mathrm{\ in\ } \mathcal{D}'(X), \ 
      \rho_{0,1} = \bar{\rho}_{0,1}, \\
      \partial_t\xi = \xi h \mathrm{\ in\ } \mathcal{D}'(Y), \ \xi_t = \TK\rho_t
      \end{array}
      \right\},
    \end{align*}
    it is evident that $\xi_t h_t = \TK(\rho_t g_t)$. Then by following the similar argument as in the proof of Theorem \ref{theorem:Kantorovich_UnSyncOT}, we obtain the bound estimate \reff{eqn:SyncFR_bounds}.

    The contraction inequality is a direct consequence of the bound \reff{eqn:SyncFR_bounds} by subtracting $c_1\mathrm{FR}_X^2$ from the two sides. On the other hand, we can also prove it by using the optimum formula \reff{eqn:FR_explicitform}. Indeed, we use the Cauchy-Schwarz inequality
    \begin{align*}
        \int_X \sqrt{\bar{\rho}_1(\bx)\bar{\rho}_0(\bx)}\pi(\bx,\by)d\bx
        \le 
        \left(\int_X \bar{\rho}_1(\bx)\pi(\bx,\by)d\bx \right)^{\frac{1}{2}}
        \left(\int_X \bar{\rho}_0(\bx)\pi(\bx,\by)d\bx \right)^{\frac{1}{2}}   
        = \sqrt{(\TK\bar{\rho}_1) (\TK\bar{\rho}_0)}.
    \end{align*}
    Then it follows that
    \begin{align*}
        (\sqrt{\TK\bar{\rho}_1} - \sqrt{\TK\bar{\rho}_0})^2 &= \TK\bar{\rho}_1 + \TK\bar{\rho}_0 - 2\sqrt{(\TK\bar{\rho}_1) (\TK\bar{\rho}_0)} \\
        &\le \TK\bar{\rho}_1 + \TK\bar{\rho}_0 - 2 \int_X \sqrt{\bar{\rho}_1(\bx)\bar{\rho}_0(\bx)}\pi(\bx,\by)d\bx \\
        &= \int_X (\sqrt{\bar{\rho}_1(\bx)} - \sqrt{\bar{\rho}_0(\bx)})^2\pi(\bx,\by) d\bx.
    \end{align*}
    Integrating the above inequality over $\by$ and noting that $\pi$ is a Markov kernel, we obtain
    \begin{align*}
        \int_Y (\sqrt{\TK\bar{\rho}_1} - \sqrt{\TK\bar{\rho}_0})^2 d\by \le \int_X (\sqrt{\bar{\rho}_1} - \sqrt{\bar{\rho}_0})^2 d\bx,
    \end{align*}
    which gives the contraction inequality (\ref{eqn:FR_contractor}) by using the explicit formulation of the FR problem in (\ref{eqn:FR_explicitform}).
\end{proof}

We discuss several special cases for the Kantorovich SyncFR problem. 

\begin{example}
    When the Markov kernel is a deterministic one: $\pi(\bx,\by) = \delta_{\by = \mathbf{T}(\bx)}$ for a diffeomorphism $\bT: X\rightarrow Y$, then $\TK$ reduces to $\bT_{\sharp}$, and the Kantorovich SyncFR degenerates to Monge SyncFR such that the upper bound in (\ref{eqn:SyncFR_bounds}) becomes equality. To see this, one can easily verify the following equation for any test function $\phi\in C_{c}^{\infty}(Y)$,
    \[
        \int_Y \phi(\by) \TK(\rho)(\by) d\by = \int_X \rho(\bx) \int_Y \phi(\by)\delta_{\by = \bT(\bx)} d\by d\bx = \int_X \rho(\bx)\phi(\bT(\bx)) d\bx = \int_Y \phi(\by) d\bT_{\sharp}(\rho)(\by).
    \]
\end{example}

\begin{example}\label{example:uniform_rate}
    In the Kantorovich SyncFR formulation,
    when the reaction rate $g_t(\bx)$ is spatially uniform: $g_t(\bx) = g(t)$, it follows that $h_t(\by) = h(t)$ is also spatially uniform and equal to $g(t)$. The ODE constraint $\partial_t\rho_t(\bx) = \rho_t(\bx)g(t)$ implies that $\rho_t(\bx)=\lambda(t)\bar{\rho}_0(\bx)$ with $\lambda(t) = \exp(\int_0^t g(s) ds)$. Hence we require $\bar{\rho}_1 = \Lambda\bar{\rho}_0$ so that a path $\rho_t$ with uniform $g(t)$ can hit $\bar{\rho}_1$.

    Since $\xi_t$ now satisfies the same ODE as $\rho_t$, the synchronized curve on $Y$ is $\xi_t(\by) = \lambda(t)(\TK\bar{\rho}_0)(\by)$. With $\rho_t = \lambda\rho_0$, $\xi_t = \lambda\xi_0$, $g = \dot{\lambda}/\lambda$, $h = g$, we have
    \[\int_X \rho_t(\bx)g(t)^2 d\bx = \int_Y \xi_t(\bx)h(t)^2 d\by = m_0 \frac{\dot{\lambda}^2}{\lambda}, \quad
    m_0 = \int_X \bar{\rho}_0(\bx) d\bx.
    \]
    Hence the Kantorovich SyncFR action reduces to
    \begin{align*}
        A^{\mathrm{K}}_{\mathrm{SFR}}[\lambda] = \int_0^1 m_0 \frac{\dot{\lambda}^2}{\lambda} dt.
    \end{align*}
    Let $s(t):=\sqrt{\lambda(t)}$. Then $\dot{\lambda}^2/\lambda = 4\dot{s}^2$. With the boundary condition $s(0)=\sqrt{1}, s(1) = \sqrt{\Lambda}$, the optimal solution is the straight line
    \begin{align*}
        s(t) = (1-t)\cdot \sqrt{1} + t \sqrt{\Lambda}, \quad
        \lambda(t) = \Big[(1-t)\cdot \sqrt{1} + t \sqrt{\Lambda}\Big]^2, \quad
        g(t) = \frac{2(\sqrt{\Lambda}-1)}{(1-t) + t \sqrt{\Lambda}}.
    \end{align*}
    Therefore $\mathrm{SFR}_{\mathrm{K}}^2 = 4m_0(\sqrt{\Lambda}-1)^2$.

    On the other hand, if $\bar{\rho}_1 = \Lambda \bar{\rho}_0$, according to (\ref{eqn:FR_explicitform}), 
    \begin{align*}
        \mathrm{FR}_{X}^2(\bar{\rho}_0,\bar{\rho}_1) = 4m_0(\sqrt{\Lambda}-1)^2, \quad
        \mathrm{FR}_{Y}^2(\TK\bar{\rho}_0,\TK\bar{\rho}_1) = 4m_0(\sqrt{\Lambda}-1)^2,
    \end{align*}
    therefore,
    \begin{align*}
        \mathrm{SFR}_{\mathrm{K}}^2(\bar{\rho}_0,\bar{\rho}_1) = c_1 \mathrm{FR}_{X}^2(\bar{\rho}_0,\bar{\rho}_1) + c_2 \mathrm{FR}_{Y}^2(\TK\bar{\rho}_0,\TK\bar{\rho}_1) = \mathrm{FR}_{X}^2(\bar{\rho}_0,\bar{\rho}_1).
    \end{align*}
    Namely, the two bounds in (\ref{eqn:SyncFR_bounds}) are reached when reaction rate $g_t$ is spatially uniform.
\end{example}

\begin{example}
    In the Kantorovich SyncFR formulation,
    when the Markov kernel is separable $\pi(\bx,\by) = a(\bx) b(\by)$, the Markovian condition (\ref{eqn:Markov_kernel}) forces $a(\bx)= (\int_Y b(\by)d\by)^{-1}=\mathrm{constant}$. Without loss of generality, we can normalize $\int_Y b(\by) d\by = 1$ and set $a(\bx)\equiv 1$.
    In this case, 
    \begin{align*}
        \xi_t(\by) = (\TK \rho_t)(\by) = \int_X\rho_t(\bx)b(\by) d\bx = m(t)b(\by), \quad m(t) = \int_X \rho_t(\bx) d\bx
    \end{align*}
    and
    \begin{align*}
        \TK(\rho_t g_t)(\by) = \left( \int_X \rho_t(\bx)g_t(\bx) d\bx \right) b(\by) = \dot{m}(t)b(\by).
    \end{align*}
    Hence the synchronized rate on $Y$ is spatially uniform
    \begin{align*}
        h_t(\by) = \frac{\TK(\rho_t g_t)}{\TK \rho_t}(\by) = \frac{\dot{m}(t)}{m(t)},
    \end{align*}
    and the $Y$-action reduces into
    \begin{align*}
        \int_Y \xi_t h_t^2 d\by = \frac{\dot{m}(t)^2}{m(t)}\int_Y b(\by) d\by = \frac{\dot{m}(t)^2}{m(t)}.
    \end{align*}
    Therefore the SyncFR action becomes
    \begin{align*}
        A^{\mathrm{K}}_{\mathrm{SFR}} = \int_0^1 \left( c_1\int_X\rho_t g_t^2d\bx + c_2 \frac{\dot{m}(t)^2}{m(t)}  \right) dt.
    \end{align*}

    According to the Cauchy-Schwarz inequality,
    \begin{align*}
        \int_X \rho_t g_t^2 d\bx \ge \frac{(\int_X \rho_t g_t d\bx)^2}{\int_X \rho_t d\bx} = \frac{\dot{m}(t)^2}{m(t)},
    \end{align*}
    which implies
    \begin{align}\label{eqn:SyncFR_mt}
        A^{\mathrm{K}}_{\mathrm{SFR}} \ge (c_1+c_2)\int_0^1 \frac{\dot{m}(t)^2}{m(t)} dt = \int_0^1 \frac{\dot{m}(t)^2}{m(t)} dt.
    \end{align}
    Minimizing the right hand side of (\ref{eqn:SyncFR_mt}) over $m(t)$ with boundary conditions $m_0 = \int_X \bar{\rho}_0 d\bx$ and $m_1 = \int_X \bar{\rho}_1 d\bx$ gives
    \begin{align*}
        m(t) = \big((1-t)\sqrt{m_0}+t\sqrt{m_1}\big)^2, \quad
        \inf_{m_t} \int_0^1 \frac{\dot{m}(t)^2}{m(t)} dt  = 4(\sqrt{m_1}-\sqrt{m_0})^2.
    \end{align*}
    Therefore Kantorovich SyncFR has lower bound: $\mathrm{SFR}_{\mathrm{K}}^2 \ge 4(\sqrt{m_1}-\sqrt{m_0})^2$. The equality holds when $g_t(\bx) = g(t)$ is spatially uniform, which is the case discussed in Example \ref{example:uniform_rate}.
\end{example}

\section{Numerical Methods for UnSyncOT}\label{section:Numerical_Monge}
In this section, we present numerical methods for solving the UnSyncOT problem. For simplicity, we focus on the 2D case and take the primary space $X=[0,1]^2$. The secondary space $Y$ depends on the transform $\mathcal{T}$.

We follow \cite{benamou2000computational,papadakis2014optimal,cang2025synchronized} to introduce the change of variable 
\begin{align}\label{eqn:change_of_variable}
(\bu_t,\rho_t, g_t) \rightarrow (\bfm_t,\rho_t, H_t),
\end{align}
where $\bfm_t = \rho_t \bu_t$ is the momentum  and $H_t=\rho_t g_t$ is the growth, just as we did in the proof of Corollary \ref{corollary:Monge_UnSyncOT_dual}. Then we introduce the uniform staggered grid for $(\mathbf{m}_t,\rho_t)$ and centered grid for $H_t$.

\subsection{Discretization on Uniform Staggered Grids}\label{subsection:Staggered}

In this section, we introduce a uniform staggered grid for the discretization of the UnSyncOT (\ref{eqn:SyncOT_momentum}). We denote by $M+1, N+1, Q+1$ the numbers of uniform staggered nodes along $\bx = (x, y)\in[0,1]^2$ and time $t\in[0,1]$. The corresponding grid sizes are $\Delta x = \frac{1}{M}, \Delta{y} = \frac{1}{N}, \Delta{t} = \frac{1}{Q}$.

The following index sets of staggered grids are defined:
\begin{align}
  \mathcal{S}_{\text{s}}^{x} &= \{ (i,j,k)\in\mathbb{Z}^3:\ 0\le i \le M, \ 0\le j \le N-1, \ 0\le k \le Q-1 \}, \\
  \mathcal{S}_{\text{s}}^{y} &= \{ (i,j,k)\in\mathbb{Z}^3:\ 0\le i \le M-1, \ 0\le j \le N, \ 0\le k \le Q-1 \}, \\ 
  \mathcal{S}_{\text{s}}^{t} &= \{ (i,j,k)\in\mathbb{Z}^3:\ 0\le i \le M-1, \ 0\le j \le N-1, \ 0\le k \le Q \}.
\end{align}
The staggered nodes over $(\bx,t) \in [0,1]^2\times [0,1]$  are given as
\begin{align}
  \mathcal{G}_{\text{s}}^x &= \left\{  (x_{i-\frac{1}{2}}, y_j, t_k) = \left( i\Delta x, \left(j + \tfrac{1}{2} \right)\Delta y, \left(k + \tfrac{1}{2} \right)\Delta t \right), \ (i,j,k) \in \mathcal{S}_{\text{s}}^{x}   \right\}, \\
  \mathcal{G}_{\text{s}}^y &= \left\{  (x_{i}, y_{j-\frac{1}{2}}, t_k) = \left( \left(i + \tfrac{1}{2} \right)\Delta x, j\Delta y, \left(k + \tfrac{1}{2} \right)\Delta t \right), \ (i,j,k) \in \mathcal{S}_{\text{s}}^{y}    \right\}, \\
  \mathcal{G}_{\text{s}}^t &= \left\{  (x_{i}, y_{j}, t_{k-\frac{1}{2}}) = \left( \left(i + \tfrac{1}{2} \right)\Delta x, \left(j + \tfrac{1}{2}\right)\Delta y, k \Delta t \right), \ (i,j,k) \in \mathcal{S}_{\text{s}}^{t}   \right\}.
\end{align}
The discrete functions defined on the staggered grid $\mathcal{G}_{\text{s}}^{q}, q = x,y,t$ are denoted by $\mathcal{M}(\mathcal{G}_{\text{s}}^{q})$, namely, for $f^{q}\in \mathcal{M}(\mathcal{G}_{\text{s}}^{q}), q = x, y, t$,
\begin{align}
  f^x = \left( f^x_{i-\frac{1}{2},j,k} \right)_{(i,j,k) \in \mathcal{S}_{\text{s}}^{x} }, \quad 
  f^y = \left( f^y_{i,j-\frac{1}{2},k} \right)_{(i,j,k) \in \mathcal{S}_{\text{s}}^{y} }, \quad
  f^t = \left( f^t_{i,j,k-\frac{1}{2}} \right)_{(i,j,k) \in \mathcal{S}_{\text{s}}^{t} }.
\end{align}
We further denote the discrete functions defined on $\mathcal{G}_{\text{s}}^{x} \times \mathcal{G}_{\text{s}}^{y} \times \mathcal{G}_{\text{s}}^{t}$ as
\begin{align*}
  U_{\text{s}} = (\bfm_{\text{s}},\rho_{\text{s}}) = (m_{\text{s}}, n_{\text{s}},\rho_{\text{s}}) = (m_{i-\frac{1}{2},j,k}, n_{i,j-\frac{1}{2},k}, \rho_{i,j,k-\frac{1}{2}})\in \mathcal{D}_{\text{s}}= \mathcal{M}(\mathcal{G}_{\text{s}}^{x}) \times \mathcal{M}(\mathcal{G}_{\text{s}}^{y}) \times \mathcal{M}(\mathcal{G}_{\text{s}}^{t}).
\end{align*}

We also define the index set of centered grid as
\begin{align}
  \mathcal{S}_{\text{c}} = \{ (i,j,k)\in\mathbb{Z}^3:\  0\le i \le M-1,\  0\le j \le N-1,\  0\le k \le Q-1 \}.
\end{align}
Then a centered grid discretization over $(\bx,t) \in [0,1]^2\times [0,1]$ is given as
\begin{align}
  \mathcal{G}_{\text{c}} = \left\{  (x_i, y_j, t_k) = \left( (i+\tfrac{1}{2}) \Delta x, (j+\tfrac{1}{2})\Delta y, (k+\tfrac{1}{2})\Delta t \right),\  (i,j,k) \in \mathcal{S}_{\text{c}}  \right\}.
\end{align}
The discrete functions defined on the centered grid $\mathcal{G}_{\text{c}}$ are denoted by $\mathcal{M}(\mathcal{G}_{\text{c}})$. The UnSyncOT variable $(\bfm,\rho)$ discretized over $\mathcal{G}_{\text{c}} \times \mathcal{G}_{\text{c}} \times \mathcal{G}_{\text{c}}$ is denoted by
\begin{align}
  U_{\text{c}} = (\bfm_{\text{c}},\rho_{\text{c}})  = (m_{\text{c}}, n_{\text{c}},\rho_{\text{c}}) = (m_{ijk}, n_{ijk}, \rho_{ijk}) \in \mathcal{D}_{\text{c}} = \mathcal{M}(\mathcal{G}_{\text{c}}) \times \mathcal{M}(\mathcal{G}_{\text{c}}) \times \mathcal{M}(\mathcal{G}_{\text{c}}).
\end{align}
The variable $H$ discretized over $\mathcal{G}_c$  is denoted by
\begin{align}
    H_\mathrm{c}=(H_{ijk})\in\mathcal{M}(\mathcal{G}_c).
\end{align}

We further introduce a midpoint interpolation operator $\mathcal{I}: \mathcal{D}_{\text{s}} \rightarrow \mathcal{D}_{\text{c}}$. For $U_{\text{s}} = (\bfm_{\text{s}},\rho_{\text{s}}) \in \mathcal{D}_{\text{s}}$, we define $\mathcal{I}(U_{\text{s}}) = (\bfm_{\text{c}},\rho_{\text{c}}) \in \mathcal{D}_{\text{c}}$ as follows:
\begin{align}
  \begin{cases}
  m_{ijk} = \frac{1}{2}( m_{i-\frac{1}{2},j,k} + m_{i+\frac{1}{2},j,k}), \\ 
  n_{ijk} = \frac{1}{2}( n_{i,j-\frac{1}{2},k} + n_{i,j+\frac{1}{2},k}), \\
  \rho_{ijk} = \frac{1}{2}( \rho_{i,j,k-\frac{1}{2}} + \rho_{i,j,k+\frac{1}{2}}),
  \end{cases}
  \quad (i,j,k)\in\mathcal{S}_{\text{c}}.
\end{align}
Then the space-time divergence operator $\div: \mathcal{D}_{\text{s}} \rightarrow \mathcal{M}(\mathcal{G}_{\text{c}}) $ is defined for $U_{\text{s}} = (\bfm_{\text{s}},\rho_{\text{s}}) \in \mathcal{D}_{\text{s}}$:
\begin{align}
  \div(U_{\text{s}})_{ijk} =  \frac{m_{i+\frac{1}{2},j,k} - m_{i-\frac{1}{2},j,k}}{\Delta x}  + 
                              \frac{ n_{i,j+\frac{1}{2},k} - n_{i, j-\frac{1}{2},k} }{\Delta y}  +
                               \frac{\rho_{i,j,k+\frac{1}{2}} - \rho_{i,j,k-\frac{1}{2}} }{\Delta t},
\end{align}
for $(i,j,k)\in\mathcal{S}_{\text{c}}$.

To handle the boundary conditions on the staggered grids, we introduce the linear operator $b$ for $U_{\text{s}} = (\bfm_{\text{s}},\rho_{\text{s}}) \in \mathcal{D}_{\text{s}}$ as
\begin{align}\label{eqn:b_Us}
  b(U_{\text{s}}) = \left( (m_{{\scriptscriptstyle-\frac{1}{2}},j,k}, m_{M-{\scriptscriptstyle\frac{1}{2}},j,k}), (n_{i,-{\scriptscriptstyle\frac{1}{2}},k}, n_{i,N-{\scriptscriptstyle\frac{1}{2}},k}), (\rho_{i,j,-{\scriptscriptstyle\frac{1}{2}}}, \rho_{i,j,Q-{\scriptscriptstyle\frac{1}{2}}})\right)_{(i,j,k)\in\mathcal{S}_{\mathrm{c}}},
\end{align}
and impose the boundary conditions
\begin{align}\label{eqn:boundarycondition}
  b(U_{\text{s}}) = b_0 : = ((\mathbf{0},\mathbf{0}),(\mathbf{0},\mathbf{0}),(\mu_{\text{c}}, \nu_{\text{c}}))\in (\mathbb{R}^{NQ})^2 \times (\mathbb{R}^{MQ})^2 \times (\mathbb{R}^{MN})^2,
\end{align}
where $\mu_{\text{c}}, \nu_{\text{c}} \in \mathbb{R}^{MN}$ are the discretized initial and terminal densities over centered grid $\{ (x_i,y_j) = ( (i+\tfrac{1}{2})\Delta x, (j+\tfrac{1}{2})\Delta y ),\ i=0:M-1, j=0:N-1 \}$.

\subsection{Primal-dual Methods for Monge Unbalanced Synchronized Optimal Transport }\label{subsection:Primal_dual_Monge}

By the change of variable \reff{eqn:change_of_variable}, we define the functional $J$ in \reff{eqn:J_linearlize}. Then the single-space Monge SyncOT formulation \reff{eqn:Monge_UnSyncOT_SingleSpace} can be recast into convex optimization problem over the pair $(\bfm,\rho,H)$,
\begin{align}\label{eqn:SyncOT_momentum}
\min_{(\bfm,\rho, H)\in\mathcal{C}_{\mathrm{WFR}_{\alpha,\beta}}} \mathcal{J}(\bfm,\rho,H) = \int_0^1 \int_{X} J(\bfm_t(\bx),\rho_t(\bx), H_t(\bx)) d\bx dt
\end{align}
where the feasibility set $\mathcal{C}_{\mathrm{WFR}_{\alpha,\beta}}$, with a slight abuse of notation, is defined as
\begin{align}\label{eqn:SyncOT_momentum_constraint}
  \mathcal{C}_{\mathrm{WFR}_{\alpha,\beta}}: = \left\{ (\bfm,\rho, H): \partial_t\rho_t + \alpha\nabla\cdot\bfm_t = \beta H_t, \ \rho_{0,1} = \bar{\rho}_{0,1},\ \mathbf{m}_t\cdot\mathbf{n}|_{\partial X}=0 \right\}.
\end{align}

With the discretization from last section, the Monge UnSyncOT (\ref{eqn:SyncOT_momentum}) is approximated by the following finite-dimensional convex problem over the staggered and centered grids:
\begin{align}\label{eqn:SyncOT_discretized}
  \min_{U_{\text{s}}\in \mathcal{D}_{\text{s}}, H_\mathrm{c}\in\mathcal{M}(\mathcal{G}_{\mathrm{c}})} \mathcal{J}(\mathcal{I}(U_{\text{s}}), H_\mathrm{c}) + \iota_{\mathcal{C}_{\mathrm{WFR},h}}(U_{\text{s}},H_\mathrm{c}),
\end{align}
where, with a slight abuse of notation, we take
\begin{align}\label{eqn:J_discretized}
  \mathcal{J}(U_{\text{c}}, H_{\text{c}}) = \sum_{(i,j,k)\in\mathcal{S}_{\text{c}}} J(m_{ijk}, n_{ijk}, \rho_{ijk}, H_{ijk}), \quad U_{\text{c}} = (\bfm_{\text{c}},\rho_{\text{c}}) \in \mathcal{D}_{\text{c}}, \  H_{\text{c}} \in \mathcal{M}(\mathcal{G}_{\text{c}}),
\end{align}
and $\iota_{\mathcal{C}_{\mathrm{WFR},h}}(U_{\text{s}}, H_{\text{c}} )$ is the indicator function
\begin{align}\label{eqn:C}
  \iota_{\mathcal{C}_{\mathrm{WFR},h}}(U_{\text{s}}, H_\mathrm{c}) = 
  \begin{cases}
     0, \ & \text{if\ } (U_{\text{s}}, H_{\text{c}}) \in \mathcal{C}_{\mathrm{WFR},h}, \\
     +\infty, \ & \text{otherwise}.
  \end{cases}
\end{align}
with $\mathcal{C}_{\mathrm{WFR},h}$ being defined as
\begin{align*}
  \mathcal{C}_{\mathrm{WFR},h} = \{ U_{\text{s}}\in\mathcal{D}_s, H_{\text{c}}\in \mathcal{M}(\mathcal{G}_{\text{c}}): \ \div(U_{\text{s}}) = H_{\text{c}},\  b(U_{\text{s}}) = b_0 \}.
\end{align*}

We adopt the Chambolle-Pock primal-dual algorithm \cite{chambolle2011first} to solve the discretized Monge UnSyncOT (\ref{eqn:SyncOT_discretized}). The Chambolle-Pock algorithm generates a sequence 
\[
(U_\mathrm{s}^{(\ell)}, \hat{U}_\mathrm{s}^{(\ell)}, U_\mathrm{c}^{(\ell)}, H_\mathrm{p}^{(\ell)}, \hat{H}_{\mathrm{p}}^{(\ell)}, H_{\mathrm{d}}^{(\ell)} ) \in \mathcal{D}_\text{s} \times \mathcal{D}_\text{s} \times \mathcal{D}_\text{c} \times \mathcal{M}(\mathcal{G}_{\mathrm{c}})\times \mathcal{M}(\mathcal{G}_{\mathrm{c}})\times \mathcal{M}(\mathcal{G}_{\mathrm{c}}),
\]
for primal variables $(U_{\mathrm{s}},\hat{U}_{\mathrm{s}},H_{\mathrm{p}},\hat{H}_{\mathrm{p}})$ and dual variables $(U_{\mathrm{c}},H_{\mathrm{d}})$, starting from initial 
$(U_\mathrm{s}^{(0)}, \hat{U}_\mathrm{s}^{(0)}, \\ U_\mathrm{c}^{(0)}, H_\mathrm{p}^{(0)}, \hat{H}_{\mathrm{p}}^{(0)}, H_{\mathrm{d}}^{(0)})$ 
by the following iterations
\begin{align}\label{eqn:PD_Iterations_2terms}
  \begin{cases}
    \vspace{0.05in}
    \begin{bmatrix*}[r]
    U_\mathrm{c}^{(\ell+1)}  \\
    H_{\mathrm{d}}^{(\ell+1)} 
    \end{bmatrix*}
    = \text{prox}_{\sigma\mathcal{J}^*} 
    \left(
    \begin{bmatrix*}[r]
    U_\mathrm{c}^{(\ell)}  \\
    H_{\mathrm{d}}^{(\ell)} 
    \end{bmatrix*}
    + \sigma 
    \begin{bmatrix*}[r]
    \mathcal{I}(U_\mathrm{s}^{(\ell)})  \\
    I(H_{\mathrm{p}}^{(\ell)} )
    \end{bmatrix*}
    \right), \\
    \vspace{0.05in}
    \begin{bmatrix*}[r]
    \hat{U}_\mathrm{s}^{(\ell+1)}  \\
    \hat{H}_{\mathrm{p}}^{(\ell+1)} 
    \end{bmatrix*}
    = \text{proj}_{\mathcal{C}_{\mathrm{WFR},h}} \left(
    \begin{bmatrix*}[r]
    \hat{U}_\mathrm{s}^{(\ell)}  \\
    \hat{H}_{\mathrm{p}}^{(\ell)} 
    \end{bmatrix*}
    - \tau 
    \begin{bmatrix*}[r]
    \mathcal{I}^*(U_\mathrm{c}^{(\ell+1)})  \\
    I^*(H_{\mathrm{d}}^{(\ell+1)}) 
    \end{bmatrix*}
    \right), \\
    \begin{bmatrix*}[r]
    U_\mathrm{s}^{(\ell+1)}  \\
    H_{\mathrm{p}}^{(\ell+1)} 
    \end{bmatrix*} = 
    \begin{bmatrix*}[r]
    \hat{U}_\mathrm{s}^{(\ell+1)}  \\
    \hat{H}_{\mathrm{p}}^{(\ell+1)} 
    \end{bmatrix*}
    + \theta \left(
    \begin{bmatrix*}[r]
    \hat{U}_\mathrm{s}^{(\ell+1)}  \\
    \hat{H}_{\mathrm{p}}^{(\ell+1)} 
    \end{bmatrix*}
    - 
    \begin{bmatrix*}[r]
    \hat{U}_\mathrm{s}^{(\ell)}  \\
    \hat{H}_{\mathrm{p}}^{(\ell)} 
    \end{bmatrix*}
    \right),
  \end{cases}
\end{align}
in which $I$ is the identity map, and $(\mathcal{J}^*, \mathcal{I}^*, I^*)$ are the corresponding adjoint operators.
Here $\text{prox}_{\sigma\mathcal{J}^*}$ is calculated as follows
\begin{align}\label{eqn:Moreau}
  \text{prox}_{\sigma\mathcal{J}^*}([U_{\text{c}},H_{\mathrm{d}}]^\mathrm{T}) = [U_{\text{c}},H_{\mathrm{d}}]^\mathrm{T} - \sigma \text{prox}_{\sigma^{-1}\mathcal{J}}(\sigma^{-1}[U_{\text{c}},H_{\mathrm{d}}]^\mathrm{T}).
\end{align}
The operators $\text{prox}_{c\mathcal{J}}$ and $\text{proj}_{\mathcal{C}_{\mathrm{WFR,h}}}$ are calculated according to Lemma \ref{lemma:iteration_proxJ} and Lemma \ref{lemma:iteration_projC} below.

Let $(U_{\text{s}}^*, H_{\mathrm{c}}^*)$ be the solution of (\ref{eqn:SyncOT_discretized}). The convergence $(U_{\text{s}}^{(\ell)}, H_{\mathrm{c}}^{(\ell)}) \rightarrow (U_{\text{s}}^*, H_{\mathrm{c}}^*)$ is guaranteed \cite{chambolle2011first} with $O(N^{-1}_{\mathrm{step}})$ rate of convergence for the primal-dual gap,  provided that $\theta = 1$ and $\tau\sigma \|[\mathcal{I},I]^\mathrm{T}\|^2 < 1$. This has been extended in \cite{banert2023chambolle} to hold when $\theta > \frac{1}{2}$ and $\tau\sigma\|[\mathcal{I},I]^\mathrm{T}\|^2\le \frac{4}{1+2\theta}$. Hereafter, we will fix $\theta = 1$ for the numerical experiments.

As the algorithm (\ref{eqn:PD_Iterations_2terms}) require computing the proximal operators for $\mathcal{J}$ and $\iota_{\mathcal{C}_{\mathrm{WFR},h}}$, in what follows, we show that the discrete cost $\mathcal{J}$ in (\ref{eqn:J_discretized}) and the indicator $\iota_{\mathcal{C}_{\mathrm{WFR},h}}$ defined by (\ref{eqn:C}) are proximal friendly.

\begin{lemma}[$\text{prox}_{c\mathcal{J}}$] \label{lemma:proxJ}
  For the separable discrete cost functional $\mathcal{J}$ defined in (\ref{eqn:J_discretized}) with $J$ given in \reff{eqn:J_linearlize}, the proximal operator $\mathrm{prox}_{c\mathcal{J}}$ can be computed component-wisely,
  \begin{align}
    \mathrm{prox}_{c\mathcal{J}}(U_{\mathrm{c}}, H_\mathrm{c}) = \Big\{\mathrm{prox}_{c J}(m_{ijk}, n_{ijk}, \rho_{ijk}, H_{ijk}) \Big\}_{(i,j,k)\in\mathcal{S}_{\mathrm{c}}}.
  \end{align}
  Here for any $(\bfm,\rho,H)\in \mathbb{R}^2\times\mathbb{R}\times\mathbb{R}$,
  \begin{align}
    \mathrm{prox}_{c J}(\bfm,\rho, H) = 
    \begin{cases}
      \big(\bfm^*(\rho^*), \rho^*, H^*(\rho^*) \big), \quad &\mathrm{if\ } \rho>0, \\
      (\mathbf{0}, 0,0),  \quad  & \mathrm{otherwise},
    \end{cases} 
  \end{align}
  where $\rho^*$ is the unique fixed point over $(0,\infty)$ for $\phi(\tilde{\rho})$:
  \begin{align}\label{eqn:phi}
    \phi(\tilde{\rho}) = \rho +  c \big\langle (2c \alpha\mathbf{A} + \tilde{\rho} \mathbf{I})^{-1}\bfm, \alpha\mathbf{A} (2c \alpha\mathbf{A} + \tilde{\rho} \mathbf{I})^{-1}\bfm \big\rangle + c
    (2c\beta + \tilde{\rho})^{-2}H\beta H,
  \end{align}
  and 
  \begin{align}\label{eqn:m_H_proximal}
   & \bfm^*(\rho^*) = (2c \alpha\mathbf{A} + \rho^* \mathbf{I} )^{-1} \rho^*\bfm,\quad
    H^*(\rho^*)=(2c\beta + \rho^*)^{-1}\rho^*H.
  \end{align}

\end{lemma}

\begin{proof}
   Let $\text{prox}_{c J}(\bfm, \rho, H) = (\bfm^*, \rho^*, H^*)$. By the definition of the proximal operator,
   \[
   \text{prox}_{c J} (\bfm,\rho, H) = \underset{(\tilde{\bfm}, \tilde{\rho}, \tilde{H})}{\mathrm{argmin}} \ J( \tilde{\bfm}, \tilde{\rho}, \tilde{H} ) + (2c)^{-1}\| (\bfm,\rho, H) - (\tilde{\bfm},\tilde{\rho},\tilde{H}) \|^2,
   \]
and the definition of $J$ in \reff{eqn:J_linearlize}, it is evident that if $\rho\le0$, then $\text{prox}_{c J} (\bfm,\rho,H) = (\mathbf{0},0,0)$. Besides, if $\bfm = \mathbf{0}$, $\rho>0$, and $H=0$, then $\text{prox}_{c J} (\mathbf{0},\rho,0) = (\mathbf{0},\rho,0)$.

On the other hand, if $\rho>0$ and $(\bfm,H) \neq (\mathbf{0},0)$, taking the first order derivatives of the objective function for the proximal operator $\text{prox}_{c J} (\bfm,\rho,H)$, it yields
\begin{align}\label{eqn:proxJ_system}
  \begin{cases}
    2c \tilde{\rho}^{-1}( \alpha \mathbf{A} \tilde{\bfm}) + \tilde{\bfm} - \bfm = 0, \\
    -c \tilde{\rho}^{-2}(\tilde{\bfm}^\mathrm{T} \alpha\mathbf{A} \tilde{\bfm}) - c  \tilde{\rho}^{-2}(\tilde{H}\beta\tilde{H}) +  \tilde{\rho} - \rho = 0,\\
    2c\tilde{\rho}^{-1}(\beta\tilde{H}) + \tilde{H}-H=0.
  \end{cases}
\end{align} 
The first and third equations in (\ref{eqn:proxJ_system}) lead to
\begin{align}\label{eqn:m_tilde}
  \tilde{\bfm} = (2c \alpha\mathbf{A} + \tilde{\rho} \mathbf{I} )^{-1} \tilde{\rho}\bfm, \quad 
  \tilde{H} = (2c\beta + \tilde{\rho})^{-1}\tilde{\rho}H.
\end{align}
Inserting the above equations to the second equation in (\ref{eqn:proxJ_system}) gives
\begin{align}\label{eqn:rho_tilde}
\tilde{\rho}& = \rho + c \frac{\langle \tilde{\bfm}, \alpha\mathbf{A}\tilde{\bfm} \rangle}{\tilde{\rho}^2} + c\frac{\tilde{H}\beta\tilde{H}}{\tilde{\rho}^2} \nonumber\\
&= \rho + c \Big\langle (2c \alpha\mathbf{A} + \tilde{\rho} \mathbf{I})^{-1}\bfm, \alpha\mathbf{A} (2c \alpha\mathbf{A} + \tilde{\rho} \mathbf{I})^{-1}\bfm \Big\rangle  + c
    (2c\beta + \tilde{\rho})^{-2}H\beta H.
\end{align}
Let $\phi(\tilde{\rho})$ be the right hand side of the above equation, we assert that $\phi'(\tilde{\rho})<0$ for $(\bfm,H)\neq (\mathbf{0},0)$. Indeed, by noting $\mathbf{A}$ is positive definite, and taking the derivative, we have 
\begin{align*}
  \phi'(\tilde{\rho}) &= c \Big\langle 2\alpha\mathbf{A}(2c \alpha\mathbf{A} + \tilde{\rho}\mathbf{I})^{-1}\bfm, \ \frac{d}{d\tilde{\rho}}(2c \alpha\mathbf{A} + \tilde{\rho}\mathbf{I})^{-1}\bfm \Big\rangle  + \frac{d}{d\tilde{\rho}}c
    (2c\beta + \tilde{\rho})^{-2}H\beta H \\
  &= -c \Big\langle 2\alpha\mathbf{A}(2c \alpha\mathbf{A} + \tilde{\rho}\mathbf{I})^{-1}\bfm, \ (2c \alpha\mathbf{A} + \tilde{\rho}\mathbf{I})^{-2}\bfm \Big\rangle-2c(2c\beta + \tilde{
  \rho})^{-3}H\beta H <0,
\end{align*}
for $(\bfm,H) \neq (\mathbf{0},0)$, $\alpha,\beta > 0$. The left hand side of the equation (\ref{eqn:rho_tilde}) is a linearly increasing function over the interval $\tilde{\rho}\in (0,\infty)$, beginning its growth from $0$; in contrast, the right hand side is a strictly decreasing function over the same interval, starting from $\phi(0)=\rho+\frac{1}{4c}\langle \bfm, (\alpha A)^{-1}\bfm \rangle + \frac{1}{4c}H\beta^{-1}H > 0 $, therefore there exists unique solution $\tilde{\rho} = \rho^*$ over $(0,\infty)$. Inserting the unique $\rho^*$ back to (\ref{eqn:m_tilde}) leads to the result (\ref{eqn:m_H_proximal}).
\end{proof}

Note that Lemma \ref{lemma:proxJ} provides a natural fixed point iteration algorithm to compute $\rho^*$. We summarize it in the following lemma.
\begin{lemma}\label{lemma:iteration_proxJ}
Taking any initial guess $\tilde{\rho}^{(0)} \in (0, \infty)$, the following fixed point iteration
\begin{align}
  &\tilde{\rho}^{(\ell+1)} = \rho + c \Big\langle \big(2c \alpha\mathbf{A} + \tilde{\rho}^{(\ell)} \mathbf{I} \big)^{-1}\bfm, \alpha\mathbf{A} \big(2c \alpha\mathbf{A} + \tilde{\rho}^{(\ell)} \mathbf{I} \big)^{-1}\bfm \Big\rangle + c
    \left(2c\beta + \tilde{\rho}^{(\ell)}\right)^{-2}H\beta H, 
\end{align}
for $l=0,1,2,\cdots$, is convergent for sufficiently large $c$.

\end{lemma}

\begin{proof}
Note that $\phi(\cdot)$ is monotonically decreasing and strictly convex, so we have that $\phi'(0) < \phi'(\tilde{\rho}) < 0$ for any $\tilde{\rho}^{(0)} \in (0, \infty)$. On the other hand, since
\[
  \phi'(0) = -\frac{1}{4c^2}\bfm^\mathrm{T} (\alpha\mathbf{A})^{-2}\bfm-\frac{1}{4c^2}H\beta^{-2}H,
\]
one can take sufficiently large $c$ such that $\phi'(0)>-1$. Then the condition $-1 < \phi'(\tilde{\rho}) \le 0$ guarantees that the fixed point iteration is always convergent.
\end{proof}

The computation of the proximal operator for $\iota_{\mathcal{C}_{\mathrm{WFR},h}}$ follows the results in \cite{papadakis2014optimal, cang2025synchronized}. Note that $\text{prox}_{\iota_{\mathcal{C}_{\mathrm{WFR},h}}}= \text{proj}_{\mathcal{C}_{\mathrm{WFR},h}}$ with $\mathcal{C}_{\mathrm{WFR},h}$ defined in (\ref{eqn:C}). Taking a similar notation as in \cite{papadakis2014optimal,cang2025synchronized},
\[
  A( U_{\text{s}}, H_{\text{c}}) =(\alpha\text{div}(U_{\text{s}})- \beta H_{\text{c}}, b(U_{\text{s}})), \quad y = (\mathbf{0}, b_0),
\]
then the convex set $\mathcal{C}_{\mathrm{WFR},h}$ can be recast into
\begin{align}\label{eqn:C02}
  \mathcal{C}_{\mathrm{WFR},h} = \{ (U_{\text{s}}, H_{\text{c}}) = (\bfm_{\text{s}},\rho_{\text{s}}, H_{\text{c}}   ) \in\mathcal{D}_s \times \mathcal{M}(\mathcal{G}_{\mathrm{c}}): \ A (U_{\text{s}}, H_{\text{c}})= y \}.
\end{align}
The projection $\text{proj}_{\mathcal{C}_{\mathrm{WFR},h}}$ can be computed directly by the following formula
\[
  \text{proj}_{\mathcal{C}_{\mathrm{WFR},h}}( U_{\text{s}}, H_{\text{c}} ) = (U_{\text{s}}, H_{\text{c}}) - A^*(AA^*)^{-1}(A(U_{\text{s}}, H_{\text{c}}) - y). 
\]
We summarize it in the following lemma \cite{papadakis2014optimal}.

\begin{lemma}[$\text{proj}_{\mathcal{C}_{\mathrm{WFR},h}}$]\label{lemma:iteration_projC}
  The proximal operator $\mathrm{proj}_{\mathcal{C}_{\mathrm{WFR},h}}$ can be computed by 
\begin{align}\label{eqn:proj_update}
 ( U_{\text{s}}^{(\ell+1)}, H_{\text{c}}^{(\ell+1)}) = \mathrm{proj}_{\mathcal{C}_{\mathrm{WFR},h}}( U_{\text{s}}^{(\ell)}, H_{\text{c}}^{(\ell)} ) = (U_{\text{s}}^{(\ell)}, H_{\text{c}}^{(\ell)}) - A^*(AA^*)^{-1}(A(U_{\text{s}}^{(\ell)}, H_{\text{c}}^{(\ell)}) -y). 
\end{align}
Equivalently, the iteration can be performed by firstly solving the inhomogeneous Helmholtz equation
\begin{align}\label{poissonSim}
   \left( \beta^2 I - \alpha^2 (D_x^2 + D_y^2 + D_t^2) \right) s_{ijk} = \alpha \mathrm{div}(U^{(\ell)}_{\text{s}})- \beta H^{(\ell)}_\mathrm{c}, \quad s_{ijk} \in \mathcal{M}(\mathcal{G}_{\text{c}})
\end{align}
with homogeneous Neumann boundary condition where $D_q^2, q=x,y,t$ represents the 2nd-order central difference along $q$-direction. This can be achieved by the fast discrete cosine transform in $O(MNQ\log(MNQ))$ operations. Then the update $(U_{\text{s}}^{(\ell+1)},H_{\text{c}}^{(\ell+1)}) \leftarrow (U_{\text{s}}^{(\ell)},H_{\text{c}}^{(\ell)})$ is achieved by
\begin{align}\label{updateHelmholtz}
\begin{cases}
  \vspace{0.1in}
  m^{(\ell+1)}_{i-{\scriptscriptstyle\frac{1}{2}},j,k} &= m^{(\ell)}_{i-{\scriptscriptstyle\frac{1}{2}},j,k} + \frac{\alpha}{\Delta x} (s_{i,j,k} - s_{i-1,j,k}), \quad (i,j,k)\in \mathcal{S}_{\text{s}}^{x}, \ i \neq 0, M ; \\ \vspace{0.1in}
  n^{(\ell+1)}_{i, j-{\scriptscriptstyle\frac{1}{2}},k} &= n^{(\ell)}_{i, j-{\scriptscriptstyle\frac{1}{2}},k} + \frac{\alpha}{\Delta y} (s_{i,j,k} - s_{i,j-1,k}), \quad (i,j,k)\in \mathcal{S}_{\text{s}}^{y}, \ j \neq 0, N ; \\ 
  \vspace{0.1in}
  \rho^{(\ell+1)}_{i, j, k-{\scriptscriptstyle\frac{1}{2}} } &= \ \rho^{(\ell)}_{i, j, k-{\scriptscriptstyle\frac{1}{2}} } + \frac{\alpha}{\Delta t} (s_{i,j,k} - s_{i,j,k-1}), \quad (i,j,k)\in \mathcal{S}_{\text{s}}^{t}, \ k \neq 0, Q ;
  \\  
  H^{(\ell+1)}_{i, j, k} &= H^{(\ell)}_{i, j, k} +  \beta s_{i,j,k}, \quad (i,j,k)\in \mathcal{S}_{\text{c}}.
\end{cases}  
\end{align}

\end{lemma}

\begin{proof}
Given 
\begin{align}\label{eqn:feasibility01}
  A( U_{\text{s}}, H_{\text{c}}) =(\alpha \text{div}(U_{\text{s}}) - \beta H_{\text{c}}, b(U_{\text{s}})) = y :=  (\mathbf{0}, b_0) 
  = (\mathbf{0},(\mathbf{0},\mathbf{0},\mathbf{0},\mathbf{0}, \mu_{\mathrm{c}},\nu_{\mathrm{c}})),
\end{align}
we revise the definition of $A( U_{\text{s}}, H_{\text{c}})$ by modifying the boundary operator $b(U_{\mathrm{s}})$ in (\ref{eqn:b_Us}) as below
\begin{align*}
    b(U_{\mathrm{s}}) = 
    \Big(\tfrac{\alpha}{\Delta x} m_{ {\scriptscriptstyle -\frac{1}{2}} , j,k}, 
    -\tfrac{\alpha}{\Delta x} m_{ {\scriptscriptstyle M-\frac{1}{2}} , j,k},
    \tfrac{\alpha}{\Delta y} n_{ i,{\scriptscriptstyle -\frac{1}{2}} , k},
    -\tfrac{\alpha}{\Delta y} n_{ i,{\scriptscriptstyle N-\frac{1}{2}} , k},
    \tfrac{\alpha}{\Delta t} \rho_{ i,j,{\scriptscriptstyle -\frac{1}{2}}},  
    -\tfrac{\alpha}{\Delta t} \rho_{ i,j,{\scriptscriptstyle Q-\frac{1}{2}}}
    \Big)_{(i,j,k)\in \mathcal{S}_{\mathrm{c}}}.
\end{align*}
We also modify the boundary tensor $b_0$ in (\ref{eqn:boundarycondition}) as below
\begin{align*}
    b_0 = \left(\mathbf{0},\mathbf{0},\mathbf{0},\mathbf{0}, \tfrac{\alpha}{\Delta t}\mu_{\mathrm{c}},-\tfrac{\alpha}{\Delta t}\nu_{\mathrm{c}}\right), \quad y = (\mathbf{0},b_0).
\end{align*}
For simplicity, we will continue to use the same notation $A, b, b_0,y$ for the revised operators. Note that under the revised definitions, the boundary condition
\begin{align}
    b(U_\mathrm{s}) =  b_0,
\end{align}
remains unchanged from original form given in equation (\ref{eqn:boundarycondition}).

Assume that we have
\[
AA^* s = A(U_{\mathrm{s}},H_{\mathrm{c}}) - y, \quad s = (s_{ijk}), \ i=-1:M, \ j = -1:N, \ k = -1:Q.
\]
Now we derive the explicit form of $A^*$. Given that
\begin{align*}
    \Big\langle s, A(U_\mathrm{s},H_{\mathrm{c}}) \Big\rangle
    =&
    \sum_{(i,j,k)\in\mathcal{S}_{\mathrm{c}}}
    s_{ijk}\Big(\alpha\mathrm{div}(U_{\mathrm{s}})_{ijk} - \beta H_{ijk}\Big) + s_{-1,j,k} \cdot \tfrac{\alpha}{\Delta x} m_{ {\scriptscriptstyle -\frac{1}{2}} , j,k}
    - s_{M,j,k} \cdot \tfrac{\alpha}{\Delta x} m_{ {\scriptscriptstyle M-\frac{1}{2}} , j,k} \\
    &+ s_{i,-1,k} \cdot \tfrac{\alpha}{\Delta y} n_{ i, {\scriptscriptstyle -\frac{1}{2}} , k}
    - s_{i,N,k} \cdot \tfrac{\alpha}{\Delta y} n_{ i, {\scriptscriptstyle N-\frac{1}{2}},k}
    + s_{i,j,-1} \cdot \tfrac{\alpha}{\Delta t} \rho_{ i, j,{\scriptscriptstyle -\frac{1}{2}}}
    - s_{i,j,Q} \cdot \tfrac{\alpha}{\Delta t} \rho_{ i, j,  {\scriptscriptstyle Q-\frac{1}{2}}} \\
    =&\ 
    \alpha\sum_{(i,j,k)\in\mathcal{S}^x_{\mathrm{s}}}
     -\frac{s_{i,j,k}-s_{i-1,j,k}}{\Delta x}m_{{\scriptscriptstyle i-\frac{1}{2}},j,k}
    +\alpha\sum_{(i,j,k)\in\mathcal{S}^y_{\mathrm{s}}}
    -\frac{s_{i,j,k}-s_{i,j-1,k}}{\Delta y}n_{i,{\scriptscriptstyle j-\frac{1}{2}},k} \\
    &+\alpha\sum_{(i,j,k)\in\mathcal{S}^t_{\mathrm{s}}}
    -\frac{s_{i,j,k}-s_{i,j,k-1}}{\Delta t}m_{i,j,{\scriptscriptstyle k-\frac{1}{2}}}
    +\beta\sum_{(i,j,k)\in\mathcal{S}_{\mathrm{c}}} - s_{ijk}H_{ijk}   \\
    =& \ \Big\langle A^*s, (U_\mathrm{s},H_{\mathrm{c}}) \Big\rangle,
\end{align*}
we obtain the adjoint operator $A^*$
\begin{align}\label{eqn:Astar}
    A^*s = -\bigg[\alpha\Big(\frac{s_{i,j,k}-s_{i-1,j,k}}{\Delta x}\Big)_{\mathcal{S}^x_{\mathrm{s}}}, 
    \alpha\Big(\frac{s_{i,j,k}-s_{i,j-1,k}}{\Delta y}\Big)_{\mathcal{S}^y_{\mathrm{s}}},
    \alpha\Big(\frac{s_{i,j,k}-s_{i,j,k-1}}{\Delta t}\Big)_{\mathcal{S}^t_{\mathrm{s}}}, \beta\big(s_{ijk}\big)_{\mathcal{S}_{\mathrm{c}}} \bigg].
\end{align}
Then it follows that
\begin{align*}
    AA^*s = \Big[& \left(\beta^2 I - \alpha^2(D_x^2+D_y^2+D_t^2)\right)s_{ijk},
    -\frac{s_{0,j,k}-s_{-1,j,k}}{\alpha^{-2}\Delta x^2},
    -\frac{s_{M,j,k}-s_{M-1,j,k}}{\alpha^{-2}\Delta x^2}, \\
    &-\frac{s_{i,0,k}-s_{i,-1,k}}{\alpha^{-2} \Delta y^2},
    -\frac{s_{i,N,k}-s_{i,N-1,k}}{\alpha^{-2} \Delta y^2},
    - \frac{s_{i,j,0}-s_{i,j,-1}}{\alpha^{-2}\Delta t^2},
    - \frac{s_{i,j,Q}-s_{i,j,Q-1}}{ \alpha^{-2}\Delta t^2} \Big]_{(i,j,k)\in\mathcal{S}_{\mathrm{c}}},
\end{align*}
where the 2nd-order central difference $D_x^2$ (similarly $D_y^2, D_t^2$) is defined as 
\begin{align*}
    D_x^2 = \frac{s_{i-1,j,k}-2s_{i,j,k}+s_{i+1,j,k}}{\Delta x^2}.
\end{align*}

Now let $AA^*s = A(U_{\mathrm{s}}^{(\ell)},H_{\mathrm{c}}^{(\ell)})-y$, we have
\begin{align}\label{eqn:}
    \begin{cases}
        \left( \beta^2 I- \alpha^2 (D_x^2+D_y^2+D_t^2) \right) s_{ijk} = \alpha\mathrm{div}(U_{\mathrm{s}}^{(\ell)})_{ijk} - \beta H^{(\ell)}_{ijk}, \quad (i,j,k)\in\mathcal{S}_{\mathrm{c}},\\
        s_{0,j,k} = s_{-1,j,k}, \quad 
        s_{M,j,k} = s_{M-1,j,k}, \\
        s_{i,0,k} = s_{i,-1,k}, \quad
        s_{i,N,k} = s_{i,N-1,k}, \\
        s_{i,j,0} = s_{i,j,-1}, \quad
        s_{i,j,Q} = s_{i,j,Q-1}.
    \end{cases}
\end{align}
After solving for $s$, we get the update $(U_{\mathrm{s}}^{(\ell+1)}, H_{\mathrm{c}}^{(\ell+1)}) = (U_{\mathrm{s}}^{(\ell)}, H_{\mathrm{c}}^{(\ell)}) - A^*s$. By the definition of $U_{\mathrm{s}}^{(\ell)} = (m_{\mathrm{s}}^{(\ell)},n_{\mathrm{s}}^{(\ell)},\rho_{\mathrm{s}}^{(\ell)})$ and the definition of $A^*s$ in (\ref{eqn:Astar}), we have that
\begin{align*}
\begin{cases}
  \vspace{0.1in}
  m^{(\ell+1)}_{i-{\scriptscriptstyle\frac{1}{2}},j,k} &= m^{(\ell)}_{i-{\scriptscriptstyle\frac{1}{2}},j,k} + \frac{\alpha}{\Delta x} (s_{i,j,k} - s_{i-1,j,k}), \quad (i,j,k)\in \mathcal{S}_{\text{s}}^{x}, \ i \neq 0, M ; \\ \vspace{0.1in}
  n^{(\ell+1)}_{i, j-{\scriptscriptstyle\frac{1}{2}},k} &= n^{(\ell)}_{i, j-{\scriptscriptstyle\frac{1}{2}},k} + \frac{\alpha}{\Delta y} (s_{i,j,k} - s_{i,j-1,k}), \quad (i,j,k)\in \mathcal{S}_{\text{s}}^{y}, \ j \neq 0, N ; \\ 
  \vspace{0.1in}
  \rho^{(\ell+1)}_{i, j, k-{\scriptscriptstyle\frac{1}{2}} } &= \ \rho^{(\ell)}_{i, j, k-{\scriptscriptstyle\frac{1}{2}} } + \frac{\alpha}{\Delta t} (s_{i,j,k} - s_{i,j,k-1}), \quad (i,j,k)\in \mathcal{S}_{\text{s}}^{t}, \ k \neq 0, Q ;
  \\  
  H^{(\ell+1)}_{i, j, k} &= H^{(\ell)}_{i, j, k} +  \beta s_{i,j,k}, \quad (i,j,k)\in \mathcal{S}_{\text{c}}.
\end{cases}  
\end{align*}
and at the boundary points, we always have
\begin{align*}
    &\Big[
    \tfrac{\alpha}{\Delta x} m^{(\ell)}_{ {\scriptscriptstyle -\frac{1}{2}} , j,k}, 
    -\tfrac{\alpha}{\Delta x} m^{(\ell)}_{ {\scriptscriptstyle M-\frac{1}{2}} , j,k},
    \tfrac{\alpha}{\Delta y} n^{(\ell)}_{ i,{\scriptscriptstyle -\frac{1}{2}} , k},
    -\tfrac{\alpha}{\Delta y} n^{(\ell)}_{ i,{\scriptscriptstyle N-\frac{1}{2}} , k},
    \tfrac{\alpha}{\Delta t} \rho^{(\ell)}_{ i,j,{\scriptscriptstyle -\frac{1}{2}}},  
    -\tfrac{\alpha}{\Delta t} \rho^{(\ell)}_{ i,j,{\scriptscriptstyle Q-\frac{1}{2}}}
    \Big]_{(i,j,k)\in\mathcal{S}_{\mathrm{c}}} \\
    =&
    \Big[0_{jk},0_{jk},0_{ik},0_{ik},\tfrac{\alpha}{\Delta t}\mu_{ij},-\tfrac{\alpha}{\Delta t}\nu_{ij}\Big]_{(i,j,k)\in\mathcal{S}_{\mathrm{c}}},
\end{align*}
namely,
\begin{align*}
    m^{(\ell)}_{ {\scriptscriptstyle -\frac{1}{2}} , j,k} = m^{(\ell)}_{ M{\scriptscriptstyle -\frac{1}{2}} , j,k} = 0,
    \ 
    n^{(\ell)}_{ i, {\scriptscriptstyle -\frac{1}{2}},k} = n^{(\ell)}_{ i, N{\scriptscriptstyle -\frac{1}{2}},k} = 0,
    \ 
    \rho^{(\ell)}_{ i,j,{\scriptscriptstyle -\frac{1}{2}}}
    = \mu_{ij},
    \ 
    \rho^{(\ell)}_{ i,j,Q{\scriptscriptstyle -\frac{1}{2}}}
    = \nu_{ij},
    \quad \forall\ l = 0,1,2,\cdots.
\end{align*}
Finally, we get the desired update for $(U_{\mathrm{s}}^{(\ell)}, H_{\mathrm{c}}^{(\ell)})$.
\end{proof}

\subsection{Primal-dual Methods for Kantorovich Unbalanced Synchronized Optimal Transport }\label{subsection:Primal_dual_Kantorovich}

In this section, we will take the approximate Kantorovich UnSyncOT formulation \reff{eqn:KantorovichSynOT_approximate} for the numerical implementation. In such way, the secondary action will be simply be a sum of HK distances, which can be effectively solved by the Sinkhorn algorithm \cite{cuturi2013sinkhorn}, instead of solving for $\Phi$ in \reff{eqn:KantorovichUnSynOT_onespace} at each iteration.

Using the discretization in Section \ref{subsection:Staggered}, the approximate Kantorovich UnSyncOT problem \reff{eqn:KantorovichSynOT_approximate} is approximated by the following finite-dimensional convex problem over the staggered and centered grids:
\begin{align}\label{eqn:SyncOT_discretized_Kantorovich}
  \min_{U_{\text{s}}\in \mathcal{D}_{\text{s}}, H_\mathrm{c}\in\mathcal{M}(\mathcal{G}_{\mathrm{c}})} \mathcal{J}(\mathcal{I}(U_{\text{s}}), H_\mathrm{c}) + \iota_{\mathcal{C}_{\mathrm{WFR},h}}(U_{\text{s}},H_\mathrm{c}) + \frac{c_2}{c_1} \mathcal{H}(\rho_{\mathrm{s}})
\end{align}
where $\mathcal{J}$ is defined in (\ref{eqn:J_discretized}) in which $J$ is defined in \reff{eqn:J_linearlize} but with $\mathbf{A} = \mathbf{I}$, the set of constraint $\mathcal{C}_{\mathrm{WFR},h}$ is given in (\ref{eqn:C}), and $\mathcal{H}$ is defined as follows,
\begin{align}\label{eqn: mathcalH}
  \mathcal{H}(U_{\mathrm{s}}) = \mathcal{H}(\rho_{\text{s}}) = \sum_{k=0}^{Q-1} \frac{1}{\Delta t_k} \mathrm{HK}^2\left( \mathcal{T}_{\mathrm{K},h}\big(\rho_{i,j,k-\frac{1}{2}}\big)_{ij}, \mathcal{T}_{\mathrm{K},h}\big(\rho_{i,j,k+\frac{1}{2}}\big)_{ij}\right).
\end{align}
Here $\mathcal{T}_{\mathrm{K},h}$ is the discretization of $\TK$ \reff{eqn:TK}, namely, $\mathcal{T}_{\mathrm{K},h}\big(\rho_{i,j,k-\frac{1}{2}}\big)_{ij} = (\Pi_{ij})_{ij}(\rho_{i,j,k-\frac{1}{2}}\big)_{ij}$ with $\Pi$ being the stochastic matrix from the discretization of the Markov kernel $\pi$ \reff{eqn:Markov_kernel}.

There are several optimizers for solving \reff{eqn:SyncOT_discretized_Kantorovich} whose objective function is a sum of three terms, including:  Condat-Vu algorithm \cite{condat2013primal, vu2013splitting} as a generalization of the Chambolle-Pock algorithm \cite{chambolle2011first}; a Primal-Dual Fixed-Point algorithm (PDFP) \cite{chen2016primal}, in which two proximal mappings are computed in each iteration; Yan algorithm in \cite{yan2018new}, which has the same regions of acceptable parameters with PDFP and the same per-iteration complexity as Condat-Vu. Among these three algorithms, Yan algorithm has the same per-iteration complexity as the Condat-Vu algorithm (both being faster than PDFP), and shares the same range of acceptable parameters as PDFP (both wider than that of Condat-Vu). Therefore, we adopt Yan algorithm for solving \reff{eqn:SyncOT_discretized_Kantorovich}.

The Yan algorithm generates a sequence 
\[
(U_\mathrm{s}^{(\ell)}, \hat{U}_\mathrm{s}^{(\ell)}, U_\mathrm{c}^{(\ell)}, H_\mathrm{p}^{(\ell)}, \hat{H}_{\mathrm{p}}^{(\ell)}, H_{\mathrm{d}}^{(\ell)} ) \in \mathcal{D}_\text{s} \times \mathcal{D}_\text{s} \times \mathcal{D}_\text{c} \times \mathcal{M}(\mathcal{G}_{\mathrm{c}})\times \mathcal{M}(\mathcal{G}_{\mathrm{c}})\times \mathcal{M}(\mathcal{G}_{\mathrm{c}}),
\]
for primal variables $(U_{\mathrm{s}},\hat{U}_{\mathrm{s}},H_{\mathrm{p}},\hat{H}_{\mathrm{p}})$ and dual variables $(U_{\mathrm{c}},H_{\mathrm{d}})$, starting from initial 
$(U_\mathrm{s}^{(0)}, \hat{U}_\mathrm{s}^{(0)}, \\ U_\mathrm{c}^{(0)}, H_\mathrm{p}^{(0)}, \hat{H}_{\mathrm{p}}^{(0)}, H_{\mathrm{d}}^{(0)})$ 
by the following iterations
\begin{align}\label{eqn:PD_Iterations_3terms}
  \begin{cases}
    \vspace{0.05in}
    \begin{bmatrix*}[r]
    U_\mathrm{c}^{(\ell+1)}  \\
    H_{\mathrm{d}}^{(\ell+1)} 
    \end{bmatrix*}
    = \text{prox}_{\sigma\mathcal{J}^*} 
    \left(
    \begin{bmatrix*}[r]
    U_\mathrm{c}^{(\ell)}  \\
    H_{\mathrm{d}}^{(\ell)} 
    \end{bmatrix*}
    + \sigma 
    \begin{bmatrix*}[r]
    \mathcal{I}(U_\mathrm{s}^{(\ell)})  \\
    I(H_{\mathrm{p}}^{(\ell)} )
    \end{bmatrix*}
    \right), \\
    \vspace{0.05in}
    \begin{bmatrix*}[r]
    \hat{U}_\mathrm{s}^{(\ell+1)}  \\
    \hat{H}_{\mathrm{p}}^{(\ell+1)} 
    \end{bmatrix*}
    = \text{proj}_{\mathcal{C}_{\mathrm{WFR},h}} \left(
    \begin{bmatrix*}[r]
    \hat{U}_\mathrm{s}^{(\ell)}  \\
    \hat{H}_{\mathrm{p}}^{(\ell)} 
    \end{bmatrix*}
    - \tau 
    \begin{bmatrix*}[r]
    \mathcal{I}^*(U_\mathrm{c}^{(\ell+1)})  \\
    I^*(H_{\mathrm{d}}^{(\ell+1)}) 
    \end{bmatrix*}
    -\frac{c_2}{c_1}\tau
    \begin{bmatrix*}[r]
    \nabla\mathcal{H}(U_{\mathrm{s}}^{(\ell)})  \\
    \mathbf{0}
    \end{bmatrix*}
    \right), \\
    \begin{bmatrix*}[r]
    U_\mathrm{s}^{(\ell+1)}  \\
    H_{\mathrm{p}}^{(\ell+1)} 
    \end{bmatrix*} = 
    2
    \begin{bmatrix*}[r]
    \hat{U}_\mathrm{s}^{(\ell+1)}  \\
    \hat{H}_{\mathrm{p}}^{(\ell+1)} 
    \end{bmatrix*}
    - 
    \begin{bmatrix*}[r]
    \hat{U}_\mathrm{s}^{(\ell)}  \\
    \hat{H}_{\mathrm{p}}^{(\ell)} 
    \end{bmatrix*}
    +\frac{c_2}{c_1}\tau
    \begin{bmatrix*}[r]
    \nabla\mathcal{H}(\hat{U}_{\mathrm{s}}^{(\ell)})  \\
    \mathbf{0}
    \end{bmatrix*}
    -\frac{c_2}{c_1}\tau
    \begin{bmatrix*}[r]
    \nabla\mathcal{H}(\hat{U}_{\mathrm{s}}^{(\ell+1)})  \\
    \mathbf{0}
    \end{bmatrix*}.
  \end{cases}
\end{align}
Similar to the discussion in Section \ref{subsection:Primal_dual_Monge}, $\text{prox}_{\tau\mathcal{J}^*}$ is calculated by Moreau identity (\ref{eqn:Moreau}). Furthermore, $\text{prox}_{\tau\mathcal{J}}$ and $\text{proj}_{\mathcal{C}_{\mathrm{WFR},h}}$ are calculated according to Lemma \ref{lemma:iteration_proxJ} and Lemma \ref{lemma:iteration_projC}, respectively. Most importantly, $\nabla \mathcal{H}(U_{\text{s}
})$ is computed via auto-differentiation.

The condition under which Yan algorithm is convergence is discussed in detail in \cite{yan2018new}. One can also refer to \cite{cang2025synchronized} for the discussion of Yan algorithm when applying to balanced Kantorovich SyncOT.

\section{Numerical Experiments}\label{expr}

\subsection{1D examples}
We first validate our solver by solving the 1D WFR problem \reff{eqn:WFR}. In 1D experiments, marginal distributions are generated by truncated normalized Gaussian distributions on the interval $[0,1]$ 
\begin{align*}
    N(x; \mu, \sigma) = \frac{p(x; \mu,\sigma)}{\int_{[0,1]} p(x;\mu, \sigma)dx}
\end{align*}
where $p = e^{-\frac{(x-\mu)^2}{2\sigma^2}}$.
Suppose $\rho^{\ast}$ is the optimal solution of ${\mathrm{WFR}^2_{\alpha,\beta}}(\bar{\rho}_0,\bar{\rho}_1)$.  
When $\alpha/\beta = 1/4$, it is known \cite{liero2016optimal} that the total mass $m(t) = \int_X \rho^{\ast}_t(\bx) d\bf{x}$ is equal to 
\begin{align}
    m(t) = (1-t)m(0)+tm(1)-t(1-t)\mathrm{HK}_{1,4}(\bar{\rho}_0,\bar{\rho}_1).
    \label{mass_curve}
\end{align} 
We take $\bar{\rho}_0=N(\cdot;0.3, 0.05)$ and $\bar{\rho}_1=N(\cdot;0.7, 0.05)$ or $\bar{\rho}_1 = 1.2N(\cdot;0.7, 0.05)$ and show that the solutions of our solver reproduce the theoretical $m(t)$  (see Figure \ref{fig:mass}). Here when determining $m(t)$ through equation \reff{mass_curve}, we adopt the solver \texttt{ot.unbalanced.lbfgsb\_unbalanced2} in the Python Optimal Transport package \cite{flamary2021pot} to solve for $\mathrm{HK}_{1,4}(\bar{\rho}_0,\bar{\rho}_1)$ with a sufficiently fine grid resoution $M = 1024$ over $[0,1]$. For each pair of marginal distributions, the difference between the squared HK distance and the WFR loss given by our solver is less than 0.01.
\begin{figure}
    \centering
    \includegraphics[width=0.8\linewidth]{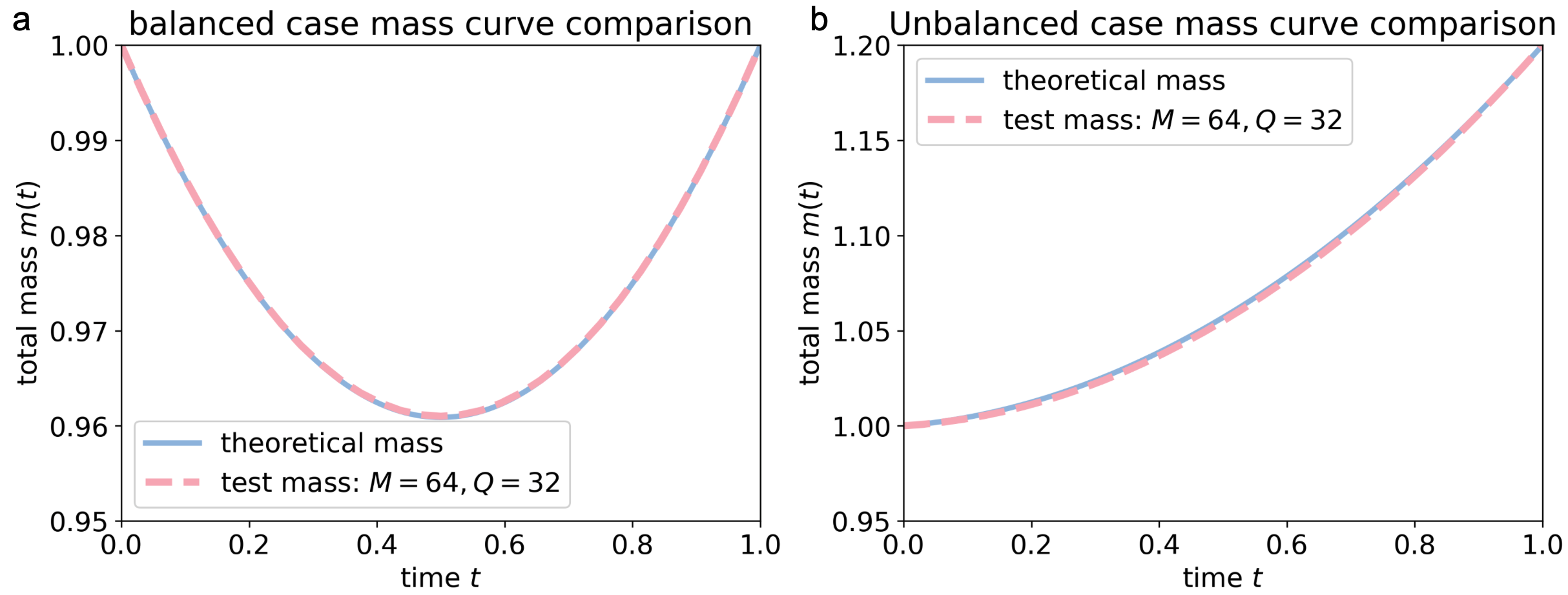}
    \caption{a) Mass curve comparison when $\bar{\rho}_0=N(\cdot;0.3, 0.05)$ and $\bar{\rho}_1=N(\cdot;0.7, 0.05)$. b) Mass curve comparison when $\bar{\rho}_0=N(\cdot;0.3, 0.05)$ and $\bar{\rho}_1=1.2N(\cdot;0.7, 0.05)$. $M = 64$ and $Q=32$ for both cases. Recall that for $X=[0,1]^2$, the grid sizes are $\Delta_x = 1/M, \Delta_y = 1/N, \Delta_t=1/Q$. If $X=[0,1]$, the grid is only determined by $M$ and $Q$. }
    \label{fig:mass}
\end{figure}

\subsection{2D unbalanced synchronized examples}

In this section we first illustrate Monge UnSyncOT with two examples where $X=[0,1]^2$ and $Y=\bT(X)$ is an embedding of $X$ in $\mathbb{R}^3$. Marginal distributions are generated with truncated normalized Gaussian distributions
\begin{align*}
    N_{[0,1]^2}(x,y;x_0,y_0,\sigma)=\frac{p(x,y;x_0,y_0,\sigma)}{\int_{[0,1]^2}p(x,y;x_0,y_0,\sigma) dxdy}
\end{align*}where $p(x,y;x_0,y_0,\sigma)=e^{-\frac{(x-x_0)^2+(y-y_0)^2}{2\sigma^2}}$.
In both examples, $\bar{\rho}_0 = N(\cdot,\cdot;0.3,0.3,0.1)$, $\bar{\rho}_1 = 1.2 N(\cdot,\cdot;0.7,0.7,0.1)$, $\alpha=\beta = 1$ and $\Delta_x=\Delta_y=\Delta_t = 1/64$. 
In the first example, $\bT(x,y)=(x,y,\exp{(-\frac{(x-0.5)^2+(y-0.5)^2}{2\times 0.15^2})})$; In the second example, $\bT(x,y)=(x,y,\sin(2\pi x)\sin(2\pi y))$. To show the effect of $c_2$, the weight of the secondary space, we choose different $c_2$ and plot $\rho$ and $H$ at different time points (see Figure \ref{fig:monge}). 
In both examples, when $c_2=0.01$, the influence of the secondary space is insignificant and the shape of $\rho$ is relatively unchanged. In the first example, when $c_2=0.05$, $\rho$ splits to avoid climbing the Gaussian bump in the middle. In the second example, when $c_2=0.1$, $\rho$ is elongated so that it does not fall into the two pits in the upper left and lower bottom of the unit square.

\begin{figure}
    \centering
    \includegraphics[width=0.9\linewidth]{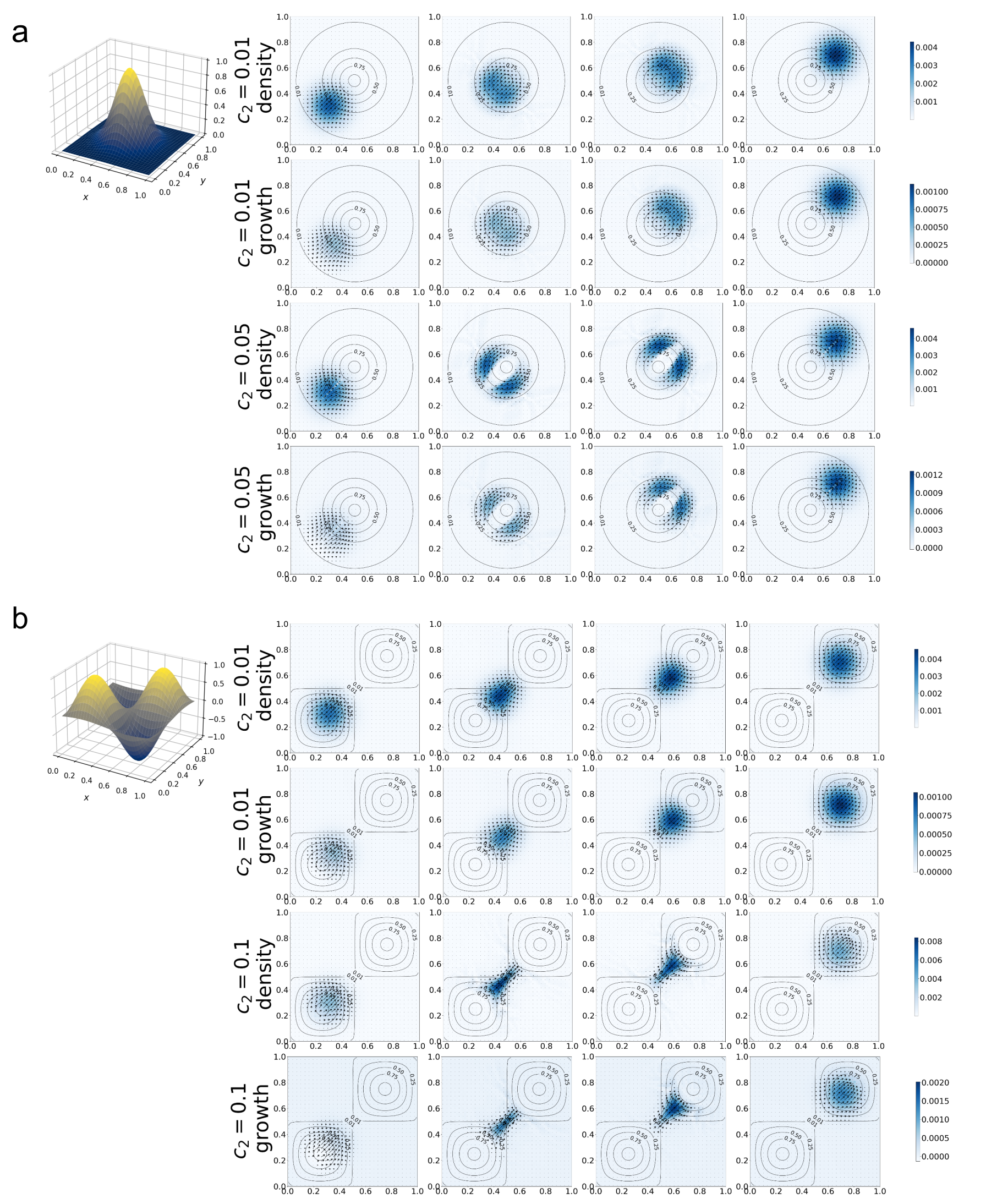}
    \caption{a) The evolving density $\rho$ and growth $H=\rho g$ in the primary space at different time points when $\bT(x,y)=(x,y,\exp{(-\frac{(x-0.5)^2+(y-0.5)^2}{2\times 0.15^2})})$ and $c_2 = 0.01$ or $0.05$. b) The evolving density $\rho$ and growth $H=\rho g$ in the primary space at different time points when $\bT(x,y)=(x,y,\sin(2\pi x)\sin(2\pi y))$ and $c_2=0.01$ or $0.1$. The map $\bT$ is visualized as contours. In both examples, $\alpha=\beta=1$.}
    \label{fig:monge}
\end{figure}

Next, we consider two examples where a map $\bT: X\to Y$ exists and demonstrate the Kantorovich UnSyncOT solver. In the numerical setting, when $\bT: X\to Y$ is given, we simply take the Markov kernel $\pi$ as an identity matrix such that the operation $\TK(\rho)$ amounts to pushing the measure $\rho(\bx)$ forward to the point $\bT(\bx)$ in $Y$, and it generalizes even to the case when $\bT$ is not injective. We revisit the first Monge UnSyncOT example with the Kantorovich solver. As shown in Figure \ref{fig:kant}, the Kantorovich UnSyncOT also produces a split trajectory of densities to avoid the Gaussian bump in the middle, similar to the results of Monge UnSyncOT. In the second example, the space $X=[0,1]^2$ is equipped with certain background color. The color can be represented as a point cloud $Y$ within the three-dimensional RGB color space. We choose $\bar{\rho}_0 = N(\cdot,\cdot;0.3,0.3,0.1)$, $\bar{\rho}_1 = 1.2 N(\cdot,\cdot;0.7,0.7,0.1)$. In this example, UnSyncOT results in trajectories that tend to stay in the darker region to minimize the transport cost in the secondary color space. 

\begin{figure}
    \centering
    \includegraphics[width=0.8\linewidth]{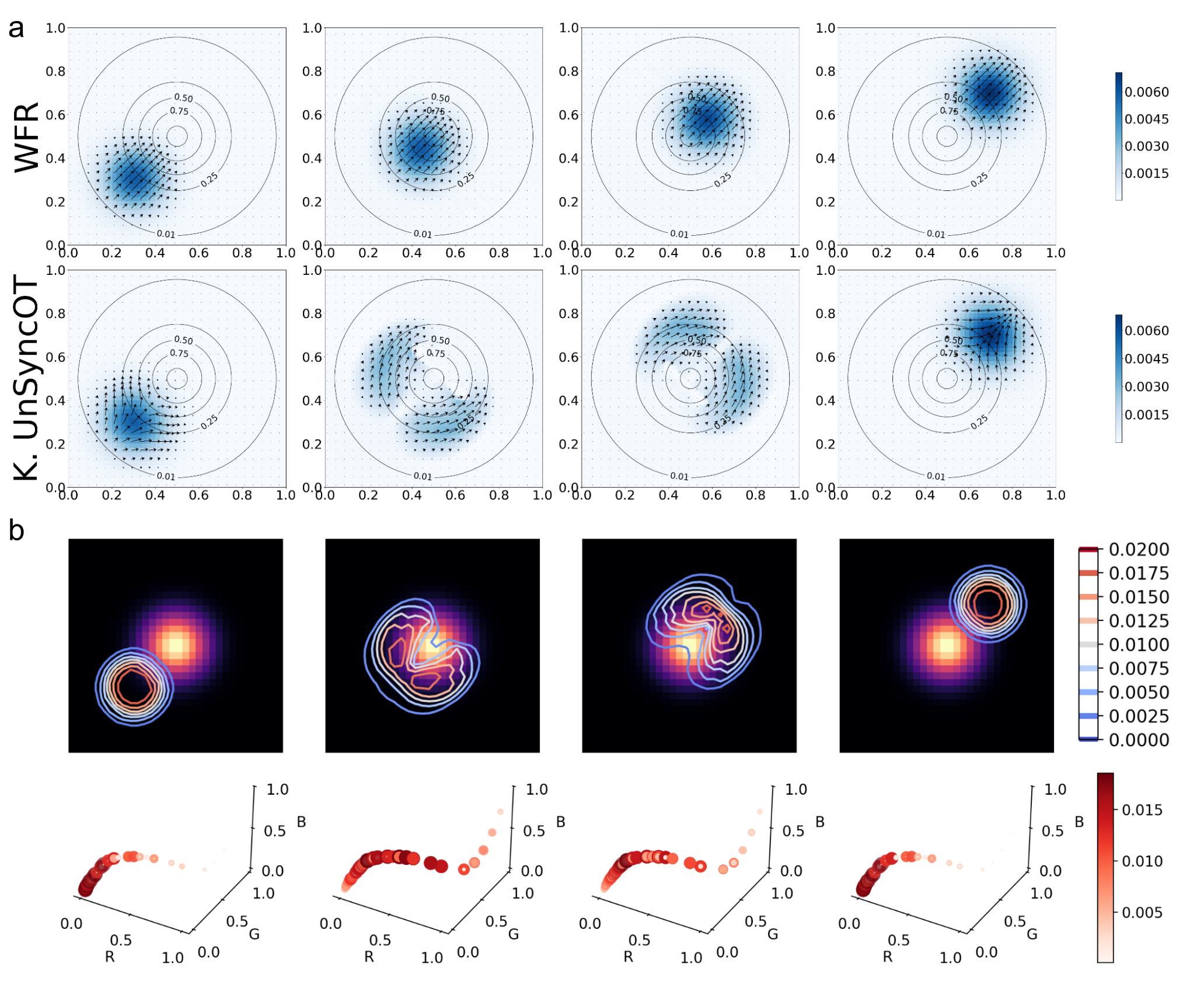}
    \caption{a) The evolving density $\rho$ in the primary space at different time points given by WFR or Kantorovich UnSyncOT for the first Monge UnSyncOT example. Here $Y=\bT(x)$ where $\bT(x,y)=(x,y,\exp{(-\frac{(x-0.5)^2+(y-0.5)^2}{2\times 0.15^2})})$,$ \alpha=\beta=1$, and $c_2/c_1=20$. b) The background color example. $\alpha=\beta=1$ and $c_2=0.5$. Top: the evolving density $\xi = \TK(\rho)$ in the primary space at different time points. The RGB color is shown in the background. Bottom: The evolving density in the secondary color space. Note that $\xi_0$ and $\xi_1$ are supported on the same region in the color space, and Kantorovich UnSyncOT is able to take that fact into account. }
    \label{fig:kant}
\end{figure}

\subsection{biology inspired examples}

In this section, we present two toy examples that are inspired by biological applications. The first one uses the human embryoid body differentiation dataset embedded into 2D with PHATE \cite{moon2019visualizing}. 
When WFR dynamic optimal transport is applied to the PHATE embedding, some trajectories might lie outside the data manifold of the PHATE embedding $X_{\text{data}}$, and therefore are not biologically meaningful.
We can use Monge UnSyncOT and introduce a constraint to enforce trajectories not leave the data manifold $X_{\text{data}}$. The human embryoid body differentiation dataset contains five time points: Day 0-3, Day 6-9, Day 12-15, Day 18-21, and Day 24-27. Since the current solver requires $X=[0,1]^2$, we first transform the PHATE embeddings of the whole dataset to the unit square, and discretize the continuous density of each time point to a 2D probability histogram using a grid where $M=N=64$. We choose Day 6-9 as $\bar{\rho}_0$ and Day 12-15 as $\bar{\rho}_1$ (we multiply the Day 12-15 histogram by 1.2 so that the transport becomes unbalanced) and we hope that most points travel near the data manifold $X_{\text{data}}$, the set of points where $\bar{\rho}_0$ or $\bar{\rho}_1$ is greater than a fixed threshold. 
The map $\bT$ is defined by
$\bT(x, y) = (x,y,10z(x,y)$), where $z(\bx)$ is 
\begin{align*}
    \inf \{\|\bx -\tilde\bx\| | \tilde\bx \in X_{\text{data}} \}.
\end{align*} 
\begin{figure}[t]
    \centering
    \includegraphics[width=0.95\textwidth]{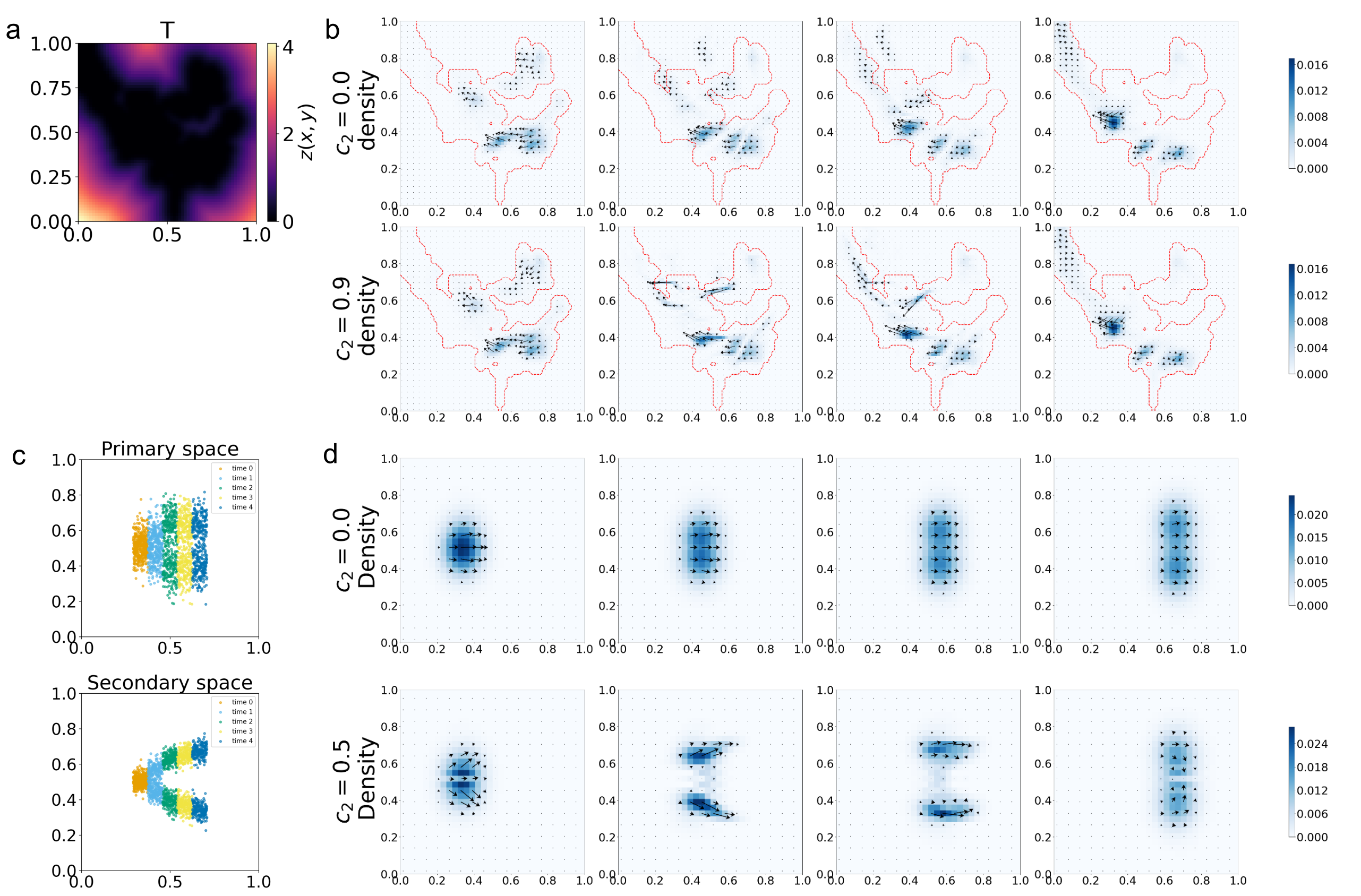}
    \caption{a) The map $\bT(x,y)=(x,y,10z(x,y))$. The data manifold $X_{\text{data}}$ occupies the black region. b) The evolving density $\rho$ in the primary space at different time points when $c_2 = 0.0$ or $0.9$ for the PHATE example. $X_{\text{data}}$ is visualized as the region inside the red contour. When $c_2=0.9$, most part of $\rho_t$ is confined inside $X_{\text{data}}$. c) The normalized simulated scRNA-seq and scATAC-seq data. d) The evolving density $\rho$ in the primary space at different time points when $c_2 = 0.0$ or $0.5$ for the multimodal example. In both examples, $\alpha=\beta=1$. When $c_2=0.5$, a branching apparently occurs.}
    \label{fig:bio}
\end{figure}
When $c_2 = 0$, the initial distribution evolves toward the terminal distribution by approximately straight trajectories. 
As $c_2$ increases, the resulting trajectories closely follow the data-supported region $X_{\text{data}}$. This demonstrates that incorporating secondary-space information imposes a meaningful geometric constraint, leading to biologically plausible trajectories.

The second example concerns the integration of multimodal information in biology.
In cell differentiation, gene expression profiles may remain similar in the early stages, while chromatin accessibility already exhibits lineage-specific signals. Consequently, cells that are indistinguishable in the scRNA-seq space can be clearly separated in the scATAC-seq space.
To simulate this setting, we generate two paired point clouds in $X=Y=[0,1]^2$. Each point is associated with a latent time $t\in[0,1]$, and the timeline is discretized into five consecutive segments. All cells start from a common cluster. After an initial shared phase, the population progressively bifurcates into two trajectories. In the primary space, the branching phenomenon are designed to be weak compared to that of the secondary space. 
We learn a map $\bT$ using a simple neural network, trained with a mean squared error loss on paired observations.
In this experiment, the time 0 distribution in $X$ is set as the initial while time 4 distribution in $X$ is set as the terminal. The terminal mass is set to be 1.5 times larger than the initial mass. 
As shown in Figure \ref{fig:bio}, incorporating information from the secondary space has a noticeable effect on the evolution of the primary space. When $c_2=0.5$, a clear branching behavior emerges. This experiment suggests that UnSyncOT can be used to integrate information from scATAC-seq data to improve interpolation of cell development.

\section{Conclusion and Discussion}

A novel framework is introduced to synchronize the unbalanced dynamical optimal transport across multiple spaces (primary and secondary spaces). The primary and secondary spaces are linked through two options, the pushforward operator (Monge type), and the Markov kernel (Kantorovich type). For both types, we show that the UnSyncOT can be recast into a single-space formulation, and also fit into the framwork of dissipation distances. Two limit cases are considered. When it is pure transport in both spaces, the problem reduces to (balanced) synchronized optimal transport, which is proposed in our early work \cite{cang2025synchronized}. We show that, under certain conditions, Monge SyncOT generates constant-speed geodesics. When it is pure reaction in both spaces, the problem becomes synchronized Fisher-Rao problem. For the Monge Fisher-Rao case, synchronization on the secondary space indeed does not introduce any new dynamics. On the other hand, for the Kantorovich Fisher-Rao case, the map $\TK$ actually contracts the Fisher-Rao distance. Meanwhile, we also design effective primal-dual algorithms to solve the UnSyncOT problems, Chambolle-Pock algorithm for Monge type and the Yan algorithm for Kantorovich type. Numerical experiments show the convergence, stability and efficiency of the proposed algorithms.

This work can be further extended in several directions. Firstly, efficient numerical algorithms for high dimensions are needed to apply UnSyncOT to high-dimensional data such as single-cell multi-omics data. Recently, deep learning-based approaches have been developed to solve high-dimensional dynamical OT problems \cite{sha2024reconstructing,tong2024improving, zhang2025modeling, neklyudov2023computational}, which can potentially be extended to solve high-dimensional UnSyncOT problems. Secondly, in practice, the inter‑space correspondence may be unknown or only partially specified. A natural extension is to jointly learn the synchronization operator, such as a parametric Markov kernel or a neural embedding, together with the UnSyncOT dynamics. Lastly, it is of interest to analyze the convergence rates of the proposed primal-dual schemes under problem-dependent conditions, and possibly refine the error estimates for the HK quadrature beyond $O(\Delta t)$ under stronger regularity conditions.

\section*{Acknowledgements}

This work is partly supported by NSF grant DMS2151934 (ZC), NIH grant R01GM152494 (ZC) and NSF grant DMS2142500 (YZ).


\end{document}